\documentclass[11pt]{article}
\usepackage{amsmath}
\usepackage{amsfonts}
\usepackage{amssymb}
\usepackage{epsfig}
\usepackage{fancybox}
\usepackage{bbm}
\usepackage{mathrsfs}
\usepackage{bm}
\usepackage{dsfont}
\usepackage{epstopdf}
\usepackage{multirow}
 \usepackage{xcolor}
\usepackage{float}
\usepackage{graphicx}
\usepackage{algorithm}
\usepackage{makecell}
\usepackage{boldline}
\usepackage{diagbox}
\usepackage{threeparttable}
\usepackage{adjustbox}
\usepackage{booktabs}
\usepackage{subfig}
\usepackage[noend]{algpseudocode}
\usepackage[title]{appendix}
\usepackage[colorlinks=true]{hyperref}
\hypersetup{urlcolor=black,citecolor=blue,linkcolor=blue}
\usepackage{tikz}
\usetikzlibrary{positioning}
\usetikzlibrary{shapes.geometric, arrows}
\usepackage{xcolor} 

\renewcommand{\Box}{\framebox{\rule{0.3em}{0.0em}}}

\newtheorem{theorem}{Theorem}[section]
\newtheorem{theorem*}{Theorem}[subsubsection]
\newtheorem{lemma}{Lemma}[section]
\newtheorem{proposition}{Proposition}[section]
\newtheorem{example}{Example}[section]

\newtheorem{definition}{Definition}[section]
\newtheorem{assumption}{Assumption}[section]

\newcommand{\setd}{{ d \kern -.15em l}}
\newcommand{\hatsetd}{ d \hat{\kern -.15em l }}
\newcommand{\dd}{\mathsf {d\kern -0.07em l}}
\newcommand{\by}{\boldsymbol}

\newcommand{\bgeqn}{\begin{eqnarray}}
\newcommand{\edeqn}{\end{eqnarray}}
\newcommand{\bgeq}{\begin{eqnarray*}}
\newcommand{\edeq}{\end{eqnarray*}}

\newcommand{\R}{{\rm I\!R}}

\newcommand{\inmat}[1]{\mbox{\rm {#1}}}

\graphicspath{{./Figures/}}

\newcommand{\diag}{{\rm diag}}

\newcommand{\bbe}{{\mathbb{E}}}

\newcommand{\be}{\begin{equation}}
\newcommand{\ee}{\end{equation}}

\tikzstyle{arrow} = [->,>=stealth]

\def\bbe{{\Bbb{E}}} 

\renewcommand{\Box}{\hfill \rule{2.3mm}{2.3mm}}

\numberwithin{equation}{section}

\renewcommand{\Box}{\framebox{\rule{0.3em}{0.0em}}}

\def\bbe{{\Bbb{E}}} 

\newcommand{\mytab}{
  \begin{tabular}{lcr}
    \toprule
    O/D & 2 & 3 \\
    \midrule
    1 & 400 & 800 \\
    4 & 600 & 200 \\
    \bottomrule
  \end{tabular}
}

\title{
A Modified 
Late Arrival Penalised User Equilibrium Model and 
Robustness in Data Perturbation
}
 \author{
 Manlan Li\footnote{
 Department of Systems Engineering and Engineering Management, 
 The Chinese University of Hong Kong. E-mail: mlli@se.cuhk.edu.hk.
} and
 Huifu Xu\footnote{Corresponding author, Department of Systems Engineering and Engineering Management, The Chinese University of Hong Kong. Email: hfxu@se.cuhk.edu.hk.}}

\begin{document}

\maketitle

\begin{abstract}
In a seminal paper \cite{Watling06},
Watling proposes a stochastic variational inequality approach to model traffic flow equilibrium over a network where
the transportation time is random and
a path is selected by 
to transport if the user's expected utility of the transportation of the path is maximized over their paths.
A key feature of Watling's model is that the user's utility function
incorporates a penalty term for lateness
and the resulting equilibrium is known as 
Late Arrival Penalised User Equilibrium (LAPUE).
In this paper, we revisit the LAPUE model with a different focus:
we begin by adopting a new penalty function
which gives a smooth transition of the boundary 
between lateness and no lateness and demonstrate 
the LAPUE model based on the new penalty function 
has a unique equilibrium  and is stable with respect to (w.r.t.)   
small perturbation of probability distribution under moderate conditions. 
We then move on to discuss statistical robustness of the modified 
LAPUE (MLAPUE)  model by considering the
case that the data to be used for fitting the density function may be perturbed in practice or there is a discrepancy between the probability distribution of the underlying uncertainty constructed with empirical data and the true probability distribution in future,
 we investigate how the data perturbation may affect the equilibrium. We undertake the analysis from two perspectives: 
(a) a few data are 
perturbed by outliers and (b) 
all data are potentially 
perturbed.
In case (a), we use the well-known influence function to quantify the sensitivity of the equilibrium by the outliers and in case (b) we examine
the difference between empirical distributions of the equilibrium based on perturbed data and the equilibrium based on unperturbed data.
To examine the performance of the MLAPUE model and our theoretical analysis of statistical robustness, we carry out some numerical experiments, the preliminary results confirm the  statistical robustness as desired.

\end{abstract}

\textbf{Key words.} Smooth penalty function for lateness, MLAPUE, distribution shift, data perturbation, statistical robustness

\section{Introduction}
\label{sec:introd}

%
Traffic assignment involves determining the flow patterns of arcs or paths in a traffic network, taking into account the network topology, origin-destination (OD) demands, and arc performance functions.  Various models have been proposed for this, most of which are based on  Wardrop's user equilibrium framework \cite{Wardrop52}.
Wardrop states two principles in assigning vehicles into transportation network: (a) user equilibrium (UE); (b) system optimum (SO). The UE principle states that in a transportation network, individual travelers selfishly choose their routes to minimize their own travel costs. In other words, no traveler can decrease their travel time by unilaterally changing their route. 
This equilibrium ensures that no traveler has an incentive to switch to a different route, as it would increase their individual travel time. {\color{black} The SO principle, also known as the social optimum, aims to minimize the total travel time or cost for all travelers in the network. Unlike the 
UE, the SO does not consider individual 
behavior,
rather 
it focuses on the overall efficiency of the transportation system. The SO condition is met when all travelers are assigned to routes that collectively minimize the total travel time or cost across the network. This may involve redistributing traffic from congested routes to less congested ones to improve the overall system performance.}
Beckmann et al \cite{BMW56} formulate UE as a mathematical program
with symmetric arc costs
while 
Smith \cite{Smith79} and Dafermos \cite{Dafermos80} 
cast it as a variational inequality (VI) problem with general asymmetric arc costs. 
In a more recent development, 
the UE model has been extended to encompass ride-hailing and ride-sharing services, see e.g. \cite{MXMC20,MMCL22,PZDG22} and references therein.

The classic UE models assume that travelers make their choices based on fixed and known arc travel times, aiming to minimize their traffic times or costs. This assumption implies that travelers have complete information about the network and can accurately predict the travel times on each arc. 
In practice, 
travelers often face uncertainty and imperfect information about the actual travel times. 
To address the issue, Daganzo and Sheffi \cite{DaganzoY77} 
propose a stochastic user equilibrium (SUE) model that takes into account the imperfect perceptions of travel times by travelers. Their model recognizes that travelers may have limited or inaccurate information about the actual travel conditions, and it incorporates these imperfect perceptions into the equilibrium framework. 
Lou et al \cite{LouYL10} develop a mathematical model that incorporates travelers' bounded rationality and their decision-making process in response to congestion pricing strategies. They consider the uncertainties in travelers' route choices and develop a robust optimization approach to determine congestion pricing schemes that are resilient to the variations in travelers' behaviors and information.
As Prakash et al mentioned in literature \cite{PSS18} that the traditional SUE models exhibit several drawbacks that limit their effectiveness. First, they do not adequately capture the risk preferences of users and the value of reliability in route choice, as empirical evidence suggests \cite{CarrionL12}. Second, while these models incorporate errors in users' perception of travel times, they treat actual travel times as deterministic \cite{Watling02}. Third, there is a lack of consistency in the assumptions made regarding the stochastic nature of different components of the traffic system within the equilibrium framework \cite{NakayamaW14}.
To address these limitations, Watling \cite{Watling02} introduces an extension to the standard SUE termed as Generalized SUE. This extension endogenously models the impact of flow variability on travel time variability, effectively addressing the latter two drawbacks mentioned above. Building upon this work, Nakayama and Watling \cite{NakayamaW14} propose a comprehensive framework for formulating network equilibrium problems in stochastic networks. They present four model variants that differ in how stochastic flows are generated. While their framework provides a flexible and detailed characterization of equilibrium problems with stochastic flows, it does not explicitly consider the risk preferences of travelers.

Another stream of previous research acknowledges that traffic conditions inherently exhibit stochastic behavior, primarily due to uncertainties in both demand and supply factors. This line of investigation aims to develop more realistic models that capture how travelers behave in an uncertain traffic environment \cite{XLYZ11}.
Multiple studies have focused on the reliability-based UE problem, which explicitly considers the risk preferences of travelers. In these studies, cost functions are employed that incorporate a measure of travel time reliability, allowing for a more comprehensive analysis of transportation network dynamics.  
A brief summary of each study is provided below, highlighting key model characteristics such as the reliability measure used, treatment of uncertainty, problem formulation, and solution algorithm. 
The proposed models encompass various behavioral assumptions regarding travelers' preference behaviors, including but not limited to models based on expected utility
theory (e.g., \cite{MirchandaniS87}), the travel time budget models (e.g., \cite{LLS06}), the late arrival penalty model (e.g., \cite{Watling06}), the mean excess travel time model (e.g., \cite{ZhouC08}), the models based on cumulative prospect theory (e.g., \cite{ConnorsS09,XLYZ11}) and the mean-risk model \cite{NikolovaS14}. 
Regarding the treatment of uncertainty, existing models can be categorized based on how they handle uncertainty on the supply side \cite{Watling06}, demand side \cite{ChenZ10,YFW18,MMCL22}, or both \cite{SiuL08,ZCS11}. 
For example, 
Lo, Luo, and Siu  
 \cite{LLS06} develop a mathematical framework to capture the travel time budget of individual travelers, considering their risk preferences and the degradability of the transport network. The model allows for a more realistic representation of traveler behavior by accounting for the trade-off between travel time and the perceived risk associated with congestion.

Watling \cite{Watling06} proposes the late arrival penalty model and extends the UE to Late Arrival Penalized User Equilibrium (LAPUE). It incorporates the concept of late arrival penalties into the equilibrium condition to capture the preference of travelers for minimizing both travel time and the risk of arriving late to their destinations. 
LAPUE is a variation of the UE principle in transportation network modeling. 
In traditional UE travelers aim to minimize their travel time without explicitly considering the potential penalties associated with late arrival. However, in real-world scenarios, individuals often have time constraints or face penalties for arriving late, such as missing appointments, incurring financial costs, or experiencing inconvenience.
The LAPUE principle addresses this by introducing penalties for late arrival. It assumes that travelers have a certain value or cost associated with arriving late, and they consider this penalty when making route choices. The objective for travelers in LAPUE is to minimize the total cost, which includes both travel time and the penalties for late arrival.
Chen et al \cite{ChenZ10} introduce the $\alpha$-reliable mean-excess traffic equilibrium model, which considers the stochastic nature of travel times and incorporates reliability considerations into the equilibrium framework.
Nikolova and Stier-Moses \cite{NikolovaS14} introduce a mean-risk traffic assignment model, which combines the mean travel time and a risk-aversion factor multiplied by the standard deviation of travel time along a path. This objective function allows users to incorporate their risk preferences into their route selection process. 
Ma et al \cite{MMCL22} explore and address the challenges of modeling and solving the user equilibrium problem in stochastic ridesharing systems with elastic demand. 
For more descriptions of relevant aspects, we refer readers to the literature review section of \cite{PSS18} and the review \cite{ZXQCC22}.

In this paper, we revisit the LAPUE model but with some new focuses.
First, we propose a
parameterized penalty function for lateness which allows a smooth transition of the boundary
between lateness and no lateness.
Second, we demonstrate under some moderate conditions that the  LAPUE  based on the new penalty function approximates the LAPUE when the parameter is driven to $0$. Third, we consider data perturbation  
in the modified LAPUE (MLAPUE) model.
There are at least three types of perturbations which we believe may occur in practice: 
(a) the perceived data  used for the construction of mathematical model are perturbed for various reasons such as measurement/recording errors; 
(b) there is a nominal distribution describing the probability distribution of the random parameters in a transportation network but the 
actual probability distribution of these uncertainty parameters might deviate from the nominal, e.g., due to unexpected closure of a road or prolonged duration of maintenance work; (c) the validation data (usually for future) may deviate from 
the current or past data that the model is developed. 
Under these circumstances, there is a need to 
investigate how the change of the data may affect the MLAPUE. 
We show under some moderate conditions that the MLAPUE based on the new penalty function 
is resilient to the exogenous data perturbation by exploiting classical techniques in robust statistics \cite{Hampel68,Huber04} and recent developments on qualitative and quantitative 
statistical robustness analysis in risk management and operations research \cite{cont10,KSZ12,KSZ14,GuoXu21a,LiuPang23}.

The rest of the paper is organized as follows. 
In Section \ref{sec2:MLPAUE}, we recall the definition of the LAPUE model, introduce a MLAPUE model with a new penalty function for lateness and discuss existence and uniqueness of MLAPUE and its relation to LAPUE. 
We also propose sample average approximation of the MLAPUE model and discuss convergence of MLAPUE obtained with sample data converges to its true counterpart as the sample size goes to infinity.
In Section \ref{sec3:SR-MLAPUE}, we discuss statistical robustness of MLAPUE when the sample data are perturbed. 
In Section \ref{se4：NT}, we report numerical test results
of the statistical robustness of sample average approximated MLPAUE. Finally we conclude with some remarks in Section \ref{sec5:Con}. All proofs of technical results are delegated to the appendix.

Throughout the paper, we use the following notation.
\begin{itemize}
    \item $G$: directed graph representing the transportation network;
    \item $\mathcal{N}$: the set of nodes on $G$ and $|\mathcal{N}|=n$ where $|\mathcal{N}|$ denotes cardinality of set ${\cal N}$;
    \item $\mathcal{A}$: the set of arcs on $G$ and $|\mathcal{A}|=A$;
    \item $\mathcal{W}$: the set of all original-destination pairs on $G$ and $|\mathcal{W}|=W$;
    \item $R_k$: the set of paths connecting the OD pair $k=1,\cdots,W$;
    \item $\mathcal{R}$: collection of paths between all OD pairs $\mathcal{R}=\cup_{k=1}^WR_K$ and $|\mathcal{R}|=N$;
    \item $q_k$: demand between OD pair $k$;
    \item $f_r$: flow on path $r$;
    \item $\by{f}\in\R^N$: vector of path flows with component
    variable $f_r$ for $r\in \mathcal{R}$ ;
    \item $v_a$: flow on arc $a$;
    \item $\by{v}\in\R^A$: vector of arc-flows with component
    variable $v_a$ for $a\in \mathcal{A}$;
    \item $\delta_{ar}$: arc-path indicator variable;
    \item $\Delta$: arc-path incidence matrix;
    \item $\Pi$: OD-path incidence matrix;
    
    \item $\xi\in\R^k$: the vector of uncertain factors;
    
    \item $t^Q_a(\by{v})$: mean travel time on arc $a\in A$ under the distribution $Q$ of random vector $\xi$ when the overall arc-flow distribution is $\by{v}$;
    \item $\by{t}^Q(\by{v})\in\R^A$: vector of mean arc travel time under the distribution $Q$ when the overall arc-flow distribution is $\by{v}$;
    \item $\by{T}(\by{v},\xi)\in\R^A$: the 
    vector of actual arc travel time which 
    depends on the arc flow vector $\by{v}$ and random vector $\xi$;
    \item $C_r(\by{f},\xi)$: random travel time along path $r\in\mathcal{R}$ when the overall path-flow distribution is $\by{f}$;
    \item $\by{C}(\by{f},\xi)$: 
    vector of path travel times 
    when the overall path-flow distribution is $\by{f}$;
    \item $u_r(\by{f},\xi)$: 
    disutility of path $r\in\mathcal{R}$ when the overall path-flow distribution is $\by{f}$;
    \item $\by{u}(\by{f},\xi)$: 
    vector of path-disutility when the overall path-flow distribution is $\by{f}$;
    \item $\R^N$ and $\R^N_+$: $N$-dimensional Euclidean space and its nonnegetive subspace.
    \item $\|\cdot\|$ denotes the Euclidean norm in a finite dimensional space unless specified otherwise.
\end{itemize}

\section{ Watling's LAPUE model and its variation}
\label{sec2:MLPAUE}

The transportation network of interest is represented by a directed graph $G(\mathcal{N},\mathcal{A})$, where $\mathcal{N}$ denotes the set of $n$ nodes and $\mathcal{A}$ represents the set of $A$ directed arcs. The set of all origin-destination (OD) pairs
on the network ${G}$ is denoted by $\mathcal{W}$, and the set of path connecting the OD pair $k$ is given by $R_k$ for $k=1,\cdots,W$. The entire collection of paths between all the OD pairs is represented by $\mathcal{R}=\cup_{k=1}^W R_k$. The total number of paths is given by $N=|\mathcal{R}|$. Let $\by{q}\in\R^W$ be a vector of demands with components $q_k$ representing the deterministic demand on the  $k$th OD. 

For a deterministic demand vector $\by{q}$, an assignment of flows to all paths is denoted by a vector $\by{f}\in\R^N$, where each component $f_r$ represents the flow on path $r$. In order to be feasible for meeting the demand, we must have $\by{f}\in D$ where 
\begin{equation}\label{FS-path}
    D=\left\{\by{f}\in\R_+^N:\Pi\by{f}=\by{q} \right\},
\end{equation}
and $\Pi$ denotes the OD-pair incidence matrix with entries $\pi_{k,r}=1$ if path $r$ connects the $k$th OD, and $\pi_{k,r}=0$ otherwise.
While an assignment of flows to all arcs is 
denoted
by $\by{v}\in\R^A$ with components $v_a$ representing 
 the flow on arc $a$.
The path flows are related to the arc flows by
\begin{equation}\label{relationship-flow-acr-path}
    \by{v}=\Delta\by{f},
\end{equation}
where $\Delta$ is the arc-path incidence matrix with components $\delta_{ar}=1$ if arc $a$ is on path $r$, and $\delta_{ar}=0$ otherwise. Based on relation in \eqref{relationship-flow-acr-path}, 
we denote the set of demand-feasible arc flow vectors 
similarly 
by 
 \begin{equation}\label{FS-arc}
     \tilde{D}=\left\{\by{v}\in\R^A:\by{v}=\Delta\by{f} \ {\rm and} \ \by{f}\in D\right\}.
 \end{equation}





 \subsection{LAPUE model}
 \label{subsec2.1:LAPUE}
Watling \cite{Watling06} 
proposes 
a generalised 
UE model 
known as {\em late arrival penalised user equilibrium (LAPUE)}
which reflects (a)
 a driver's valuation of path's expected attributes, such as distance, expected travel time, tolls, and more; and
 (b) the extent to which following the path is likely, in the light of travel time variability, to satisfy a traveller on that OD pair in achieving an ``acceptable'' arrival time at the destination.
 
Let $\by{T}(\by{v})$ denote 
the vector of arc travel times, 
where component $T_a(\by{v})$ 
represents
the actual travel time on arc $a$ for $a=1,\cdots,A$.
Since arc travel time is often random,
the path travel time is  also a random vector.
We denote it  by $\by{C}(\by{f})$, where the 
$r$th component $C_r$ represents 
 the travel time  on path $r$, which is related to the random arc travel time vector by the transformation  
\begin{equation}\label{relationship-CT}
\by{C}(\by{f})=\Delta^\top\by{T}(\by{v}). 
\end{equation}
For each OD pair $k\in\{1,\cdots,W\}$, let $\tau_k$ be a longest acceptable travel time. 
Watling \cite{Watling06} 
uses a composite path disutility to incorporate both the standard ``generalized travel time'' and the travel time acceptability in the form of a lateness penalty:
 \begin{equation}
 \label{defi:pathdisutility}
 u_r = \theta_0 d_r + \theta_1 C_r + \theta_2\max(C_r-\tau_k,0), \; \inmat{for}\;  r=1,\cdots,N, \ k=1,\cdots,W,
 \end{equation}
 where 
 $d_r$ 
 denotes the composite of attributes (such as distance) that are independent on time/flow and $\theta_0$ is the value placed on these attributes, $\theta_1$ is the value-of-time, and $\theta_2$ reflects the value of being one time unit later than acceptable. 
 This is indeed a key feature of the LAPUE model.
 By introducing $\psi_r$ 
 to represent
 the marginal density function of 
 $C_r$,
 we can reformulate
 the expected disutility of path $r$
 as 
 \bgeqn\label{defi:pathdisutility-random}
 \bbe[u_r]& = &\theta_1 \bbe[C_r] + \theta_2\bbe[\max(C_r-\tau_k,0)] \nonumber\\
&=&  \theta_0 d_r + \theta_1 \mathbb{E}[C_r] + \theta_2\int_{\tau_k}^\infty (c-\tau_k)\psi_r(c)dc, \ \inmat{for} \ r=1,\cdots,N, \ k=1,\cdots,W,
 \edeqn
where $\bbe[\cdot]$ is the expectation taken w.r.t. the probability distribution of $C_r$.

In the forthcoming discussions, 
we 
will investigate UE in terms of path flow and arc flow. To this end, we need to represent the user's disutility 
in terms of 
$\by{f}$ and $\by{v}$.
We begin by rewriting the path travel time $\by{C}$ and the arc travel time $\by{T}$ 
in \eqref{relationship-flow-acr-path} and \eqref{relationship-CT} respectively as 
$\by{C}(\by{f},\xi)$ and $\by{T}(\by{v},\xi)$ to indicate 
explicitly 
that both quantities 
depend on $\by{f}$, $\by{v}$ and
random vector $\xi$, 
where the latter is known as arc performance function in transportation network.
The 
random vector $\xi$ 
is used to describe 
various uncertain factors 
affecting the 
travel time on a path (or arc) and we assume that $\xi$ takes values on a compact set $\Xi\subseteq\R^k$. Note that in the literature 
of transportation research, the sources of uncertainty behind travel time are not explicitly indicated, rather the capacity of arc are made random when the demand is deterministic. Here we adopt a new form of representation to separate the the decision vector $\by{f}$ ($\by{v}$) from random vector $\xi$.  

Recall that $\by{C}(\by{f},\xi)=\Delta^\top\by{T}(\by{v},\xi)$ and $\by{v}=\Delta\by{f}$, where $\Delta$ is the arc-path incidence matrix with entries $\delta_{a,r}=1$ if arc $a$ is on the path $r$, and $\delta_{a,r}=0$ otherwise.
In some 
literature of transportation research 
with uncertain supply,
the components of the arc travel time 
$\by{T}(\by{v},\xi)$ 
are described by a so-called {\em generalized Bureau of Public Roads (GBPR) function}: 
\bgeqn
\label{func:GBPR}
T_a(\by{v},\xi)=t_a^0\left(1+b_a\left(\frac{v_a}{C_a(\xi)}\right)^{n_a}\right),
\edeqn
where $t_a^0,b_a$ and $n_a$ are given parameters, $t_a^0$ denotes the arc free-flow travel time of arc $a$,
$C_a(\xi)$ denotes the random capacity of arc $a$ and $v_a$ represents the flow of arc $a$, see e.g.~\cite{YML09,ZCS11,ZZZY19}. Let $\by{\tau}=(\by{\tau}^\top_1,\cdots,\by{\tau}^\top_W)^\top\in\R^N$ be a block vector, where the $k$th block $\by{\tau}_k$ is  a column vector with dimensions $|R_k|$ and each  component takes the same value $\tau_k$.
Under the settings 
outlined above,
 we 
 can write the disutility function 
 of path as
 \bgeqn
 \by{u}(\by{f},\by{C}(\by{f},\xi))=\theta_0\by{d}+\theta_1\by{C}(\by{f},\xi)+\theta_2\max\{\by{C}(\by{f},\xi)-\by{\tau},0\},
 \edeqn
and the LAPUE model in \cite{Watling06} 
as the following stochastic variational inequality problem: 
\bgeqn
\inmat{(SVIP-path)}\quad \label{eq:SVI-LAPUE}
0\in \mathbb{E}_{P}[{\bm u}({\bm f},\by{C}(\by{f},\xi))]+{\cal N}_{D}({\bm f}),
\edeqn
where $\by{u}:\R^N\times\R^k\to\R^N$ is a continuous vector-valued function, 
$$
\mathcal{N}_D({\bm f})=\{{\bm w}\in \R^N: {\bm w}^T({\bm g}-{\bm f})\leq 0, \forall {\bm g}\in D\}
$$
is the normal cone to the convex set $D$ at ${\bm f}$, $\xi:\Omega\to\R^k$ is a random vector defined on a probability space $(\Omega,\mathcal{F},\mathbb{P})$ with probability distribution $P\in\mathscr{P}(\Xi)$, which can be reflected on the uncertain factors, $\mathbb{E}_{P}[\cdot]$
is the expected value w.r.t. $P$ and $\mathscr{P}(\Xi)$ denotes the set of all probability measures over $\Xi$. 
We call a solution to SVIP-path (\ref{eq:SVI-LAPUE}) 
an LAPUE and use 
$\by{\mathcal{F}}(P)$ 
to denote 
the set of all such solutions.
Note that in the case when $\theta_2=0$,
the LAPUE model collapses to UE model.
Moreover, by the connection between $\by{C}(\by{f},\xi)$ and $\by{T}(\by{v},\xi)$ as well as path flow $\by{f}$ and arc flow $\by{v}$, the SVP-path \eqref{eq:SVI-LAPUE} can be equivalently presented as follows:
\bgeqn
\inmat{(SVIP-arc)}\quad \label{eq:SVI-LAPUE-arc}
0\in \mathbb{E}_{P}[\tilde{\bm u}({\bm v},\by{T}(\by{v},\xi))]+{\cal N}_{\tilde{D}}({\bm v}),
\edeqn
where 
\bgeqn
 \tilde{\by{u}}(\by{v},\by{T}(\by{v},\xi))=\theta_0\by{d}+\theta_1\Delta^\top\by{T}(\by{v},\xi)+\theta_2\max\{\Delta^\top\by{T}(\by{v},\xi)-\by{\tau},0\},
 \edeqn
  and $\tilde{D}$ is defined as in \eqref{FS-arc}. 
  We also call a solution to SVIP-arc (\ref{eq:SVI-LAPUE-arc}) an LAPUE. In the rest of the paper, we use terminologies ``LAPUE'' and  ``solution to SVIP-arc''
 interchangeably.  
  The difference between the two forms of LAPUE is that the former is presented in terms of path flow whereas the latter is in arc flow.
The next theorem summarizes existence and uniqueness of 
LAPUE from 
Watling \cite{Watling06}
in terms of the solutions to 
SVIP-path \eqref{eq:SVI-LAPUE} and SVIP-arc (\ref{eq:SVI-LAPUE-arc}).
\begin{theorem}[Existence and uniqueness of LAPUE \cite{Watling06}]
\label{thm:exist-solu}
Consider the marginal path travel time density function for path $r$ and write it as $\psi_r(c_r;\mu_r)$ to denote its (partial) parameterisation by the mean path travel time $\mu_r$. 
Assume: 
(a) the functions 
$$
F_r(\mu_r)=\int_{\tau_k}^\infty(c-\tau_k)\psi_r(c;\mu_r)dc, \; \inmat{for}\;  r\in R_k, k=1,\cdots,W
$$
are well-defined, continuous and non-decreasing; (b)  $\theta_1>0$ and $\theta_2\geq 0$  in \eqref{defi:pathdisutility-random}; (c) the mean arc travel time functions $t^P_a(\by{v})$ are continuous and strictly monotone mapping. Then both SVIP-path \eqref{eq:SVI-LAPUE} and SVIP-arc (\ref{eq:SVI-LAPUE-arc}) have a solution,  and  SVIP-arc (\ref{eq:SVI-LAPUE-arc}) has a unique solution.
\end{theorem}

Non-uniqueness of solution to SVIP-path \eqref{eq:SVI-LAPUE} is due to the fact that the matrix $\Delta$ is not necessarily  of full column rank. For the purposes of theoretical analysis,
we make the following assumption for SVIP-path \eqref{eq:SVI-LAPUE}.



\subsection{A new penalty function for lateness and the MLAPUE model}
\label{subsec2.2:MLAPUE}

In this paper, our focus is not on Watling's original LAPUE model, rather it is on its modification and statistical robustness of the resulting  UE against data perturbation. 
Specifically we 
propose to adopt a different penalty function for lateness as follows: 
\begin{equation}
\label{eq:Alexander-smoothing}
    h(z,t)=\left\{
    \begin{array}{lcl}
    z, & & \inmat{for}\; t<z,\\
    \frac{1}{4t}(z^2+2zt+t^2), & & \inmat{for}\; -t\leq z\leq t,\\
    0, & & \inmat{for}\; z<-t,
    \end{array}\right.
\end{equation}
where $t>0$ is a small positive number.
\begin{figure}[H]
  \centering
  \subfloat[]{
     \label{fig:penalpoint5}
     \includegraphics[scale=0.4]{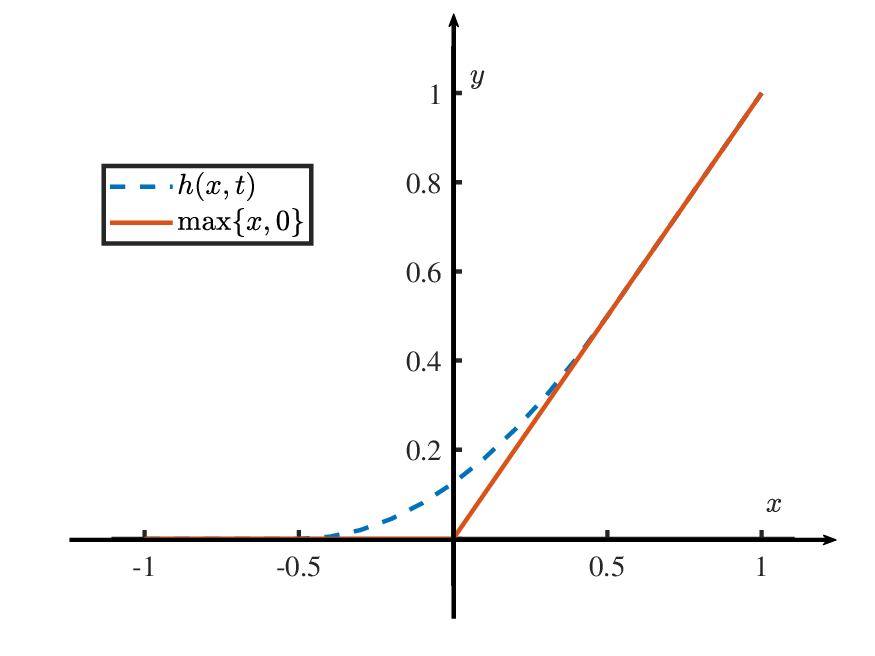}}
  \subfloat[]{
    \label{fig:penal1}
    \includegraphics[scale=0.4]{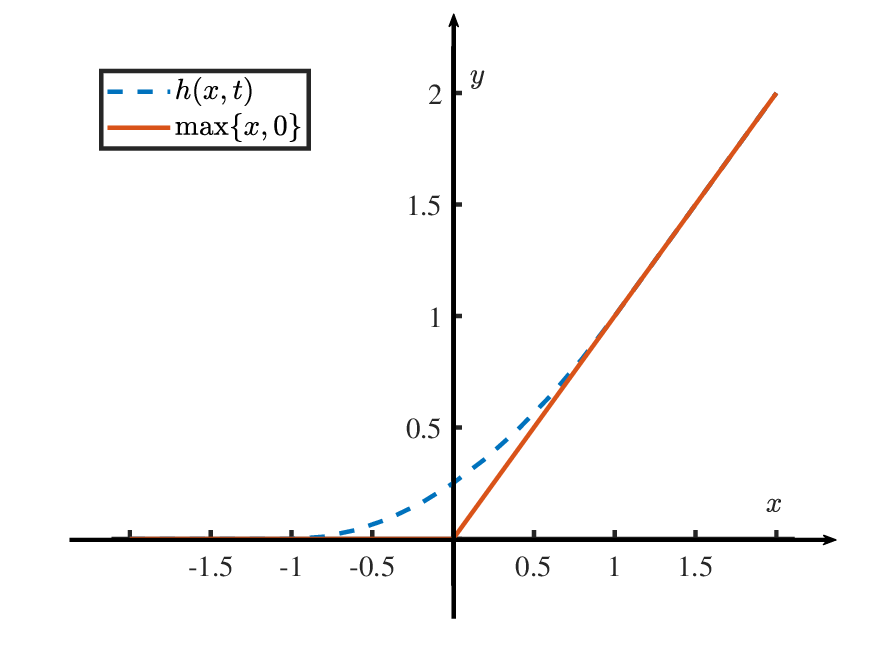}}
  \caption{(Color online) (a)-(b) Parameterized penalty function with parameter $t=0.5$ and $t=1$.
  }
  \label{fig:penal}
\end{figure}
Compared to the max penalty function $\max\{z,0\}$,
 $h(z,t)$ is parameterized by a positive number $t$,
 which means that when the lateness exceeds $-t$,
  there is a penalty and this penalty increases gradually as $z$ increases. 
  The new penalty function  gives 
a smooth transition of the boundary between lateness and no lateness.
In practice, we can interpret 
  this as driver/traveller becoming increasingly 
  worried when the travel time approaches the longest acceptable travel time $\tau_k$. When $z$ exceeds $t$, the penalty is back up to the Watling's maximum penalty. The new penalty function gives us some  flexibility to model a driver/traveller's risk preference on lateness by setting a different value of $t$.
    Mathematically, the new penalty function enjoys nicer property because it is continuously differentiable at $z=0$ and it approximates the maximum penalty function  when $t\to 0$. 
    The continuous differentiability allows use to undertake robustness/resilient 
    of the resulting LAPUE againt data perturbation, and this is indeed another major reason why we adopt the new penalty. Note that it is possible to adopt other parametric penalty functions 
    with similar properties, here we 
    concentrate on (\ref{eq:Alexander-smoothing}) as it is
    the simplest.

Let
\begin{equation}
\hat{\by{u}}(\by{f},\by{C}(\by{f},\xi),t)
:=\by{g}(\by{f},\by{C}(\by{f},\xi))+\by{h}(\by{f},\by{C}(\by{f},\xi),t),
\end{equation}
where
$\by{g}(\by{f},\by{C}
(\by{f},\xi))=\theta_0\by{d}+\theta_1\by{C}(\by{f},\xi)$,
$
{\by{h}}(\by{f},\by{C}(\by{f},\xi),t) :=\theta_2{h}(\by{C}(\by{f},\xi)-\by{\tau},t)
$
 and the penalty function $h$ acts componentwise. 
The modified LAPUE model can be written as the following modified SVIP
\bgeqn
\label{eq:SVI-LAPUE-smooting}
\inmat{(MSVIP)} \quad
0\in \mathbb{E}_{P}[\hat{{\bm u}}({\bm f},\by{C}(\by{f},\xi),t)]+{\cal N}_{D}({\bm f}).
\edeqn
We call a solution to the MSVIP 
{\em a modified LAPUE (MLAPUE)}
and use $\by{\mathcal{F}}_t(P)$ to denote 
the set of all such solutions.
We indicate in the notation explicitly the dependence on $t$ and $P$  because we will consider variation of $t$ later in this section and 
the case when $P$ is perturbed from Section 3.
The next proposition states 
existence and uniqueness of
an MLAPUE. 
\begin{proposition}[Existence and uniqueness of MLAPUE]
\label{coy:exist-solu-MSVIP}
Assume: (a) the conditions of Theorem \ref{thm:exist-solu} (b-c) hold, 
(b) the functions 
$$
F_r(\mu_r,t)=\int_{\tau_k-t}^{\infty}h(c-\tau_k,t)\psi_r(c;\mu_r)dc, \; \inmat{for}\;  r\in R_k, k=1,\cdots,W
$$
are well-defined, continuous and non-decreasing for fixed $t>0$, where $h(\cdot,t)$ is defined as in \eqref{eq:Alexander-smoothing} $\psi_r(c_r;\mu_r)$ is defined as in Theorem \ref{thm:exist-solu}. Then
$\inmat{MSVIP}$ has a solution, and the 
associated MLAPUE
 in terms of arc flow is unique.
\end{proposition}

We skip the proof because 
the only difference between Proposition \ref{coy:exist-solu-MSVIP} and Theorem \ref{thm:exist-solu} 
is that $F_r(\mu_r)$ is replaced by $F_r(\mu_r,t)$.
The replacement does not affect the key properties of the disutility function $u$ in the proof of Theorem  \ref{thm:exist-solu}, which means the proof of Theorem \ref{thm:exist-solu} is applicable to that of Proposition \ref{coy:exist-solu-MSVIP}.

Next, we investigate the relationship between 
MLAPUE and LAPUE. We do so by 
showing $\mathbb{D}(\by{\mathcal{F}}_t(P),
\by{\mathcal{F}}(P))\to 0
$ as $t\to 0$, where $\mathbb{D}(A,B)$ denotes the deviation distance 
$\mathbb{D}(A,B)=\sup_{x\in A}d(x,B)$, in other words,
we can interpret MLAPUE as an approximation of LAPUE as $t$ is sufficiently small.
To this end, we need to make some 
technical 
assumptions.
\begin{assumption}
\label{asm:compact-MSVIP}
    For fixed $t>0$, there exists a convex compact set $\mathbb{F}_t\subset\R^N$ which contains $\by{\mathcal{F}}_t(P)+b\mathcal{B}$ for some positive number $b$, where $\mathcal{B}$
    is a unit ball in $\R^N$.
    Let $\mathbb{V}_t:=\{\by{v}=\Delta\by{f}|\by{f}\in\mathbb{F}_t\}\subseteq\R^A$.
\end{assumption}

The assumption virtually requires
$\by{\mathcal{F}}_t(P)$ to be bounded. We introduce set $\mathbb{F}_t$ containing an open neighborhood of $\by{\mathcal{F}}_t(P)$
purely for the convenience of  theoretical analysis in the forthcoming discussions in Section \ref{sec3:SR-MLAPUE}.
The next assumption requires the travel time function to be continuous differentiable
and Lipschitz continuous with respect to 
arc flow, and integrablly boundedness both of which are needed in the forthcoming theoretical analysis in this section and the next section.


 \begin{assumption}
\label{asm:T-value}
    For the probability distribution $P\in\mathscr{P}(\Xi)$,
    \begin{itemize}
       \item [{\rm (a)}] $\by{T}(\cdot,\xi)$ is continuously differentiable for every $\xi\in\Xi$;
        \item [{\rm (b)}] there exists some measurable function $\Phi_1(\xi)$ such that $\sup_{\by{v}\in\mathbb{V}_t}\|\by{T}(\by{v},\xi)\|\leq\Phi_1(\xi)$ and $\mathbb{E}_P[\Phi_1(\xi)]<\infty$;
        \item [{\rm (c)}] there exists a positive measurable function $\Phi_2(\xi)$ such that $\mathbb{E}_P[\Phi_2(\xi)]<\infty$ and 
        $$
        \|\by{T}(\by{v},\xi)-\by{T}(\by{v}',\xi)\|\leq\Phi_2(\xi)\|\by{v}-\by{v}'\|, 
        $$
        for all $\by{v},\by{v}'\in\mathbb{V}_t$, and $\xi\in\Xi$, where $\mathbb{V}_t$ is defined as in Assumption \ref{asm:compact-MSVIP}.
    \end{itemize}
\end{assumption}

Assumption \ref{asm:T-value} is satisfied when each component of $\by{T}(\by{v},\xi)$ is BPR function defined as in \eqref{func:GBPR}. We will come back to this in Section \ref{sec:alldata}.
Let 
\bgeqn
\label{set:well-define}
\hat{\mathscr{P}}:=\left\{Q\in\mathscr{P}(\Xi):\mathbb{E}_Q[\Phi_1(\xi)]<\infty,\mathbb{E}_Q[\Phi_2(\xi)]<\infty\right\}.
\edeqn
The next proposition states some important properties of user's disutility function 
under the new penalty function (\ref{eq:Alexander-smoothing}).

\begin{proposition}[Properties of the new disutility function]
\label{prop:property-of-u}
Under Assumption \ref{asm:T-value}, the following assertions hold
for the parametric disutility function $\hat{\by{u}}(\by{f},\by{C}(\by{f},\xi),t)$. 
\begin{itemize}
        \item [{\rm (i)}] $\hat{\by{u}}(\cdot,\by{C}(\cdot,\xi),t)$ is continuously differentiable for every $\xi\in\Xi$ and $t>0$;
         \item [{\rm (ii)}]
        There exists some measurable function $\tilde{\Phi}_1(\xi,t)$ such that $\sup_{\by{f}\in\mathbb{F}_t}\|\hat{\by{u}}(\by{f},\by{C}(\by{f},\xi),t)\|\leq \tilde{\Phi}_1(\xi,t)$ 
        and $\mathbb{E}_P[\tilde{\Phi}_1(\xi,t)]<\infty$ for any fixed $t>0$.

        \item [{\rm (iii)}] There exists a positive measurable function $\Phi_2(\xi)$ such that $\mathbb{E}_P[\Phi_2(\xi)]<\infty$ and for every $\xi\in\Xi$ and $t>0$, 
        $$
        \|\hat{\by{u}}(\by{f},\by{C}(\by{f},\xi),t)-\hat{\by{u}}(\by{f}',\by{C}(\by{f}',\xi),t)\|\leq \eta_{\Delta^\top}\eta_{\Delta}(\theta_1+\theta_2)\Phi_2(\xi)\|\by{f}-\by{f}'\|, \ \forall\by{f},\by{f}'\in\mathbb{F}_t,
        $$
    \end{itemize}
where $\mathbb{F}_t$ is defined as in Assumption \ref{asm:compact-MSVIP} and $\tilde{\Phi}_1(\xi,t)=\eta_{\Delta^\top}(\theta_1+\theta_2)\Phi_1(\xi)+\theta_0\|\by{d}\|+\theta_2\|\by{\tau}\|+t$.
\end{proposition}

Proposition \ref{prop:property-of-u} (ii) implies that $\mathbb{E}_P[\hat{\by{u}}(\by{f},\by{C}(\by{f},\xi),t)]$ is well-defined for each 
fixed $ t\geq0$, $\by{f}\in\mathbb{F}$  and $P\in\hat{\mathscr{P}}$, where $\hat{\mathscr{P}}$ is defined as \eqref{set:well-define}. Together with (i) and (iii), $\mathbb{E}_P[\hat{\by{u}}(\by{f},\by{C}(\by{f},\xi),t)]$ is continuously differentiable in $\by{f}$, this is, $\nabla_{\by{f}}\mathbb{E}_P[\hat{\by{u}}(\by{f},\by{C}(\by{f},\xi),t)]=\mathbb{E}_P[\nabla_{\by{f}}\hat{\by{u}}(\by{f},\by{C}(\by{f},\xi),t)]$ for each $\by{f}\in\mathbb{F}$, fixed $t>0$ and $P\in\hat{\mathscr{P}}$, where $\nabla_{\by{f}}\hat{\by{u}}(\by{f},\by{C}(\by{f},\xi),t)$ denotes the Jacobian of $\hat{\by{u}}$ in $\by{f}$.

Next, we need to introduce a notion of so-called strong regularity of a generalized equation.
This is because we will use generalized equation as a mathematical framework for undertaking the desired theoretical analysis of MLAPUE defined by
MSVIP, a special generalized equation.

\begin{definition}
[Strong regularity]
Let $Z$ and $W$ be Banach spaces, $\phi:Z\to W$ be a continuously differentiable mapping,
${\cal N}:Z \rightrightarrows W$ be a set-valued mapping.
Consider the generalized equation:
find ${\bm z}\in Z$ such that 
\bgeqn
\label{eq:abstract_GE-0}
0\in \phi(z)+ {\cal N}(z).
\edeqn
A solution $z_0$ is 
said to be strongly regular
if there exist neighborhoods ${\cal V}_Z$ and ${\cal V}_W$ of $z_0\in Z$ and $0\in W$, respectively,
such that for every $\delta\in {\cal V}_W$,
the linearized abstract generalized equation 
$$
\delta\in \phi(z_0)+D\phi(z_0)(z-z_0)+{\cal N}(z)
$$
has a unique solution in ${\cal V}_Z$, 
denoted by $\zeta(\delta)$,
and the mapping $\zeta: {\cal V}_W \to {\cal V}_{Z}$ is Lipschitz continuous with constant $c$.
\end{definition}

In the case that  ${\cal N}(z)=0$, the strong regularity conditions is equivalent to 
the case that $D\phi(z_0)$ is  one-to-one and onto mapping.
With the proposition and the notion of 
strong regularity, we are now ready to state our main result in this section.

 {\color{black}
 \begin{theorem}[Approximation of MLAPUE to LAPUE]
 \label{thm:smth-approx}
 Under 
 the conditions of Proposition \ref{coy:exist-solu-MSVIP} and Assumption~\ref{asm:T-value}, 
the following assertions hold.
 \begin{itemize}
     \item [{\rm (i)}]
 If there exists an integral function $\varphi(\cdot)$ such that each component of $\hat{\by{u}}(\by{f},\by{C}(\by{f},\xi),t)$ is  bounded by $\varphi(\xi)$, then 
    \begin{equation}
    \limsup_{t\downarrow 0}\by{\mathcal{F}}_t(P)\subset \by{\mathcal{F}}(P),
    \end{equation}
    where $\limsup_{t\downarrow 0}\by{\mathcal{F}}_t(P)$ denotes the outer limit of $\by{\mathcal{F}}_t(P)$ at point $0$ (see 
    Definition \ref{def:outerlim} in the appendix).

\item [{\rm (ii)}] 
Let $\by{f}(P)\in\by{\mathcal{F}}(P)$.
If  $\Upsilon(\by{f}):=\phi(\by{f},P)+\mathcal{N}_D(\by{f})$ is metrically regular (see Definition \ref{def:MR}) at $\by{f}(P)$ for 0 with regularity modulus $\alpha$,
then there exist
a closed neighborhood $\mathcal{V}_{\by{f}(P)}$ of $\by{f}(P)$ and a sufficiently small constant $t_0>0$ such that 
\bgeqn
\label{thm3.1-equ2}
d(\by{f}_t(P),\by{\mathcal{F}}(P))\leq\alpha\sup_{\by{f}\in\mathcal{V}_{\by{f}(P)}}\|\hat{\phi}_t(\by{f},P)-\phi(\by{f},P)\|
\edeqn
for all $\by{f}_t(P)\in\by{\mathcal{F}}_t(P)\cap\mathcal{V}_{\by{f}(P)}$ and $t\in(0,t_0]$, where 
\bgeqn 
\hat{\phi}_t(\by{f},P) :=\mathbb{E}_P[\hat{\by{u}}(\by{f},\by{C}(\by{f},\xi),t)] \;
\inmat{and} \; 
\phi(\by{f},P) :=\mathbb{E}_P[\by{u}(\by{f},\by{C}(\by{f},\xi))],
\label{eq:hat-phi}
\edeqn
$\|\cdot\|$ denotes the Euclidean norm on $\R^N$
and $d(b,B)$ denotes the distance from a point $b$ to a set $B$; if $\Upsilon$ is strongly metrically regular at $\by{f}(P)$ for 0 with same regularity modulus, then there exist a unique $\by{f}_t(P)\in\by{\mathcal{F}}_t(P)\cap\mathcal{V}_{\by{f}(P)}$ such that
\bgeqn
\label{thm3.1-equ3}
 \|\by{f}_t(P)-\by{f}(P)\|\leq\alpha\sup_{\by{f}\in\mathcal{V}_{\by{f}(P)}}\|\hat{\phi}_t(\by{f},P)-\phi(\by{f},P)\|
\edeqn
for all $t\in(0,t_0]$
, where $\Upsilon^{-1}(0):=\{\by{f}\in\R^N: 0\in\Upsilon(\by{f})\}=\by{\mathcal{F}}(P)$.
 \end{itemize}
\end{theorem}

}
Theorem \ref{thm:smth-approx} (i) guarantees the convergence of MLAPUE to LAPUE
as $t$ is driven to $0$. Theorem \ref{thm:smth-approx} (ii) is an implicit function theorem, where \eqref{thm3.1-equ3} gives a bound for the difference of $\by{f}_t(P)$ and $\by{f}(P)$ in terms of the uniform difference of $\hat{\phi}_t(\by{f},P)$ and $\phi(\by{f},P)$. 
Metric regularity is a generalization of Jacobian nonsingularity of a vector-valued function to a set-valued mapping \cite{Robinson76}. 
It is virtually about Lipschitz continuity of the inverse of a set-valued mapping, that is, $\Upsilon(\by{f}):=\phi(\by{f},P)+\mathcal{N}_D(\by{f})$
in our context. Strong regularity requires  the inverse of the set-valued mapping to
be single-valued in the considered neighborhoods.



{\color{black}

\subsection{Sample average approximation and
perturbation of distribution }
\label{subsec2.3:SAA}

In practice, we may use empirical data to solve 
 MSVIP problem \eqref{eq:SVI-LAPUE-smooting}, that is, use empirical travel time dataset 
to approximate 
the probability distribution
of true travel time. 
For example, when the travel time function ${T}_a(\by{v},\xi)$ is the BPR function given in \eqref{func:GBPR}, we can 
use the empirical data of $T_a(\by{v},\xi)$ to derive the empirical data of random capacity $C_a(\xi)$.
To facilitate theoretical analysis, 
we 
assume that $\xi^1,\cdots,\xi^M$
are the corresponding samples and they
are generated by the true probability distribution $P$ of $\xi$. We consider the sample average approximation of MSVIP 
\bgeqn
\label{eq:SVI-LAPUE-smooting-SAA}
\inmat{(SAA-MSVIP)} \quad0\in \frac{1}{M}\sum_{i=1}^M\hat{{\bm u}}({\bm f},\by{C}(\by{f},\xi^i),t)+{\cal N}_{D}({\bm f}).
\edeqn
Let $P_M :=\frac{1}{M}\sum_{i=1}^M\delta_{\xi^i}$ be an empirical distribution of $\xi^1,\cdots,\xi^M$, where $\delta_{\xi^i}$ denotes the Dirac measure at $\xi^i$. 
We use $\by{\mathcal{F}}_t(P_M)$ to denote the set of solutions to \eqref{eq:SVI-LAPUE-smooting-SAA}. 
In practice, we often obtain only one solution
instead of multiple solutions of \eqref{eq:SVI-LAPUE-smooting-SAA}. Since the solution is based on sample data, we call it a statistical estimator of some $\by{f}_t(P)\in\by{\mathcal{F}}_t(P)$.
The next theorem states 
existence of a solution to \eqref{eq:SVI-LAPUE-smooting-SAA} and 
convergence of the solution as the sample size goes to infinity. 

\begin{theorem} 
\label{prop:SAA-uniquesolu}
   Let $\by{f}_t(P)$ be a strongly regular solution to MSVIP problem \eqref{eq:SVI-LAPUE-smooting}, i.e., there exist neighborhoods $\mathcal{V}_0$ and $\mathcal{V}_{\by{f}_t(P)}$ of $0\in\R^N$ and $\by{f}_t(P)$, respectively, such that for $\delta\in\mathcal{V}_0$, 
   the following generalized equation 
   \begin{equation}
   \label{GE-LGE}
\delta\in\mathbb{E}_P[\hat{\by{u}}(\by{f}_t(P),C(\by{f}_t(P),\xi),t)]+\nabla_{\by{f}}\mathbb{E}_P[\hat{\by{u}}(\by{f}_t(P),\by{C}(\by{f}_t(P),\xi),t)](\by{f}-\by{f}_t(P))+\mathcal{N}_D(\by{f})  
   \end{equation} 
   has a unique solution $\by{f}_t(\delta)$ in $\mathcal{V}_{\by{f}_t(P)}$ which is Lipschitz continuous in $\delta$ over $\mathcal{V}_{\by{f}_t(P)}$. 
   Let $\hat{\phi}_t(\by{f},P_M)=\frac{1}{M}\sum_{i=1}^M\hat{{\bm u}}({\bm f},\by{C}(\by{f},\xi^i),t)$.
   If, in addition, both $\hat{\phi}_t(\by{f},P)$ and $\hat{\phi}_t(\by{f},P_M)$ are continuously differentiable in a neighborhood of $\by{f}_t(P)$ and 
   $\|\hat{\phi}_t(\by{f},P_M)-\hat{\phi}_t(\by{f},P)\|_{1,\mathcal{V}_{\by{f}_{t}(P)}}\to 0$ with probability 1 (w.p.1) as $M\to\infty$, where $\|\cdot\|_{1,\mathcal{V}_{\by{f}_{t}(P)}}$ is a norm defined as in \eqref{def:norm-psi},
   then w.p.1 for $M$ large enough the SAA-MSVIP problem \eqref{eq:SVI-LAPUE-smooting-SAA} possesses a unique solution $\by{f}_t(P_M)$ in a neighborhood of $\by{f}_t(P)$ and $\by{f}_t(P_M)\to\by{f}_t(P)$ w.p.1 as $M\to\infty$. 
\end{theorem}

The theorem follows directly from Proposition~21 in  \cite{Shapiro03}.
 It 
requires strong regularity condition which is essentially about coherent orientation of the Jacobian matrix $\nabla_{\by{f}}\mathbb{E}_P[\hat{\by{u}}(\by{f},\by{C}(\by{f},\xi),t)]$,
we refer readers to Robinson \cite{Robinson92} and 
G\"{u}rkan et al \cite{GYR99}
for sufficient conditions.
In this context, the condition depends on the topological structure of the transportation network (i.e. the structure of $\Delta$ matrix) and the 
properties of the arc performance function $\by{T}(\by{v},\xi)$.

}

In the SAA problem \eqref{eq:SVI-LAPUE-smooting-SAA}, we assume that $\xi^1,\cdots,\xi^M$ are generated by the true probability distribution $P$ of $\xi$.
In practice, $P$ is usually unknown and the samples are  obtained via empirical data which contain noise for various reasons 
such measurement/recording errors. We call the former {\em real data} and the latter {\em perceived data} denoted by $\tilde{\xi}^1,\cdots,\tilde{\xi}^M$.
In other words, the  perceived data may be generated by another distribution of $Q$ which deviates from $P$. The shift from $P$ to 
$Q$ may also be interpreted as follows: $P$ is the nominal distribution of $\xi$ which we are familiar with based on experience 
whereas $Q$ is a mixture distribution of $P$ and $H$, i.e., 
$Q=(1-\epsilon)P+\epsilon H$, where $H$ captures random shocks such as unexpected road accidents 
and/or maintenance work which have significant effect on traffic flows on some arcs.  
Another way to interpret 
the data shift or distribution shift is that $P$ 
is the probability distribution of $\xi$ corresponding to the past or present data    
whereas $Q$ is the probability distribution 
for the future, which means that the traffic equilibrium in future might be different from 
the one obtained based on the past or present data. Our main objective in this paper is to derive qualitative and quantitative statistical analysis 
on the difference between the equilibrium based on  the sample data of $Q$ and the one based on the sample data of $P$. If the difference is small, then the equilibrium obtained from past or present available data might provide a viable forecast
of equilibrium in the future. In the case that 
the deviation is caused by 
measurement 
errors of data, we may interpret the equilibrium based on the samples of $Q$ as model equilibrium whereas the one based on the samples of $P$ as the true equilibrium. A small
difference between the two 
equilibria means that the true equilibrium 
is stable, resilient  or insensitive
against perturbation of the underlying random data. Figure~\ref{eq:steam-licrm}
summarizes this. 
\tikzstyle{startstop} = [rectangle, rounded corners, minimum width = 1.5cm, minimum height=1cm,text centered, draw = black]
\tikzstyle{io} = [trapezium, trapezium left angle=70, trapezium right angle=110, minimum width=2cm, minimum height=1cm, text centered, draw=black]
\tikzstyle{process} = [rectangle, minimum width=3cm, minimum height=1cm, text centered, draw=black]
\tikzstyle{decision} = [diamond, aspect = 3, text centered, draw=black]
\tikzstyle{arrow} = [->,>=stealth]
\begin{figure}[http]
\centering
\begin{tikzpicture}[node distance=2cm]
\node[startstop](origi) at (4.5,2) 
{ Data perturbation};
\node[process](lcrm1) at (0,0)
{Distribution shift};
\node[process](lcrm2) at (4.5,0)
{Data contamination};
\node[process](lcrm3) at (9,0)
{Random shocks};
\node[startstop](llcrm1) at (0,-2)
{$\mathcal{U}(P,\varepsilon):=\{Q|d(P,Q)\leq\varepsilon\}$};
\node[startstop](llcrm2) at (4.5,-2)
{$Q=(1-\varepsilon)P+\varepsilon \delta_{\tilde{\xi}}$};
\node[startstop](llcrm3) at (9,-2)
{$Q=(1-\varepsilon)P+\varepsilon H$};
\node[startstop](origi0) at (4.5,-4) 
{$\mathcal{P}(P,\varepsilon) := \{(1-\varepsilon)P+\varepsilon H|H\in\mathscr{P}(\R^k)\}$};
\draw[->,thick] (origi) to [in = 60, out = -120] (lcrm1);
\draw[arrow,thick] (origi) to  (lcrm2);
\draw[->,thick] (origi) to [in = 110, out = 300] (lcrm3);
\draw[arrow,thick] (lcrm1) -- node [right]  {heterogeneities}  (llcrm1);
\draw[arrow,thick] (lcrm2) -- node [right]  {outliers}  (llcrm2);
\draw[arrow,thick] (lcrm3) -- node [right]  {accidents}  (llcrm3);
\draw[->,thick] (llcrm1) to [in = 90, out = 300] (origi0);
\draw[arrow,thick] (llcrm2) to  (origi0);
\draw[->,thick] (llcrm3) to [in = 90, out = 250] (origi0);
\end{tikzpicture}
\caption{Data perturbations 
}
\label{eq:steam-licrm}
\end{figure}
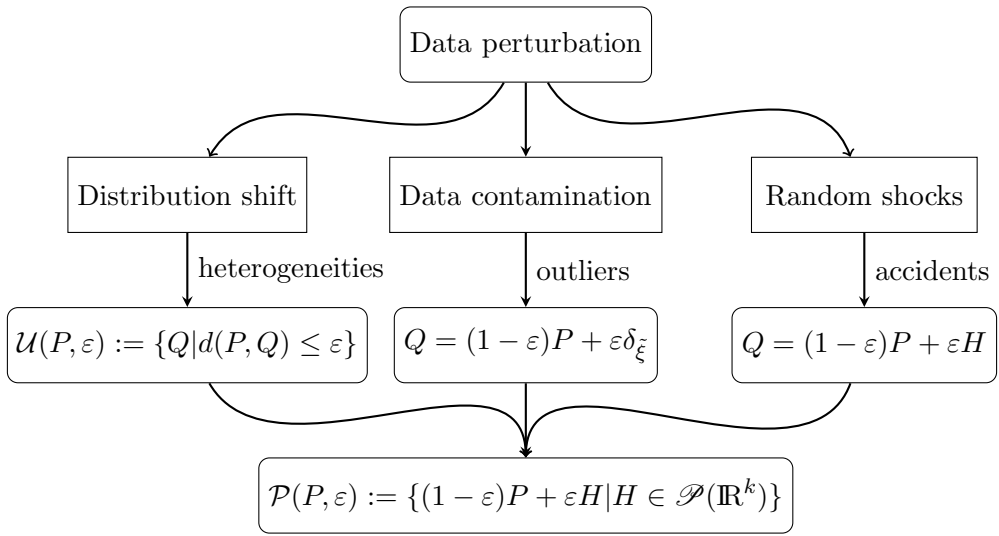

\section{Statistical robustness of MLAPUE}
\label{sec3:SR-MLAPUE}
We proceed our analysis of impact of data perturbation on MLAPUE in two cases:  (a) a few data are perturbed such as outliers and (b) all data are potentially perturbed. To this end,  
we 
write the perturbed MSVIP \eqref{eq:SVI-LAPUE} as:
\bgeqn
\label{eq:GE-purt}
\inmat{(MSVIP-ptb)}\quad 0\in \mathbb{E}_Q[\hat{{\bm u}}({\bm f},\by{C}(\by{f},\xi),t)]+{\cal N}_{D}({\bm f}),
\edeqn
 where $Q\in\mathscr{P}(\R^k)$ is the underlying probability distribution 
 of the perceived data and $\mathscr{P}(\R^k)$ denotes the set of all probability measures over $\R^k$.
In case (a), we may represent $Q$ as a mixture distribution 
 $Q=(1-\varepsilon)P+\varepsilon H$
 where $P$ is the true probability distribution, $H$ is the distribution generating the perturbed data and $\varepsilon$ is a small positive number. 
  In this case,  the perturbed data are generated by $H$ with probability $\varepsilon$.

\subsection{Single data perturbation}
\label{sec3.1:singledata}

We begin with single data perturbation, that is, 
$Q = (1-\varepsilon)P+\varepsilon\delta_{\tilde{\xi}}$, where  
$\delta_{\tilde{\xi}}$ denotes the Dirac measure at $\tilde{\xi}$ and $\tilde{\xi}$ is an outlier which is located outside the support set of $\xi$ under distribution $P$.
In this case, 
MSVIP-ptb (\ref{eq:GE-purt}) can be written as
\bgeqn
\label{eq:GE-purt-single}
0\in \mathbb{E}_{(1-\varepsilon)P+ \varepsilon\delta_{\tilde{\xi}}}[\hat{{\bm u}}({\bm f},\by{C}(\by{f},\xi),t)]+{\cal N}_{D}({\bm f}).
\edeqn
Under the conditions of Proposition \ref{coy:exist-solu-MSVIP}  (a) and (b), we can show
(\ref{eq:GE-purt-single}) has a solution provided that
$\mathbb{E}_{(1-\varepsilon)P+ \varepsilon\delta_{\tilde{\xi}}}[\hat{{\bm u}}({\bm f},\by{C}(\by{f},\xi),t)]$ is bounded. 
As we discussed earlier, (\ref{eq:GE-purt-single}) has
multiple solutions because it concerns path flow equilibrium.
Let $\by{\mathcal{F}}_t((1-\varepsilon)P+ \varepsilon\delta_{\tilde{\xi}})$ denote the set of  solutions to problem \eqref{eq:GE-purt-single}. We investigate
the impact of the perturbation of the probability distribution on the set of solutions by so called
{\em generalized influence function} of $\by{\mathcal{F}}_t(\cdot)$.
Let 
\bgeq
\Psi_{\tilde{\xi},P,t}(\by{f},\varepsilon):=
\mathbb{E}_{(1-\varepsilon)P+ \varepsilon\delta_{\tilde{\xi}}}[\hat{{\bm u}}({\bm f},\by{C}(\by{f},\xi),t)]=
\hat{\phi}_t(\by{f},P)+\varepsilon\left(\hat{\by{u}}(\by{f},\by{C}(\by{f},\tilde{\xi}),t)-\hat{\phi}_t(\by{f},P)\right) 
\edeq
where $\hat{\phi}_t(\by{f},P)=\mathbb{E}_P[\hat{\by{u}}(\by{f},\by{C}(\by{f},\xi),t)]$, and 
\bgeqn
\by{\tilde{\mathcal{F}}}_{\tilde{\xi},P,t}(\varepsilon):=
\by{\mathcal{F}}_t\left((1-\varepsilon)P+ \varepsilon\delta_{\tilde{\xi}}\right)=
\left\{\by{f}\in\R^N\left|0\in\Psi_{\tilde{\xi},P,t}(\by{f},\varepsilon)+\mathcal{N}_D(\by{f})\right.\right\}.
\edeqn
In line with influence function approach in robust statistics \cite{Huber04}, we want to investigate the derivative of 
$\by{\tilde{\mathcal{F}}}_{\tilde{\xi},P,t}(\varepsilon)$ in $\varepsilon$ at $0$, or in other words, the directional derivative of   
$\by{\mathcal{F}}_t\left((1-\varepsilon)P+ \varepsilon\delta_{\tilde{\xi}}\right) =
\by{\mathcal{F}}_t\left(P+ \varepsilon(\delta_{\tilde{\xi}}-P)\right)
$
at $P$ along the ``direction''  $\delta_{\tilde{\xi}}-P$. 
Unfortunately $\by{\tilde{\mathcal{F}}}_{\tilde{\xi},P,t}(\varepsilon)$ is a set-valued mapping. This requires us to
exploit the notion of so-called proto-derivative,  see Definition \ref{def:proto-d} in the appendix. 
Observe that 
$\by{\tilde{\mathcal{F}}}_{\tilde{\xi},P,t}(0)=\by{\mathcal{F}}_t(P)$ and the proto-derivative 
of $\by{\tilde{\mathcal{F}}}_{\tilde{\xi},P,t}(\varepsilon)$ at $0$ for $\by{f}_t(P)\in\by{\tilde{\mathcal{F}}}_{\tilde{\xi},P,t}(0)$ can be written as 
\bgeqn
\by{\tilde{\mathcal{F}}}'_{\tilde{\xi},P,t;0,\by{f}_t(P)}(1)=\lim_{\epsilon\downarrow 0}\sup_{s\to 1}\frac{\by{\tilde{\mathcal{F}}}_{\tilde{\xi},P,t}(\epsilon s)-\by{f}_t(P)}{\epsilon}=\lim_{\epsilon\downarrow 0}\sup_{s\to 1}\frac{\by{\mathcal{F}}_t((1-\epsilon s)P+ \epsilon s\delta_{\tilde{\xi}})-\by{f}_t(P)}{\epsilon}.
\edeqn
Following \cite[Definition 2]{GuoX23}, 
we call $\by{\tilde{\mathcal{F}}}'_{\tilde{\xi},P,t;0,\by{f}_t(P)}(1)$ 
{\em generalized influence function} and denote it by $
\inmat{GIF}(\tilde{\xi};\by{\mathcal{F}}_t(\cdot),P,\by{f}_t(P))$.

The generalized influence function evaluates the influence of an infinitesimal amount of perturbation on the distribution $P$, which 
extends the classical definition of influence function introduced by   
Hampel \cite{Hampel68}.
In general, we are unable to obtain a closed form of $\by{\mathcal{F}}_t(\cdot)$. 
However, we may 
derive an upper bound for
$\inmat{GIF}(\tilde{\xi},\by{\mathcal{F}}_t(P),P)$ by employing the implicit function theorem.

Since $\hat{\by{u}}(\by{f},\by{C}(\by{f},\xi),t)$ is continuously differentiable in $\by{f}$ for almost every $\xi\in\Xi$,
by exploiting 
 proto-derivative of the normal cone
$\mathcal{N}_D(\by{f})$, 
we can represent $\inmat{GIF}(\tilde{\xi};\by{\mathcal{F}}_t(\cdot),P,\by{f}_t(P))$  as a set of solutions of the following equation:
$$
-\hat{\by{u}}(\by{f}_t^*,\by{C}(\by{f}_t^*,\tilde{\xi}),t)+\hat{\phi}_t(\by{f}_t^*,P)-\mathbb{E}[\nabla_{\by{f}}\hat{\by{u}}(\by{f}_t^*,\by{C}(\by{f}_t^*,\tilde{\xi}),t)]\by{f}\in(\mathcal{N}_D)'_{\by{f}_t^*,\hat{\by{u}}^*_t}(\by{f}),
$$
where $\by{f}_t^*=\by{f}_t(P)$ and 
     $\hat{\by{u}}^*_t=-\hat{\phi}_t(\by{f}_t^*,P)$.
Proposition \ref{prop:IF} in the appendix states the details about this. 
Define 
\bgeq
\left\|\inmat{GIF}(\tilde{\xi};\by{\mathcal{F}}_t(\cdot),P,\by{f}_t^*)\right\|:=\min\left\{\nu:\|\eta\|\leq\nu, \forall\eta\in\inmat{GIF}(\tilde{\xi};\by{\mathcal{F}}_t(\cdot),P,\by{f}_t^*)\right\}.
\edeq
The next theorem states a sufficient condition for the boundedness of $\inmat{GIF}(\tilde{\xi};\by{\mathcal{F}}_t(\cdot),P,\by{f}_t^*)$.

\begin{theorem}
\label{thm:IF}
Assume that the generated equations
 \bgeq
 0\in\mathbb{E}_P\left[\nabla_{\by{f}}\hat{\by{u}}(\by{f}_t^*,\by{C}
    (\by{f}_t^*,\xi),t)\right]\by{f}+(\mathcal{N}_D)'_{\by{f}_t^*,\hat{\by{u}}^*}(\by{f})
 \edeq
 has a unique solution $0$. Then the following assertions hold.
 \begin{itemize}
     \item [\inmat{(i)}] $\inmat{GIF}(\tilde{\xi};\by{\mathcal{F}}_t(\cdot),P,\by{f}_t^*)$ is bounded for every $\tilde{\xi}\in\R^k$.
     \item [{\rm (ii)}] If, in addition, $\Upsilon_t(\by{f})=\hat{\phi}_t(\by{f},P)+\mathcal{N}_D(\by{f})$ is strongly metrically regular at $\by{f}_t^*$ for $0$ with regular modulus $\gamma$, then 
 \bgeqn
 \label{GIF:bounded}
 \sup_{\tilde{\xi}\in\R^k}\left\|\inmat{GIF}(\tilde{\xi};\by{\mathcal{F}}_t(\cdot),P,\by{f}_t^*)\right\|
 \leq
\gamma\sup_{\by{f}\in\mathbb{F}_t,\tilde{\xi}\in\R^k}
 \left\|\hat{\phi}_t(\by{f},P)-\hat{\by{u}}(\by{f},\by{C}
    (\by{f},\tilde{\xi}),t)\right\|,
 \edeqn
 where $\mathbb{F}_t$ is defined as in Assumption \ref{asm:compact-MSVIP}.
 \end{itemize}
\end{theorem}





To see how the generalized influence function may be calculated, we 
use a simple example to illustrate.

\begin{example}
The \inmat{MSVIP} \eqref{eq:SVI-LAPUE-smooting} can be reformulated as a complementarity problem:
\bgeq
\by{0}\leq\left(
\begin{matrix}
\by{f}\\
\tilde{\by{z}}
\end{matrix}\right)\bot
\left(
\begin{matrix}
\hat{\phi}_t(\by{f},P)-\Pi^\top\tilde{\by{z}}\\
\Pi\by{f}-\by{q}
\end{matrix}\right)\geq\by{0},
\edeq
where $\tilde{\by{z}}\in\R^{W}$ is a vector of minimum OD disutilities corresponding to $\by{f}$, $\hat{\phi}(\by{f},P)$ is defined as in \eqref{eq:hat-phi} and $\Pi$ is the OD-path incidence matrix.
Let $\by{f}_t^*\in\by{\mathcal{F}}_t(P)$ be a solution to \inmat{MSVIP} problem \eqref{eq:SVI-LAPUE-smooting}, 
$\hat{\by{u}}^*_t=-\mathbb{E}_P[\hat{\by{u}}(\by{f}^*_t,\by{C}(\by{f}^*_t,\xi),t)]$ and $\tilde{z}=\min_{1\leq r\leq N}\{-\hat{u}^*_{t,r}\}$, where $\hat{u}^*_{t,r}$ represents the $r$th component of vector $\hat{\by{u}}_t^*$. Let 
\bgeq
\mathcal{I}_{+}:=\left\{r\in\{1,\cdots,N\}|f_{t,r}^*>0, -\hat{u}_{t,r}^*-[\Pi^\top]_r\tilde{z}=0\right\},\\
\mathcal{I}_{0}:=\left\{r\in\{1,\cdots,N\}|f_{t,r}^*=0, -\hat{u}_{t,r}^*-[\Pi^\top]_r\tilde{z}=0\right\},\\
\mathcal{I}_{-}:=\left\{r\in\{1,\cdots,N\}|f_{t,r}^*=0, -\hat{u}_{t,r}^*-[\Pi^\top]_r\tilde{z}<0\right\},
\edeq
where $f_{t,r}^*$ denotes the $r$th component of vector $\by{f}_t^*$ and $[\Pi^\top]_r$ denotes the  $r$th row vector of matrix $\Pi^\top$.
Since
\bgeqn
 \inmat{gph}(\mathcal{N}_D)'_{\by{f}_t^*,\hat{\by{u}}_t^*}
&=&\inmat{gph} (\mathcal{N}_{\R^N_+})'_{\by{f}_t^*,\hat{\by{u}}_t^*+\Pi^\top \tilde{z}}=\inmat{gph}D\mathcal{N}_{\R^N_+}(\by{f}_t^*|\hat{\by{u}}_t^*+\Pi^\top \tilde{z})\nonumber\\
&=&\mathcal{T}_{\inmat{gph}\mathcal{N}_{\R_+^N}}(\by{f}_t^*,\hat{\by{u}}_t^*+\Pi^\top \tilde{z})\nonumber\\
 &=&\left\{(\by{f}_t,\hat{\by{u}}_t)\in\R^N\times\R^N\left|\begin{array}{l}
\hat{u}_{t,r}+\Pi^\top\tilde{z}=0(r\in\mathcal{I}_+),{f}_{t,r}=0 (r\in\mathcal{I}_{-}) \\
{f}_{t,r}=0, \hat{u}_{t,r}+\Pi^\top\tilde{z}=0 \;\inmat{or}\;\\
{f}_{t,r}=0, \hat{u}_{t,r}+\Pi^\top\tilde{z}\leq0(r\in\mathcal{I}_{0})
 \end{array}
 \right.
 \right\}\nonumber,
 \edeqn
 where $\inmat{gph}(\mathcal{N}_D)'_{\by{f}_t^*,\hat{\by{u}}_t^*}$ denotes the graph of proto-derivative $(\mathcal{N}_D)'_{\by{f}_t^*,\hat{\by{u}}_t^*}$, $\inmat{gph}D\mathcal{N}_{\R^N_+}(\by{a}|\by{b})$ denotes the graph of graphical derivative $D\mathcal{N}_{\R^N_+}(\by{a}|\by{b})$ and $\mathcal{T}_{\inmat{gph}\mathcal{N}_{\R_+^N}}(\by{a},\by{b})$ denotes the tangent cone of $\inmat{gph}\mathcal{N}_{\R_+^N}$ at point $(\by{a},\by{b})$ (see Definition \ref{def:proto-d} in the appendix). 
 Let 
 $$
 \by{g}(\by{f},t) :=-\hat{\by{u}}(\by{f}_t^*,\by{C}(\by{f}_t^*,\tilde{\xi}),t)-\hat{\phi}_t(\by{f}_t^*,P)-\mathbb{E}_P[\nabla_{\by{f}} \hat{\by{u}}(\by{f}_t^*,\by{C}(\by{f}_t^*,\xi),t)]\by{f}+\Pi^\top\tilde{z}.
    $$
    By Proposition \ref{prop:IF} in the appendix,
    \bgeq
   \inmat{GIF}(\tilde{\xi};\by{\mathcal{F}}_t(\cdot),P,\by{f}_t^*)=
    \left\{\by{f}\in\R^N\left|\begin{array}{l}
g_r(\by{f},t)=0  \;(\inmat{for}\;r\in\mathcal{I}_+),{f}_{r}=0 \;(\inmat{for}\;r\in\mathcal{I}_{-}) \\
{f}_{r}=0, g_r(\by{f},t)=0 \;\inmat{or}\;
{f}_{r}=0, g_r(\by{f},t)\leq0\;(\inmat{for}\;r\in\mathcal{I}_{0})
 \end{array}
 \right.
 \right\}.
    \edeq
The rhs is a complementarty problem which can be easily solved by an existing code \cite{FerrisP97}. 
 \end{example}



\subsection{Breakdown point analysis}
\label{sec3.1:BP}
 We now turn to discuss the case that 
 the data set has multiple outliers instead of a single outlier
 and
 how the number of outliers may affect the 
quality of MLAPUE by using 
 the notion of breakdown point
introduced   by Hampel \cite{Hampel68}.
 Let $P_M$ denote the empirical distribution 
 based on 
 unperturbed samples
 and  
 $P_{M,m}$ is the one where
 $m$ of the samples 
 are outliers. 
 Let $\hat{\by{f}}_t(P_M)\in\hat{\by{\mathcal{F}}_{t}}(P_M)$ and $\hat{\by{f}_{t}}({P}_{M,m})\in\hat{\by{\mathcal{F}}_{t}}({P}_{M,m})$ denote 
  the solution obtained from solving SAA-MSVIP problem \eqref{eq:SVI-LAPUE-smooting-SAA} with $P_M$ and $P_{M,m}$. 
Following \cite{Huber04}, we define 
the  finite-sample breakdown point for estimator $\hat{\by{f}}_{t}$ 
as
\bgeqn
\epsilon^*(\hat{\by{f}}_t,P_M):=\max_m\left\{\frac{m}{M}\left|\sup_{\tilde{P}_{M,m}}\|\hat{\by{f}}_{t}(P_M)-\hat{\by{f}}_{t}(\tilde{P}_{M,m})\|<\infty\right.\right\},
\edeqn
where the supremum is taken because 
outliers are not identifiable and the maximum 
is taken w.r.t. the number of outliers.
{\color{black} $\epsilon^*(\hat{\by{f}}_t,P_M)$ represents the maximum proportion of outliers within the entire dataset such that the deviation between the estimator based on the mixed data (containing both outliers and unperturbed data) and the estimator based solely on unperturbed data remains bounded.}


In the case when the data are generated by a mixture distribution $(1-\varepsilon)P+\varepsilon H$ where
$P$ is the nominal distribution generating good data 
and $H$ is
the distribution generating outliers, 
we may define the breakdown point as
\bgeqn
\epsilon^*(\hat{\by{f}}_t,P):=\min_{\varepsilon\in[0,1]}\left\{\varepsilon:\sup_{P'\in\mathcal{P}(P,\varepsilon)}\|\hat{\by{f}}_t(P)-\hat{\by{f}}_t(P')\|=\infty\right\},
\edeqn
where
\bgeqn
\mathcal{P}(P,\varepsilon) := \{(1-\varepsilon)P+\varepsilon H|H\in\mathscr{H}\subset\mathscr{P}(\R^k)\}
\edeqn
for $\varepsilon\in[0,1]$ and $\mathscr{H}$ is a subset of distributions generating 
the outliers. In this formulation,
the information about the true probability distribution $H$ generating outliers is incomplete. Consequently, we consider the worst-case distribution from  $\mathscr{H}$.

\subsection{All data perturbation}
\label{sec:alldata}



The influence function and breakdown point approaches 
are concerned with the cases that 
perceived sample data can be  divided into two categories: 
good data which are not perturbed 
and bad data which are outliers. 
In some cases, all data 
could be potentially perturbed 
but they are not necessarily outliers. 
Let
$\tilde{\xi}^1,\cdots,\tilde{\xi}^M$ denote
the perceived data and 
$Q_M:=\frac{1}{M}\sum_{i=1}^M\delta_{\tilde{\xi}^i}$. 
Assume that $\tilde{\xi}^1,\cdots,\tilde{\xi}^M$ are i.i.d copies of $\tilde{\xi}$ with distribution $Q$, 
where $Q$ be regarded as  a perturbation of $P$. 
Differing from (\ref{eq:SVI-LAPUE-smooting-SAA}),
we consider sample average approximation of the MSVIP based 
perceived data:
\bgeqn
\label{eq:GE-purt-saa}
0\in \frac{1}{M}\sum_{i=1}^M\hat{{\bm u}}({\bm f},\by{C}(\by{f},\tilde{\xi}^i),t)]+{\cal N}_{D}({\bm f}),
\edeqn
which 
may be viewed as the sample average approximation of MSVIP-ptb problem \eqref{eq:GE-purt}.
Let $\by{\mathcal{F}}_t(Q_M)$ 
denote the set of solutions to problem \eqref{eq:GE-purt-saa} and $\by{f}_t(Q_M)\in \by{\mathcal{F}}_t(Q_M)$. 
We use 
$\by{f}_t(Q_M)$ as an estimator of a solution to MSVIP (\ref{eq:SVI-LAPUE-smooting}).
Unfortunately $\by{f}_t(Q_M)$ converges 
to $\by{f}_t(Q)$ 
instead of $\by{f}_t(P)$ (a solution to MSVIP (\ref{eq:SVI-LAPUE-smooting})) whereas 
$\by{f}_t(P_M)$ converges to $\by{f}_t(P)$ 
as $M\to \infty$ under some moderate 
conditions.
In what follows, we derive conditions under which 
the distributions of $\by{f}_t(Q_M)$ and 
$\by{f}_t(P_M)$ are close under some metric.
By then, we will be able to use
$\by{f}_t(Q_M)$ as an estimator of $\by{f}_t(P)$. 
To this end, we need 
additional conditions on $\by{T}(\by{v},\xi)$
and $\hat{\by{u}}(\by{f},\by{C}(\by{f},\xi),t)$.

\begin{assumption}
\label{asm:ATT}
For the actual arc travel time function $\by{T}(\by{v},\xi)$,  we assume that
there exists a constant $ L >$ such that 
    \bgeqn
    \|\by{T}(\by{v},\xi)-\by{T}(\by{v},\xi')\|\leq L\|\xi-\xi'\| \;\inmat{and}\;
    \|\nabla_{\by{v}}\by{T}(\by{v},\xi)-\nabla_{\by{v}}\by{T}(\by{v},\xi')\|\leq  L \|\xi-\xi'\|
   \edeqn
 for all $\by{v}\in\mathbb{V}_t$, $\xi,\xi'\in\Xi$,  where $\mathbb{V}_t$ is defined as in Assumption \ref{asm:compact-MSVIP} and $\Xi\subset \R^k$.
\end{assumption}
    To justify Assumption \ref{asm:ATT} as well as 
    Assumption \ref{asm:T-value} in Section \ref{sec2:MLPAUE}, we consider the case that each component of $\by{T}(\by{v},\xi)$ is the common BPR function (see  \eqref{func:GBPR}). 
        It is obvious that $\by{T}(\by{v},\xi)$ is continuously differentiable in $\by{v}$ for any fixed $t_a^0,b_a,n_a>0$. 
        Moreover, it is not hard to prove that $\by{T}(\by{v},\xi)$ satisfies Assumption \ref{asm:T-value} (b) and (c) if $C_a(\xi)$ is 
        continuous in $\xi$. This verifies Assumption \ref{asm:T-value}.
    To see how 
    Assumption \ref{asm:ATT} may be fulfilled, we note that
\bgeqn
\label{func:GBPR-1}
\by{T}(\by{v},\xi)=\{T_a(\by{v},\xi)\}_{a=1}^A=\left\{t_a^0\left(1+b_a\left(\frac{v_a}{C_a(\xi)}\right)^{n_a}\right)\right\}_{a=1}^A,
\edeqn
we can obtain that 
\bgeq
\nabla_{\by{v}}\by{T}(\by{v},\xi)=\left\{\frac{t_a^0n_a}{C_a(\xi)}\left(1+b_a\left(\frac{v_a}{C_a(\xi)}\right)^{n_a-1}\right)\cdot e_a\right\}_{a=1}^A,
\edeq
where $e_a$ denotes the unit vector whose $a$th component is 1 and all the other components are 0. Thus 
\bgeq
&&\|\nabla_{\by{v}}\by{T}(\by{v},\xi)-\nabla_{\by{v}}\by{T}(\by{v},\xi')\|\\
&&=\max_{a\in\{1,\cdots,A\}}\left|\frac{t_a^0n_a}{C_a(\xi)}\left(1+b_a\left(\frac{v_a}{C_a(\xi)}\right)^{n_a-1}\right)-\frac{t_a^0n_a}{C_a(\xi')}\left(1+b_a\left(\frac{v_a}{C_a(\xi')}\right)^{n_a-1}\right)\right|\\
&&=\left|\frac{t_{a^*}^0n_{a^*}}{C_{a^*}(\xi)}-\frac{t_{a^*}^0n_{a^*}}{C_{a^*}(\xi')}+\left(\frac{t_{a^*}^0n_{a^*}b_{a^*}v^{n_{a^*}-1}_{a^*}}{C^{n_{a^*}}_{a^*}(\xi)}-\frac{t_{a^*}^0n_{a^*}b_{a^*}v^{n_{a^*}-1}_{a^*}}{C^{n_{a^*}}_{a^*}(\xi')}\right)\right|.
\edeq
The discussions above show that 
$\by{T}(\by{v},\xi)$ satisfy Assumptions \ref{asm:ATT}~(a) and (b) when 
$C_a(\xi)$ is Lipschitz continuous w.r.t.~$\xi$, 
$|C^{n_{a^*}}_{a^*}(\xi)|$ is lower bounded by a positive constant 
and $\Xi$ is a compact set. 

\begin{proposition}
\label{prop:u-LipC}
For the disutility function $\hat{\by{u}}(\by{f},\by{C}(\by{f},\xi),t)$, the following assertions hold. 
\begin{itemize}
    \item [{\rm (i)}] There exists a constant $L>0$ such that
\begin{equation}
    \|\hat{\by{u}}(\by{f},\by{C}(\by{f},\xi),t)-\hat{\by{u}}(\by{f},\by{C}(\by{f},\xi'),t)\|\leq  L\eta_{\Delta^\top}(\theta_1+\theta_2)  
    \|\xi-\xi'\|.
\end{equation}
for all $\xi,\xi'\in\Xi, \by{f}\in\mathbb{F}_t$ and $t\geq 0$.
    \item [{\rm (ii)}] There exist a constant $\hat{C}(t)$ dependent on $t$ such that  
\begin{equation}
     \|\nabla_{\by{f}}\hat{\by{u}}(\by{f},\by{C}(\by{f},\xi),t)-\nabla_{\by{f}}\hat{\by{u}}(\by{f},\by{C}(\by{f},\xi'),t)\|\leq 
     L\eta_{\Delta^\top}\left(\theta_1\eta_{\Delta}+\theta_2\eta_{\Delta}+\theta_2\hat{C}(t)\right) \|\xi-\xi'\|,
\end{equation}
\end{itemize}
for all $\xi,\xi'\in\Xi, \by{f}\in\mathbb{F}_t$ and $t>0 $, where $\mathbb{F}_t$ is defined as in Assumption \ref{asm:compact-MSVIP}, $\hat{C}(t)=\frac{C_0'}{2t}$, $C_0'=\sup_{\by{f}\in\mathbb{F}_t,\xi\in\Xi}\|\nabla_{\by{f}}\by{C}(\by{f},\xi)\|$ and $\eta_\Delta=\|\Delta\| $ for matrix $\Delta\in\R^{A\times N}$.
\end{proposition}

The proposition above implies that $\hat{\by{u}}(\by{f},\by{C}(\by{f},\xi),t)$ is locally Lipschitz continuous in $\xi$ and its Jacobian $\nabla_{\by{f}}\hat{\by{u}}(\by{f},\by{C}(\by{f},\xi),t)$ is locally Lipschtz continuous with the modulus related to parameter $t$. 
The explicit form of the 
Jacobian matrix 
of 
$\hat{\by{u}}(\by{f},\by{C}(\by{f},\xi),t)$
can be obtained.  
Observe that $\nabla_{\by{f}}\by{g}(\by{f},\by{C}(\by{f},\xi))=\theta_1\Delta^\top\nabla_{\by{v}}\by{T}(\Delta\by{f},\xi)\Delta$ and 
\begin{equation}
    \left[\nabla_{\by{f}}\hat{\by{h}}(\by{f},\by{C},t)\right]_{r}=\left\{
    \begin{array}{ll}
    \theta_2[\Delta^\top]_r\nabla_{\by{v}}\by{T}(\Delta\by{f},\xi)\Delta, & \inmat{for} \ t< C_r-\tau_k,\\
    \frac{\theta_2\left(C_r(\by{f},\xi)+t-\tau_k\right)}{2t}[\Delta^\top]_r\nabla_{\by{v}}\by{T}(\Delta \by{f},\xi)\Delta, & \inmat{for} \ -t\leq C_r-\tau_k\leq t,\\
    0\in\R^{1\times N}, & \inmat{for} \ C_r-\tau_k< -t,
    \end{array}
    \right.
\end{equation} 
for $r=1,\cdots,N$ and $k=1,\cdots,W$, where $[A]_r$ denotes the $r$th row vector of matrix $A$. 
Thus
\begin{equation}
    \left[\nabla_{\by{f}}\hat{\by{u}}(\by{f},\by{C},t)\right]_{r}=\left\{
    \begin{array}{ll}
    (\theta_1+\theta_2)[\Delta^\top]_r\nabla_{\by{v}}\by{T}(\Delta\by{f},\xi)\Delta, & \inmat{for} \ t< C_r-\tau_k,\\
    \left(\theta_1+\frac{\theta_2\left(C_r(\by{f},\xi)+t-\tau_k\right)}{2t}\right)[\Delta^\top]_r\nabla_{\by{v}}\by{T}(\Delta \by{f},\xi)\Delta, & \inmat{for} \ -t\leq C_r-\tau_k\leq t,\\
    \theta_1[\Delta^\top]_r\nabla_{\by{v}}\by{T}(\Delta \by{f},\xi)\Delta, & \inmat{for} \ C_r-\tau_k< -t.
    \end{array}
    \right.
\end{equation} 
Notice that $\nabla_{\by{f}}\hat{\by{u}}(\by{f},\by{C},t)$ can be written as $\diag(b_1,\cdots,b_N)\Delta^\top\nabla_{\by{v}}T(\Delta\by{f},\xi)\Delta$, where $b_r$ is either $(\theta_1+\theta_2)$ or $\left(\theta_1+\frac{\theta_2\left(C_r(\by{f},\xi)+t-\tau_k\right)}{2t}\right)$ or $\theta_1$ for $r=1,\cdots,N$.  


Under the proposition above and Lemma \ref{lem:LIp-solu} in the appendix, we are able to derive the main stability result about an isolated MLAPUE solution to (\ref{eq:SVI-LAPUE-smooting}).


\begin{theorem}[Stability of MLAPUE]
\label{thm:Lips-solu}
Let $\by{f}_t(P)\in\by{\mathcal{F}}_t(P)$ be a strongly regular solution to MSVIP \eqref{eq:SVI-LAPUE-smooting} and $\mathcal{O}_t\subset\mathbb{F}_t$ be an open neighborhood of $\by{f}_t(P)$, where $\mathbb{F}_t$ is  defined as in Assumption \ref{asm:compact-MSVIP}.
Let 
\bgeqn
\label{def:p-moment}
{\cal M}_k^1 := \left\{{\bm Q}\in \mathscr{P}(\Xi): \int_{\Xi} \|\xi\|{\bm Q}(d{\xi})< \infty \right\}.
\edeqn
The following assertion hold. 
\begin{itemize}
 
    \item [{\rm (i)}] Under Assumption \ref{asm:ATT}, 
   there exist a unique 
   vector-valued function 
   $\by{f}_t(\cdot)$  defined 
   in a neighborhood of $P$ and 
   positive 
   constants $\delta$ and $\kappa$ such that
    \bgeqn
    \|\by{f}_t(Q_1)-\by{f}_t(Q_2)\|\leq L\eta_{\Delta^\top}\kappa\left(\theta_1+\theta_2+\theta_1\eta_{\Delta}+\theta_2\eta_{\Delta}+\theta_2\hat{C}(t)\right)\dd_K(Q_1,Q_2)
    \edeqn
    for any $Q_1,Q_2\in{\cal M}_k^1$ satisfying $\dd_K(Q_1,P)\leq\delta$ and $\dd_K(Q_2,P)\leq\delta$, where ${\cal M}_k^1$ is defined as in \eqref{def:p-moment}, $\hat{C}(t)$ is defined as in Proposition \ref{prop:u-LipC} and $\dd_K(Q_1,Q_2)$ denotes the Kantorovich metric between $Q_1$ and $Q_2$, defined as in \eqref{dist-Kan} in the appendix.

    \item [{\rm (ii)}] Let
    \bgeqn
    \mathcal{M}_{k,M_0}^{\gamma}:=\left\{Q\in\mathscr{P}(\R^k):\int_{\R^k}\|\xi\|^{\gamma}Q(d\xi)\leq M_0\right\}
    \edeqn
    for some constants $\gamma>1$ and $M_0>0$. If, in addition, 
    there is a continuous solution $\by{f}_t(\cdot):\mathcal{M}_{k,M_0}^{\gamma}\to\R^N$ such that the strongly regular condition holds at $\by{f}_t(P')$ for every $P'\in\mathcal{M}_{k,M_0}^{\gamma}$, then there exists a constant $\tilde{C}$ such that 
    \bgeqn\label{ieq:SR-solu}
    \|\by{f}_t(Q_1)-\by{f}_t(Q_2)\|\leq L\tilde{C}\eta_{\Delta^\top}\left(\theta_1+\theta_2+\theta_1\eta_{\Delta}+\theta_2\eta_{\Delta}+\theta_2\hat{C}(t)\right)\dd_K(Q_1,Q_2), 
    \edeqn
    for all $ Q_1,Q_2\in\mathcal{M}_{k,M_0}^{\gamma}$.
\end{itemize} 
\end{theorem}

Part (i) of the theorem is an implicit function theorem 
for MSVIP-ptb \eqref{eq:GE-purt}  in the space of probability measures ${\cal M}_k^1$. Part (ii) is about 
global Lipschitz continuity of MLAPUE solution mapping
over the $\mathcal{M}_{k,M_0}^{\gamma}$.  
In the case that when the support set of $\xi$ is 
a compact set, $\mathcal{M}_{k,M_0}^{\gamma}$ coincides with $\mathscr{P}(\Xi)$.
With Theorem \ref{thm:Lips-solu}, we are ready to state 
the main 
quantitative statistical robustness of 
an estimator of
MLAPUE ${\by{f}}_{t}(\cdot)$.
The assumption on the continuous solution $\by{f}_t(\cdot)$  over $\mathcal{M}_{k,M_0}^{\gamma}$ is strong and theoretical in general. It may be fulfilled under some circumstance when MSVIP \eqref{eq:SVI-LAPUE-smooting}
has a unique solution for all $P\in \mathscr{P}(\Xi)$. Again, this depends on the problem structure including the structure of the network and the properties of $\by{T}(\by{v},\xi)$.

With Theorem~\ref{thm:Lips-solu}, we are ready to state 
robustness of statistical estimator of MLAPUE
in the next theorem.

\begin{theorem}[Statistical robustness of MLAPUE estimator]
\label{thm:SR-LAPUE}
    Assume the conditions of Theorem~\ref{thm:Lips-solu}~(iii) hold. Let $\hat{\by{f}}_{M,t}=\by{f}_t(Q_M)$ be a statistical estimator
of $\by{f}_t(Q)$, where $\by{f}_t(\cdot)$ is defined as 
in Theorem \ref{thm:Lips-solu}.
    If the support of $\xi$, denoted by $\Xi$, is a compact set, then 
    \bgeqn
    \label{ieq:SR-solu-1}
    \dd_K\left(P^{\otimes M}\circ\hat{\by{f}}_{M,t}^{-1},Q^{\otimes  M}\circ\hat{\by{f}}_{M,t}^{-1}\right)\leq L\tilde{C}\eta_{\Delta^\top}\left(\theta_1+\theta_2+\theta_1\eta_{\Delta}+\theta_2\eta_{\Delta}+\theta_2\hat{C}(t)\right)\dd_K(P,Q)
    \edeqn
    for all $M$ and $P,Q\in\mathscr{P}(\Xi)$, where   $\hat{C}(t)$ is defined as in Proposition \ref{prop:u-LipC}, $\tilde{C}$ is defined as in Theorem \ref{thm:Lips-solu} and $P^{\otimes M}$ denotes the Cartesian product of $\underbrace{P\times\cdots\times P}_{M}$. 
    
\end{theorem}

The lhs of (\ref{ieq:SR-solu-1}) is the 
difference between 
the empirical distribution 
$P^{\otimes M}\circ\hat{\by{f}}_{M,t}^{-1}$ 
based on unperturbed samples and 
the empirical distribution $Q^{\otimes M}\circ\hat{\by{f}}_{M,t}^{-1}$ 
based on perturbed samples under the Kantorovich metric. The rhs provides a bound for the lhs
in terms of the difference between the original probability distributions generating the unperturbed and perturbed samples ($P$ and $Q$) under the Kantorovich metric.   
The constant at the rhs 
depends on $t$. As $t$ goes to $0$, the 
constant goes to infinity, which means the quantitative statistical robustness result is not applicable to the 
LAPUE. 
Since the inequality is derived under some sufficient conditions rather than necessary conditions, it does not mean  that  the LAPUE is not statistically robust -- it is simply because we are unable to establish the result under the current framework of analysis.  

{\color{black} 
The 
theoretical result is useful in at least two cases: (a) $P$ and $Q$ are known but in actual calculations, only empirical data of $Q$ are used. This is either because
 errors 
occur in the process of data generation, processing and recording, or the distribution of validation data (for the future)  
is shifted from the distribution of training data (in the past); (b) the difference between $Q$ and $P$
is known in the sense that
the shift/perturbation of the distribution   
is within a controllable range.  
In the case only when perceived data (the sample data of $Q$) are known, we will not be able to say much about the quality of MLAPUE.
Since the result is built on Theorem \ref{thm:Lips-solu},
it might be desirable 
to relax some of the conditions imposed on 
Theorem~\ref{thm:Lips-solu}
such as strong regularity to extend the applicability of 
Theorem~\ref{thm:SR-LAPUE}.

}

\section{Numerical tests}
\label{se4：NT}
We have carried out 
 some numerical experiments on 
 the MLAPUE model. In this section, we report the 
 test results. For the vector of 
 arc travel time functions $\by{T}(\by{v},\xi)$, we 
 set each of its 
 components 
 $T_a(\by{v},\xi)$ to 
 be a GBPR function,
\bgeqn
\label{func:GBPR-2}
T_a(\by{v},\xi)=t_a^0\left(1+b_a\left(\frac{v_a}{C_a(\xi)}\right)^{n_a}\right),
\edeqn
where $t_a^0,b_a$ and $n_a$ are given parameters,
$C_a(\xi)$ denotes the random capacity of arc $a$ 
in scenario $\xi$ and $v_a$ 
is flow at arc $a$.

\subsection{Example 1: a simple 2-OD network}
To 
examine the effect 
of the perturbation of probability distribution 
on arc travel time, 
we first consider a 2-OD pairs network varied from \cite{zhu2019}, see Figure \ref{fig:network-1}. 
For network 1, we let the path from node $1$ to node $2$ be the 1-OD pair with demand $q_1=3500$ and from node $1$ to node $3$ be the 2-OD pair with demand $q_2=4000$. In this case, we 
set the GBPR function as
\bgeqn
\label{func:GBPR-3}
T_{a}(\by{v},\xi)=t_{a}^0\left(1+0.15\left(\frac{v_{a}}{C_{a}(\xi)}\right)^4\right).
\edeqn
Detail
of other parameters
are given  in Table \ref{tab:detail-n1}.
Following \cite{zhu2019}, we assume that $C_a(\xi)$ follows a normal distribution $C_{a}(\xi)\sim N(\mu_{a},\sigma_{a})$ and denote its CDF by $F(x;\mu_a,\sigma_a)$ for $a=1,\cdots,6$.  
It is easy to verify that the arc-path incidence matrix of network 1 is of full 
column  rank, and that the Jacobian matrix of travel time function $\by{T}(\by{v},\xi)$ is a diagonal matrix with full rank for $\by{v}\neq0$. Thus corresponding MSVIP problem \eqref{eq:SVI-LAPUE-smooting} satisfies strong regularity condition and has a unique MLAPUE.

\begin{figure}[htbp]
    \centering
	\begin{minipage}{0.3\linewidth}
		\centering
		\begin{tikzpicture}
\node[circle,
minimum width =3pt,
minimum height =3pt, draw=black] (1) at(0,0){1};
\node[circle,
minimum width =3pt,
minimum height =3pt, draw=black] (2) at(6,2){2};
\node[circle,
minimum width =3pt,
minimum height =3pt, draw=black] (3) at(3,-2){3};
\node[circle,
minimum width =3pt,
minimum height =3pt, draw=black] (4) at(3,0){4};
\node[circle,
minimum width =3pt,
minimum height =3pt, draw=black] (5) at(6,0){5};
\draw[arrow] (1) -- node [above] {$1$} (2);
\draw[arrow] (1) -- node [above] {$2$} (3);
\draw[arrow] (1) -- node [above] {$3$} (4);
\draw[arrow] (4) -- node [above] {$4$} (5);
\draw[arrow] (5) -- node [right] {$5$} (2);
\draw[arrow] (4) -- node [right] {$6$} (3);
\end{tikzpicture}
		\caption{Test network 1 ($\theta_0=0,\theta_1=1,\theta_2=2,\tau_1=27,\tau_2=22$)}
		\label{fig:network-1}
	\end{minipage}
	\hfill
	\begin{minipage}{0.5\linewidth}
		\centering
        \begin{tabular}{llll}
    \toprule
    Demand\\
    \midrule
    OD pair & Demand & Path\\
    1-2     & 3500   & $\{1\},\{3,4,5\}$\\
    1-3     & 4000   & $\{3,6\},\{2\}$\\
    \midrule
     Arc\\
   \midrule
         \textbf{Arc} & $t_{a_i}^0$ & $\mu_{a_i}$ &  $\sigma_{a_i}$ \\
    \midrule
        $ 1$& 16 &  1500  &  5\\
        $ 2$& 13 &  1500  &  30\\
        $ 3$& 9  &  3600  &  80\\
        $ 4$& 2  &  3600  &  5\\
        $ 5$& 3  &  1500  &  6\\
        $ 6$& 3  &  1500  &  30\\
    \bottomrule
    \end{tabular}
        \captionof{table}{Details for network 1}
        \label{tab:detail-n1}
	    \end{minipage}
\end{figure}

The UE path flows are $(2182,1318,850,3150)$ with corresponding disutility values \linebreak $(26.75,26.79,21.52,21.54)$. 
For parameter $t=0.01$,
the MLAPUE path flow are (2193,1306,860,
3139) with corresponding disutility values $(27.10,27.08,21.67,21.71)$. 
We can see that the flow patterns are similar, the  
disutility values are larger in the MLAPUE model
because of addition of the penalty of lateness.

Next, we 
investigate the effect of perturbation of the probability distribution on 
MLAPUE. We consider the case that 
only the distribution of the capacity at arc
$a=1$ is perturbed.
Let $P_a$ denote the distribution 
of the capacity at arc $a$ and $F(x;\mu_a,\sigma_a)$ its CDF.
Let  $Q_a$ 
be the perturbed distribution with CDF
\bgeqn
F_{Q_a}(x)=\left\{
\begin{array}{ll}
F(x;\mu_a,\sigma_a), &\inmat{for}\; x\leq x_0,\\
q + \beta(x-x_0), & \inmat{for}\; x_0\leq x \leq x_1,\\
1,&\inmat{for}\; x>x_1,
\end{array}\right.
\edeqn
where $x_0=F^{-1}(q;\mu_a,\sigma_a)$, $x_1=x_0+\frac{1}{\beta}(1-q)$, $\beta$ and $q\in(0,1)$ are fixed positive constants.
Figure~\ref{fig:CDF-PERT} 
depicts 
the difference of the two CDFs.
We use the CDFs to generate independent 
and identically distributed 
samples $C_a^j$ and $\tilde{C}_a^j$ for $j=1,\cdots,M$ respectively.
Let $P_M = \frac{1}{M}\sum_{j=1}^M\delta_{{C}_a^j}$ and $Q_M = \frac{1}{M}\sum_{j=1}^M\delta_{\tilde{C}_a^j}$.
Denote the MLAPUE 
obtained from solving 
\eqref{eq:GE-purt-saa} 
by $\by{f}_t(Q_M)$. 
 We use the CDFs of 
 $Q_a$ and $P_a$ to generate $L$ groups of samples with size $M=1000$. 
 We calculate 
 the MLAPUE 
 $\by{f}_t(Q_M)$ and $\by{f}_t(P_M)$ for each group of samples with
 corresponding arc flows 
 $\by{v}(Q_M)=\Delta\by{f}_t(Q_M)$ and $\by{v}(P_M)=\Delta\by{f}_t(P_M)$. 
 We can use each of 
the  $L$ data points 
to construct 
empirical CDFs. 
Figure \ref{fig:CDFs-500} depicts the CDFs of $v_{1}(P_M)$ and $v_1(Q_M)$  when $t=0.1$. We can see that the flow on arc $1$ is mainly distributed within the range $[2192,2195]$ and the blue dashed line approximates the red ones very well. Figure \ref{fig:arcflow_EMP} simulates the empirical distribution of $v_1(P_M)$ and $v_1(Q_M)$ when $L=10$ and $L=30$.
\begin{figure}[H]
  \centering
  \subfloat[]{
     \label{fig:CDF-PERT}
     \includegraphics[scale=0.4]{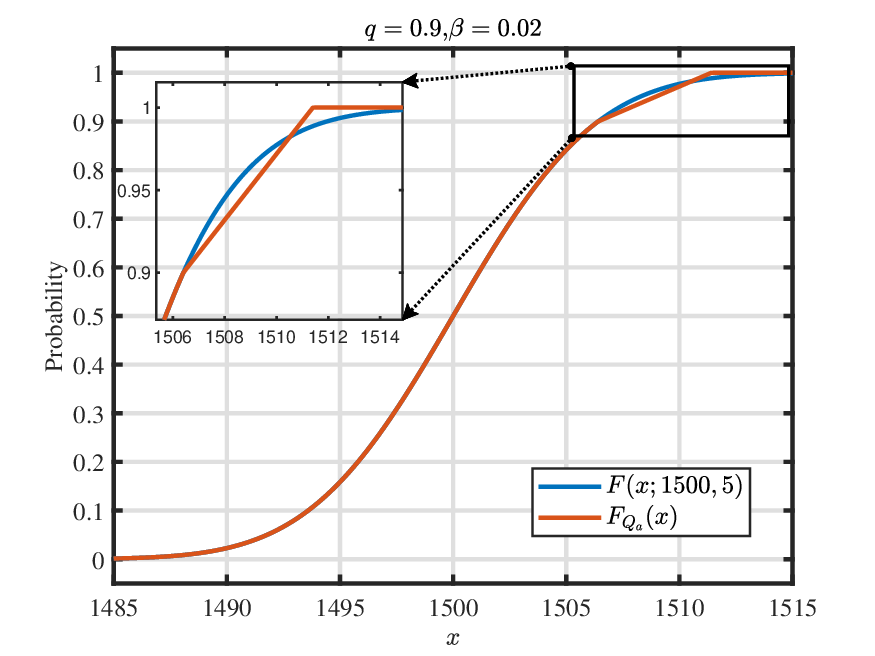}}
  \subfloat[]{
    \label{fig:CDFs-500}
    \includegraphics[scale=0.4]{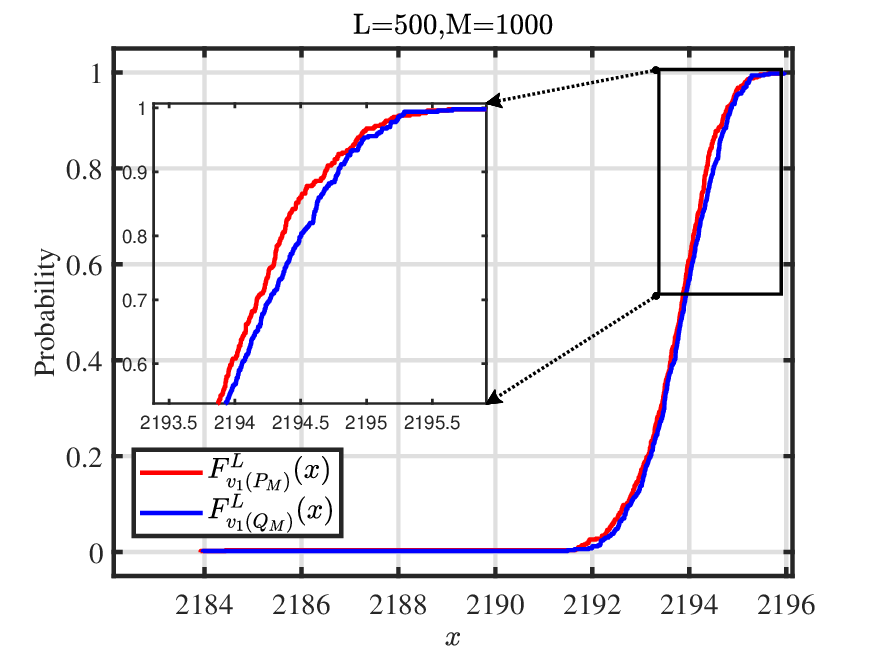}}
  \caption{(Color online) (a) $F_{Q_a}(x)$ and $F(x;1500,5)$. (b) Empirical distributions of $v_{1}(P_M)$ and $v_{1}(Q_M)$ based on $L=500$ simulations with parameter $t=0.01$.
  }
  \label{fig:arcflow}
\end{figure}
\begin{figure}[H]
  \centering
  \subfloat[]{
     \label{fig:LAPUEsolu}
     \includegraphics[scale=0.4]{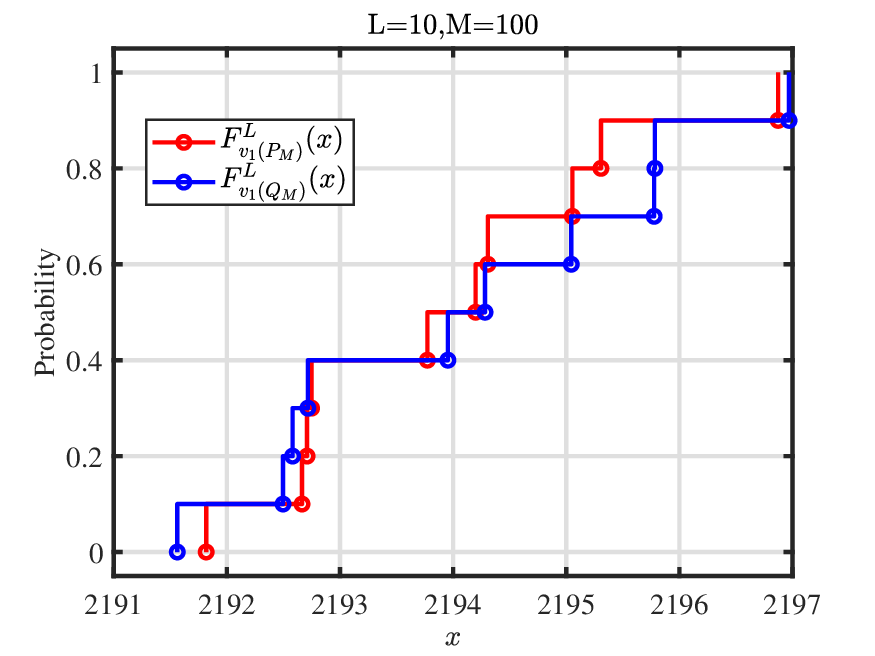}}
  \subfloat[]{
    \label{fig:CDFs}
    \includegraphics[scale=0.4]{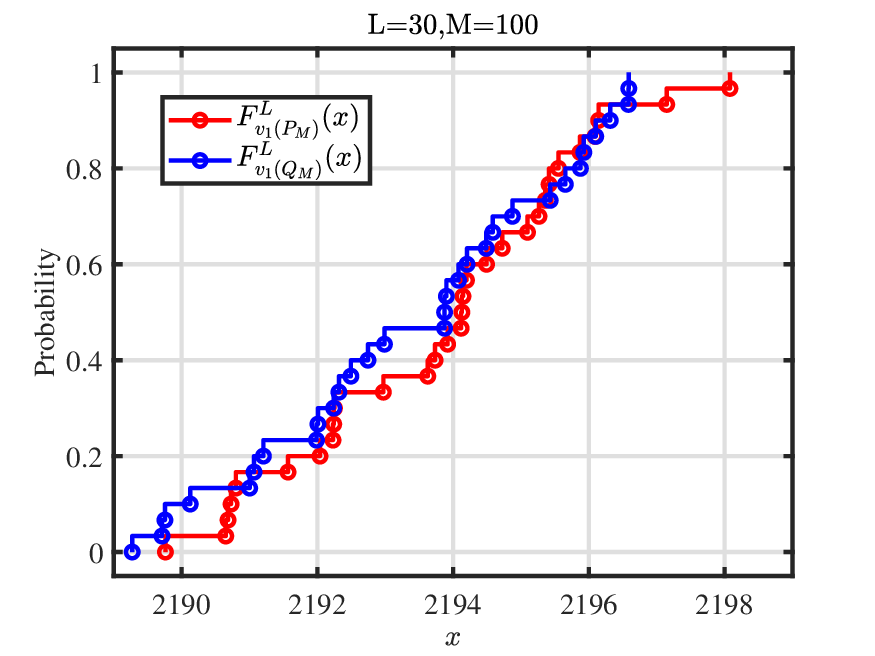}}
  \caption{ (Color online) (a)-(b) Empirical distributions of $v_{1}(P_M)$ and $v_{1}(Q_M)$ based on 
  $L=10$ and $L=30$ simulations with parameter $t=0.01$. 
  }
  \label{fig:arcflow_EMP}
\end{figure}

To examine the difference between the flows  $v_{1}(P_M)$ and $v_1(Q_M)$ at arc $1$,
we use the Kantorovich metric 
to quantify the difference between the CDFs of $v_{1}(P_M)$ and $v_1(Q_M)$ relative to
the Kantorovich metric on the difference between the CDFs of the input data generated by $F_{P_a}$ and $F_{Q_a}(x)$. 
Let $\Delta_2 = \dd_K(P_a,Q_a)$, and 
$$
\Delta_1^L=\dd_K(F_{v_1(P_M)}^L,F_{v_1(Q_M)}^L),
$$
where $F_{v_1(P_M)}^L$ is the empirical distribution of random variable $v_1(P_M)$ using $L$ samples. Figures~\ref{fig:box_ratia_01} and \ref{fig:box_ratia_1} 
display 
convergence
of the ratio $\frac{\Delta^L_1}{\Delta_2}$ as $L$ varies from $20$ to $200$ as $t=0.01$ and $t=0.1$, which implies that the ratio converges when $L$ increases. 
The ratio $\frac{\Delta^L_1}{\Delta_2}$ 
provides some insight for the constant 
at the rhs of \eqref{ieq:SR-solu-1}
in 
Theorem \ref{thm:SR-LAPUE}.

We have also conducted tests to evaluate the 
effect of
parameter $t$ on  MLAPUE.
Figure \ref{fig:Arc1-parameter-t} 
depicts the trend of flows on  path $1$ (which is also the flow of arc $1$)  as $t$ varies from $0.2$ to $2.2$.
The graph demonstrates that the flow exhibits an upward trend as parameter $t$ increases.
Figure \ref{fig:arcflow_EMP-1} displays
change of the influence function of $f_1(\cdot)$ when $t=0.01$ and $t=0.1$. 
We can see that the influence function $\inmat{IF}(\tilde{C}_1,f_1(\cdot),P_1)$ takes negative value. 
This occurs because, for path $1$, the outlier scenario arises when the capacity of arc $1$ is degraded (due to road incidents). 
To minimize user disutility, the flow on path $1$ is reduced accordingly.
\begin{figure}[http]
  \centering
  \subfloat[]{
  \label{fig:box_ratia_01}
  \includegraphics[scale=0.33]{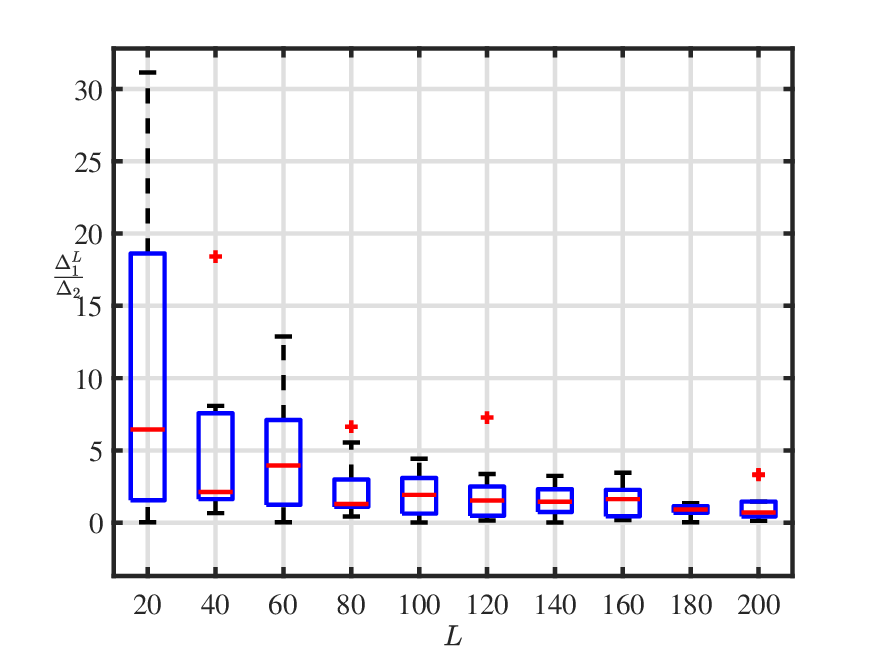}
  }
  \subfloat[]{
     \label{fig:box_ratia_1}
     \includegraphics[scale=0.33]{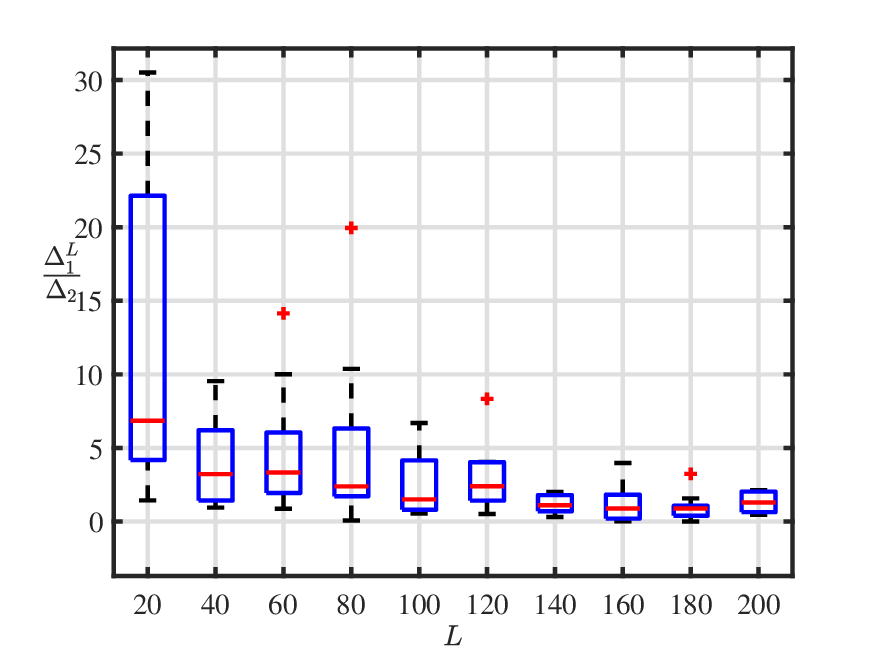}}
  \subfloat[]{
     \label{fig:Arc1-parameter-t}
     \includegraphics[scale=0.33]{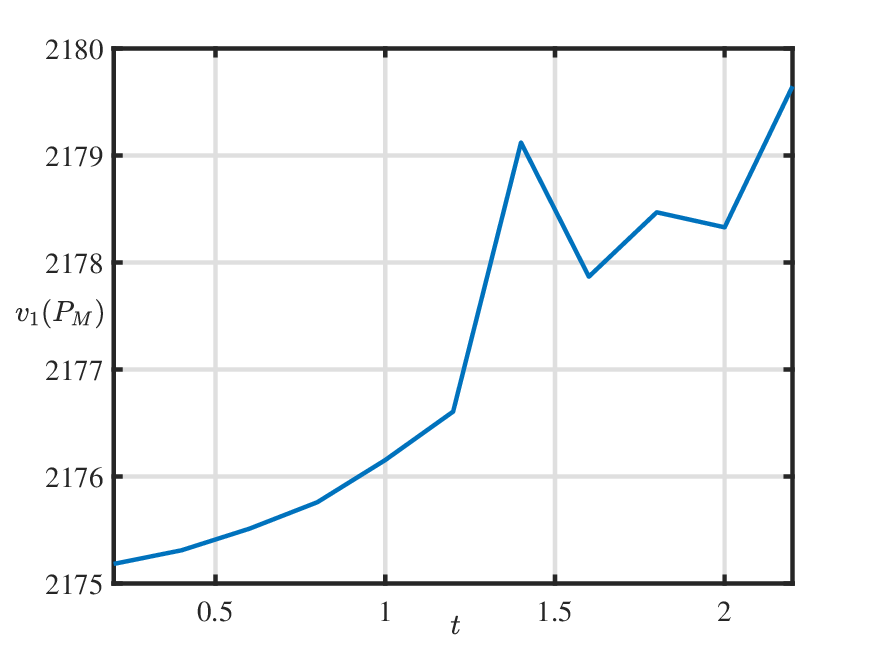}}
  \caption{ (Color online) (a)-(b) $\frac{\Delta_1^L}{\Delta_2}$ by $L$ simulations as $L$ varies from 20 to 200 
  when $t=0.01$ and $t=0.1$. (c) $v_1(P_M)$ when $t$ varies from $0.2$ to $2.2$.  
  }
  \label{fig:arcflow_EMP0}
\end{figure}
\begin{figure}[http]
  \centering
  \subfloat[]{
  \label{fig:LAPUEsolu_inf01}
  \includegraphics[scale=0.4]{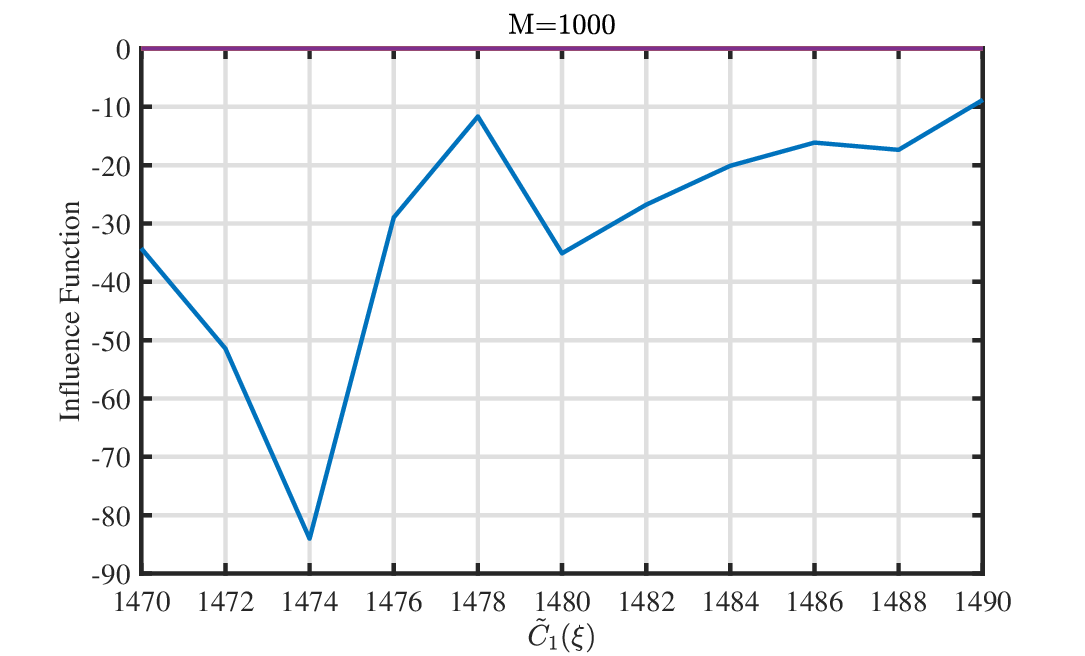}
  }
  \subfloat[]{
     \label{fig:LAPUEsolu_inf1}
     \includegraphics[scale=0.4]{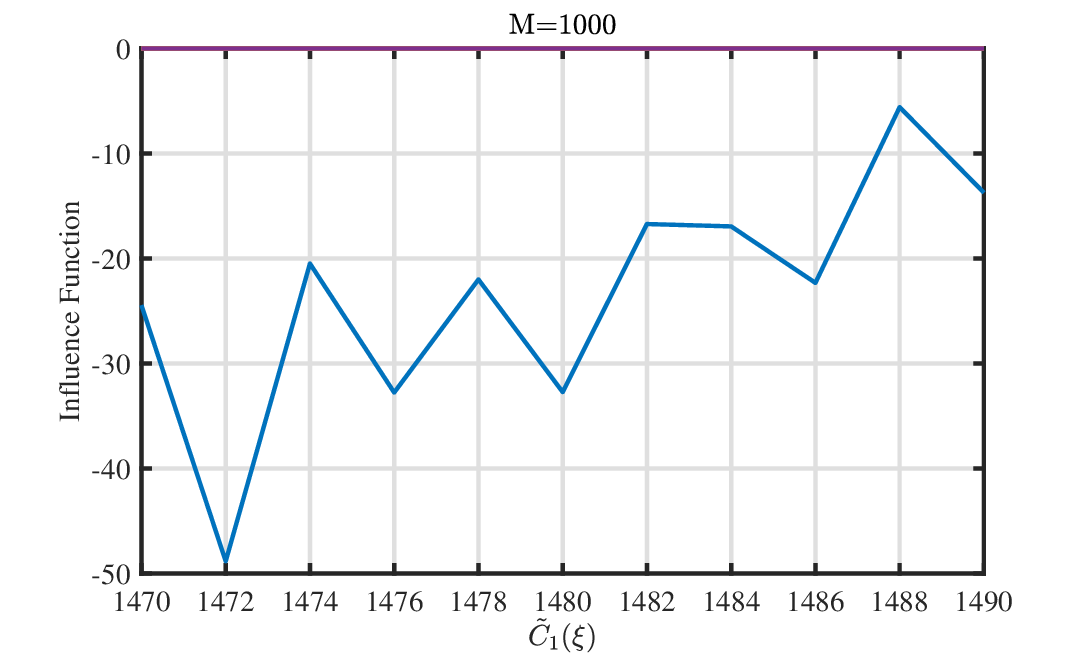}}
  \caption{ (Color online) (a)-(b) The influence function $IF(\tilde{C}_1(\xi), f_1(\cdot), P_1)$ when $\tilde{C}_1$ changes from $1470$ to $1490$ with  parameter $t=0.01$ and $t=0.1$.
  }
  \label{fig:arcflow_EMP-1}
\end{figure} 

Finally, we investigate the case that the breakdown  occurs on arc $1$ by replacing $m$ of the original 
 $M$ samples drown from the true normal distribution with $m$ 
 outliers. We calculate the MLAPUE and the corresponding minimum OD disutility for $10$ cases where $m$ takes values of $ 10, 20,\cdots, 100$ and $M$ is $1000$. The results are depicted in
 Figures~\ref{fig:LAPUEsolu-ODdisu}, \ref{fig:ODpairs-shift-pure} and  \ref{fig:pathFlow-shift-pure}. 
 Figure \ref{fig:LAPUEsolu-BR} shows that 
 as  the number of outliers increases,
 path flow 
 on arc $1$ (path 1) 
 decreases gradually at the MLAPUE, 
 whereas the 
 flow 
 on the other path 
 under the same 1-2 OD pair increases gradually. 
 Since the OD pair 1-2 and the OD pair 1-3 
 share arc $3$, the MLAPUE path flows over the OD pair 1-3 changes accordingly: 
 the flow over path 3
 gradually decreases, whereas
 the 
 flow on the other path gradually increases.
 Moreover, Figure \ref{fig:ODdisu} describes that the minimum OD disutility of each OD pair increases with increase of the number of outliers. 
 Furthermore, 
 Figures \ref{fig:ODpairs-shift-pure} and  \ref{fig:pathFlow-shift-pure} manifest the facts
 that the MLAPUE solutions and the corresponding minimum OD disutilities 
 deviate gradually  from those under the true normal distribution as the number of outliers increases. 

\begin{figure}[http]
  \centering
  \subfloat[]{
     \label{fig:LAPUEsolu-BR}
     \includegraphics[scale=0.4]{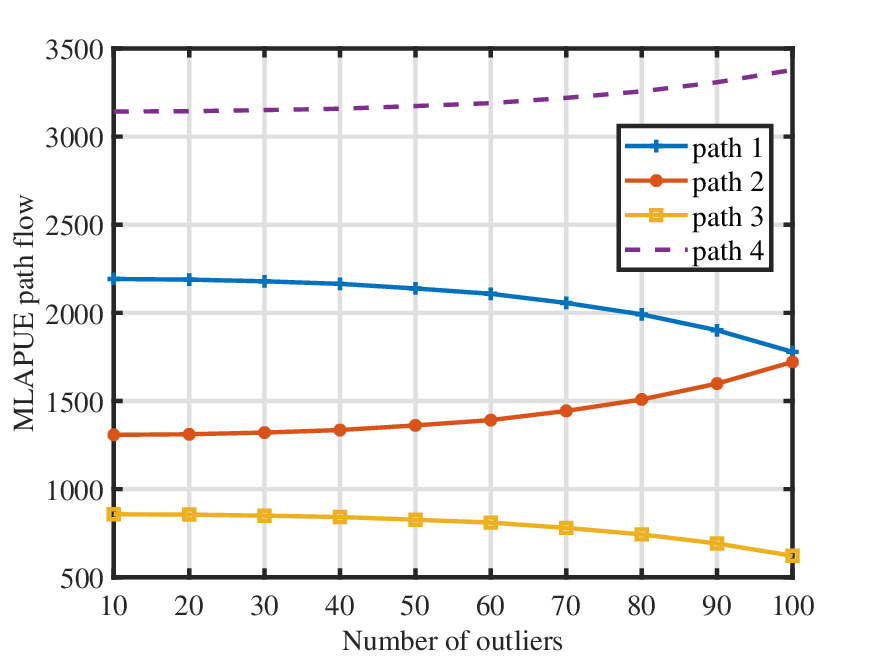}}
  \subfloat[]{
    \label{fig:ODdisu}
    \includegraphics[scale=0.4]{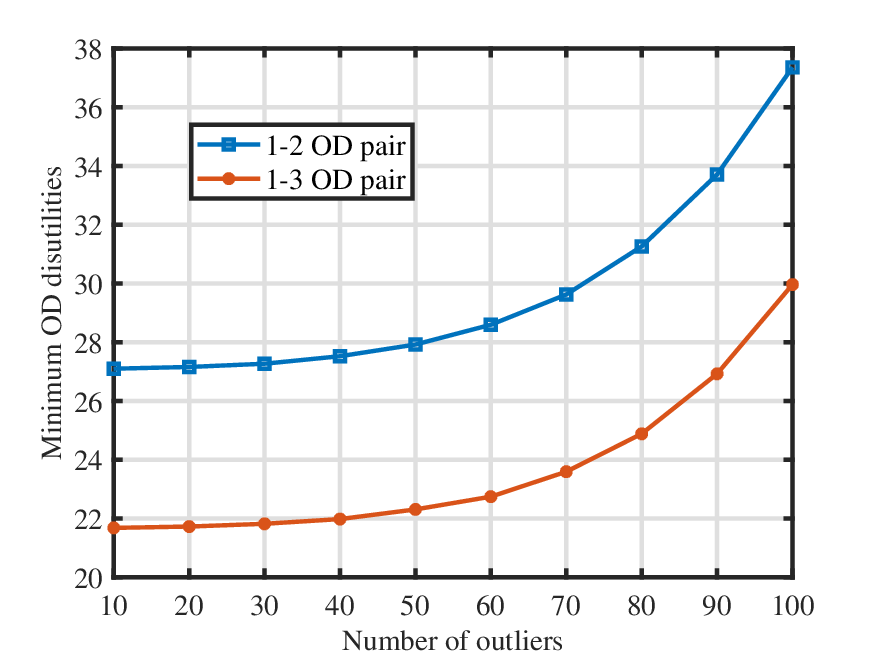}}
  \caption{ (Color online) MLAPUE  path flows and minimum OD disutilities for network $1$ 
  as the number of outliers  varies from $10$ to $100$.
  }
  \label{fig:LAPUEsolu-ODdisu}
\end{figure}
\begin{figure}[H]
  \centering
  \subfloat[]{
     \label{fig:Numbet-OPT}
     \includegraphics[scale=0.4]{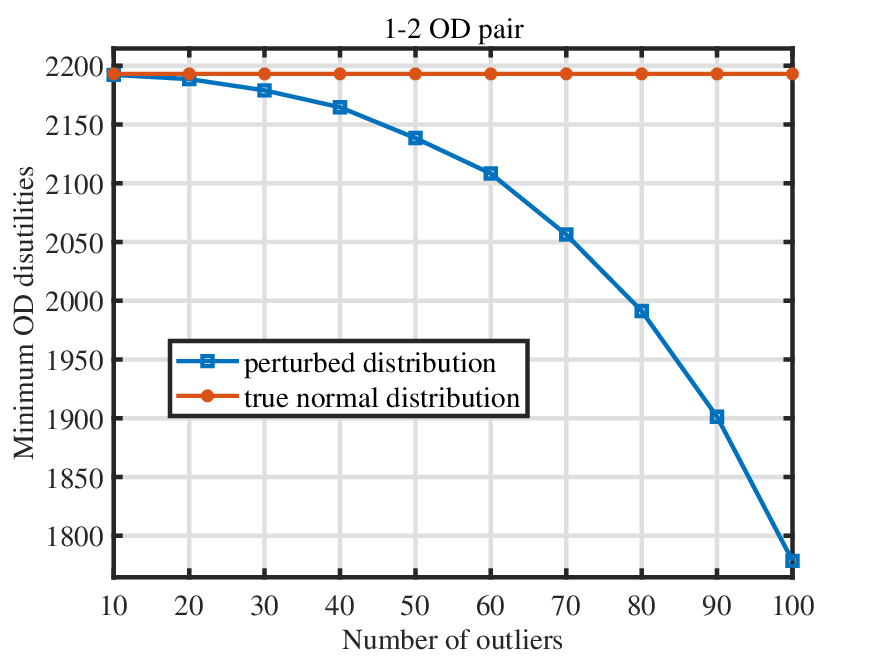}}
  \subfloat[]{
    \label{fig:Numbet-PW}
    \includegraphics[scale=0.4]{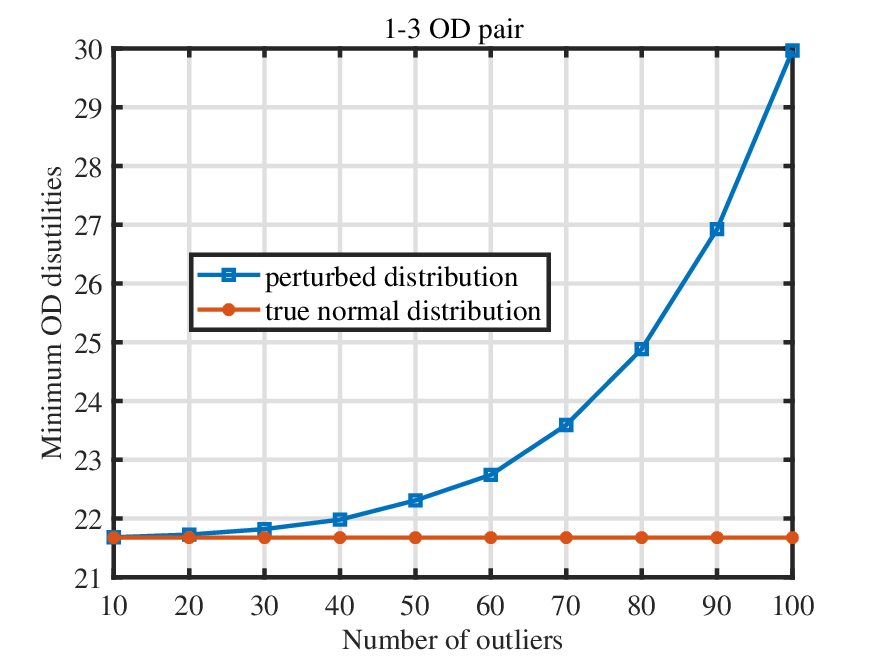}}
  \caption{(Color online) Impact of perturbed distribution of arc travel time for arc $1$ on minimum OD disutilities for network 1.
  }
  \label{fig:ODpairs-shift-pure}
\end{figure}
\begin{figure}[http]
  \centering
  \subfloat[]{
     \label{fig:path1}
     \includegraphics[scale=0.45]{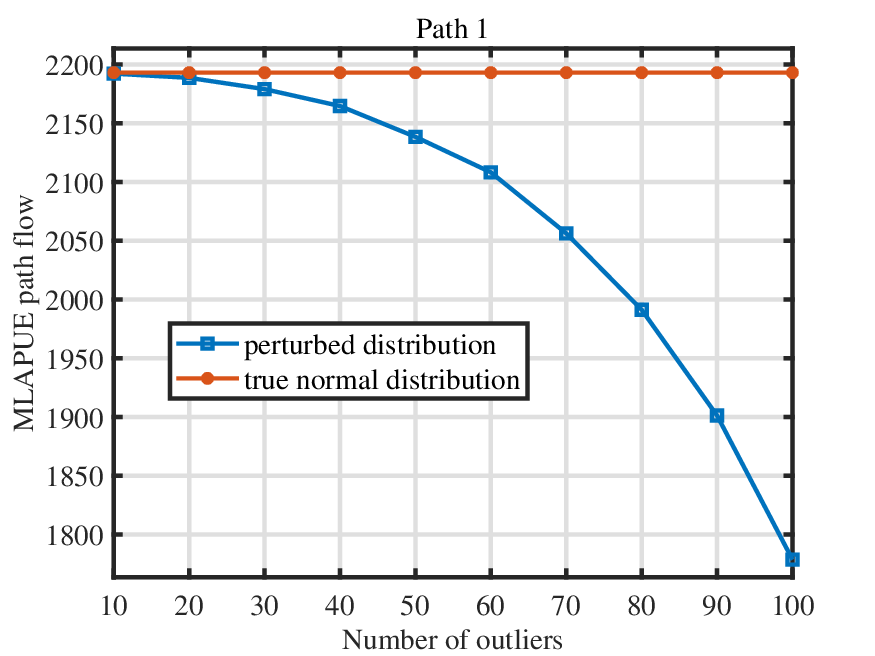}}
     \hfill
  \subfloat[]{
    \label{fig:path2}
    \includegraphics[scale=0.45]{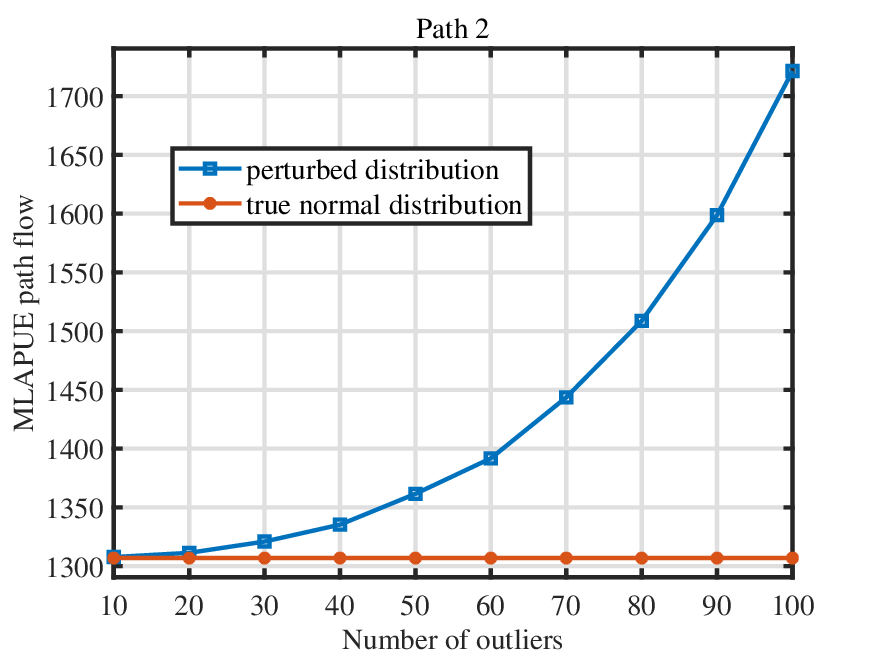}}
    \hfill
 \subfloat[]{
    \label{fig:path3}
    \includegraphics[scale=0.45]{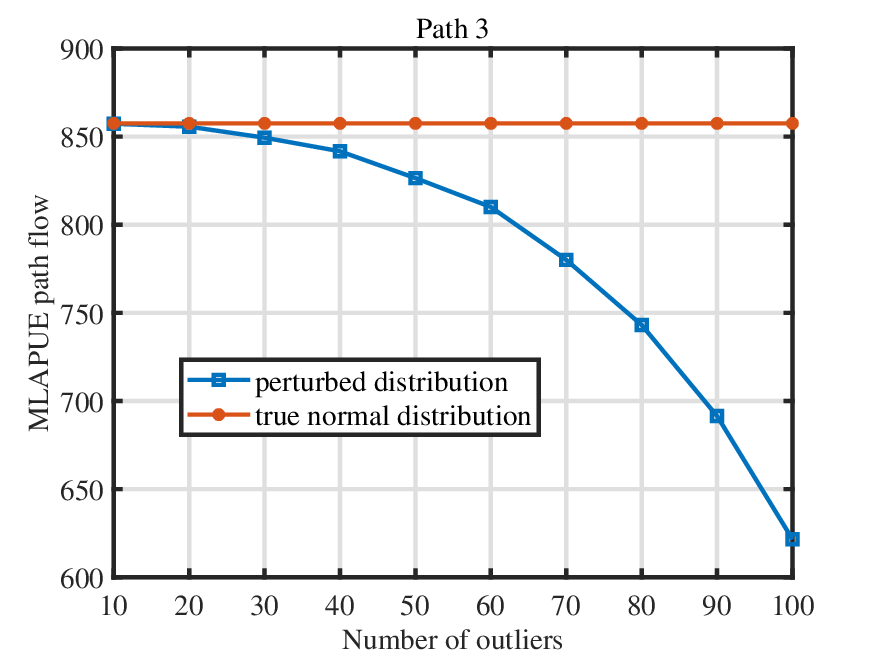}}
    \hfill
 \subfloat[]{
    \label{fig:path4}
    \includegraphics[scale=0.45]{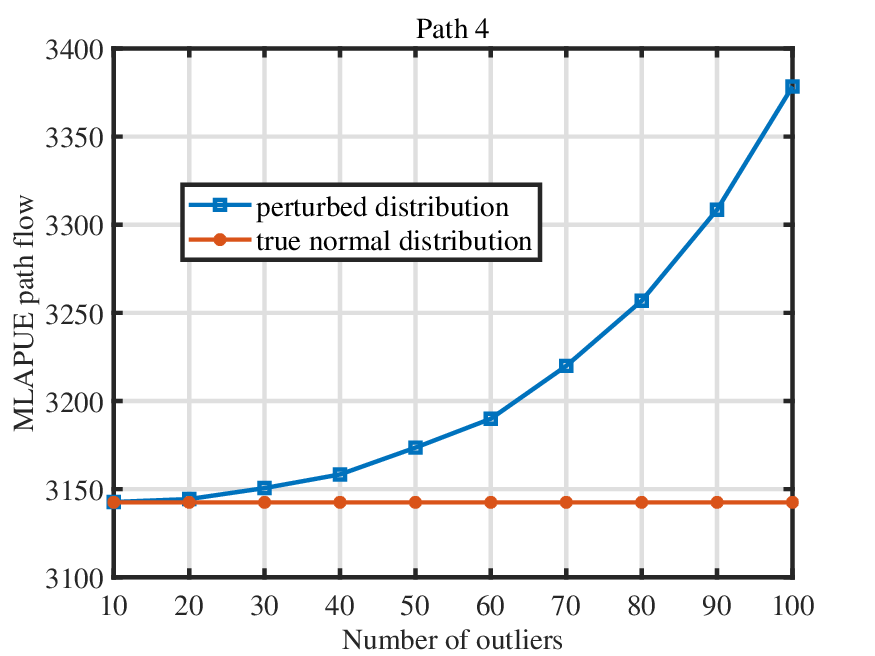}}
  \caption{ (Color online) Impact of perturbed distribution of arc travel time for arc $1$ on MLAPUE solutions for network 1 with parameter $t=0.1$.
  }
  \label{fig:pathFlow-shift-pure}
\end{figure}


\subsection{Example 2: the Nguyen and Dupuis network}

We also consider the Nguyen and Dupuis network shown in Figure \ref{fig:network2}, which includes $13$ nodes, $19$ directed arcs and $4$ OD pairs $1\to2$, $1\to 3$, $4\to 2$ and $4\to 3$. The free-flow arc travel time $t_a^0$ and the mean of arc capacity $\mathbb{E}_P[C_a(\xi)]$ of the network are the same as those used by Yin et al.~\cite[Table 1]{YML09}. The demands of all OD pairs are shown as Figure \ref{table:demand}. 
Note that in this case, the arc-path incidence matrix 
is a $19\times 25$ matrix with column rank $10$. So 
the metric regularity condition is not satisfied and indeed the MLAPUE is not unique.
We record the computational results for traffic assignment pattern $\by{f}_{\inmat{UE}}$, $\by{f}_{\inmat{MLAPUE}}$, $\by{f}_{\inmat{MLAPUE-ptb}}$ and the corresponding minimum OD disutilities in Table \ref{tab:demand-n1}, where $\by{f}_{\inmat{MLAPUE-ptb}}$ is the result for the case that only the distribution of the capacity at arc $a = 2$ is perturbed. From the results, we can see that the flows of the path through arc $2$ under the MLAPUE-ptb pattern either decrease or remains $0$ compared with the flows in the MLAPUE model. 
\begin{figure}[http]
    \centering
    \subfloat[]{
    \label{fig:network2}
    \begin{tikzpicture}[scale=0.9]
\node[circle,
minimum width =3pt,
minimum height =3pt, draw=black] (4) at(0,0){4};
\node[circle,
minimum width =3pt,
minimum height =3pt, draw=black] (5) at(2,0){5};
\node[circle,
minimum width =3pt,
minimum height =3pt, draw=black] (6) at(4,0){6};
\node[circle,
minimum width =3pt,
minimum height =5pt, draw=black] (7) at(6,0){7};
\node[circle,
minimum width =3pt,
minimum height =3pt, draw=black] (8) at(8,0){8};
\node[circle,
minimum width =3pt,
minimum height =3pt, draw=black] (1) at(2,2){1};
\node[circle,
minimum width =3pt,
minimum height =3pt, draw=black] (9) at(2,-2){9};
\node[circle,
minimum width =3pt,
minimum height =3pt, draw=black] (10) at(4,-2){10};
\node[circle,
minimum width =3pt,
minimum height =3pt, draw=black] (12) at(4,2){12};
\node[circle,
minimum width =3pt,
minimum height =3pt, draw=black] (11) at(6,-2){11};
\node[circle,
minimum width =3pt,
minimum height =3pt, draw=black] (2) at(8,-2){2};
\node[circle,
minimum width =3pt,
minimum height =3pt, draw=black] (13) at(4,-4){13};
\node[circle,
minimum width =3pt,
minimum height =3pt, draw=black] (3) at(6,-4){3};
\draw (2,2.7) node [draw=none] {Origin};
\draw (0,0.7) node [draw=none] {Origin};
\draw (8,-2.7) node [draw=none] {Destination};
\draw (6,-4.7) node [draw=none] {Destination};
\draw[arrow] (1) -- node [right] {$1$} (5);
\draw[arrow] (1) -- node [above] {$2$} (12);
\draw[arrow] (4) -- node [above] {$3$} (5);
\draw[arrow] (4) -- node [above] {$4$} (9);
\draw[arrow] (5) -- node [above] {$5$} (6);
\draw[arrow] (5) -- node [right] {$6$} (9);
\draw[arrow] (6) -- node [above] {$7$} (7);
\draw[arrow] (6) -- node [right] {$8$} (10);
\draw[arrow] (7) -- node [above] {$9$} (8);
\draw[arrow] (7) -- node [right] {${10}$} (11);
\draw[arrow] (8) -- node [right] {${11}$} (2);
\draw[arrow] (9) -- node [above] {${12}$} (10);
\draw[arrow] (9) -- node [above] {${13}$} (13);
\draw[arrow] (10) -- node [above] {${14}$} (11);
\draw[arrow] (11) -- node [above] {${15}$} (2);
\draw[arrow] (11) -- node [right] {${16}$} (3);
\draw[arrow] (12) -- node [right] {${17}$} (6);
\draw[arrow] (12) -- node [above] {${18}$} (8);
\draw[arrow] (13) -- node [above] {${18}$} (3);
\end{tikzpicture}
}
\subfloat[]{
\label{table:demand}
\mytab
}
    \caption{Nguyen and Dupuis network and OD pair demands.
}
\label{fig:ND-network}
\end{figure}
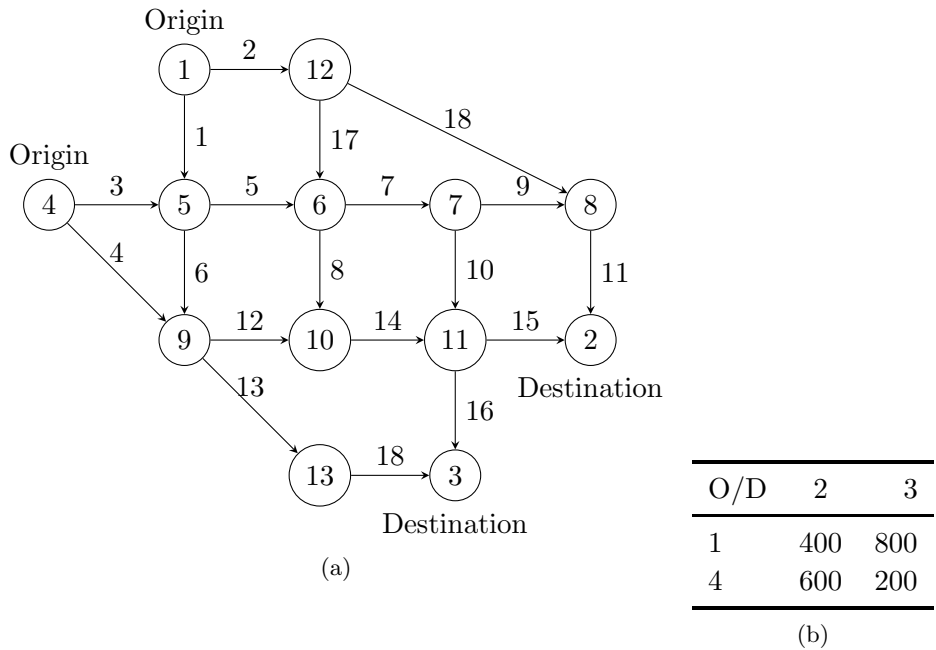

\begin{table}[htbp]
    \centering
    \begin{tabular}{lrrr}
    \toprule
    Path   &  $\by{f}_{\inmat{UE}}$ & $\by{f}_{\inmat{MLAPUE}}$ & $\by{f}_{\inmat{MLAPUE-ptb}}$\\
    \midrule
    $f_1$(2-18-11)          & 140.14   &  206.21 & 181.01\\
    $f_2$(2-17-7-9-11)      & 86.57    &  70.10  & 52.30\\
    $f_3$(2-17-7-10-15)     & 0        &  0      & 
    0\\
    $f_4$(2-17-8-14-15)     & 0.76     &  0      &
    0\\
    $f_5$(1-5-7-9-11)       & 134.22   & 104.54  &
    122.47\\
    $f_6$(1-5-7-10-15)      & 38.31    &  19.15  &
    44.22\\
    $f_7$(1-6-12-14-15)     & 0        &   0     & 
    0\\
    $f_8$(1-5-8-14-15)      &  0       &   0     & 
    0\\
    $f_9$(2-17-7-10-16)     & 89.46    &  127.67 & 92.27\\
    $f_{10}$(2-17-8-14-16)  & 28.85    &   0     & 
    0\\
    $f_{11}$(1-5-7-10-16)   & 134.32   &  157.14 & 154.22\\
    $f_{12}$(1-5-8-14-16)   & 77.96    &  62.14  & 74.59\\
    $f_{13}$(1-6-12-14-16)  & 115.31   &  1.95   & 36.40\\
    $f_{14}$(1-6-13-19)     & 354.10   & 451.10  & 442.52\\
    $f_{15}$(3-5-7-9-11)    & 159.85   & 116.22  & 121.03\\
    $f_{16}$(3-5-7-10-15)   & 65.36    &  44.03  & 54.06\\
    $f_{17}$(3-5-8-14-15)   & 0        &  4.31   & 
    0\\
    $f_{18}$(3-6-12-14-15)  & 40.86    &   0     &   0\\
    $f_{19}$(4-12-14-15)    & 333.93   &  435.44 & 424.91\\
    $f_{20}$(3-5-7-10-16)   & 0        &  0      & 
    0.16\\
    $f_{21}$(3-5-8-14-16)   & 0        &  0      & 
    0.46\\
    $f_{22}$(3-6-12-14-16)  & 0        &  0      &  0.56\\
    $f_{23}$(3-6-13-19)     & 0        &  0      &  0.59\\
    $f_{24}$(4-12-14-16)    & 0        &  0      & 
    0\\
    $f_{25}$(4-13-19)       & 200.00   & 200.00  &  198.05\\
    \midrule
    Minimum OD disutilities &  & &\\
    \midrule
    OD $1\to2$ &  39.52  &  37.21  & 40.90\\
    OD $1\to3$ &  47.84  &  48.41  & 53.56\\
    OD $4\to2$ &  45.48  &  39.68  & 42.35\\
    OD $4\to3$ &  34.38  &  39.90  & 35.80\\
    \bottomrule
    \end{tabular}
    \caption{Optimal path flow and  minimum OD disutilities}
    \label{tab:demand-n1}
\end{table}

\section{Concluding remarks}
\label{sec5:Con}
In this paper, we propose a modification of 
the well-known LAPUE model by adopting a new parameterized penalty function for user's lateness. 
Under the same conditions as for the LAPUE model, we demonstrate existence and uniqueness of LAPUE under the new penalty function and show that it approximates Watling's LAPUE when the parameter is driven to zero.
While the new penalty function offers flexibility to describe a user's preference of riskiness in lateness, 
it has a main advantage that the resulting LAPUE based on perturbed data is statistically robust which means 
that they are not very sensitive to the data perturbation. The result may provide a mathematical framework for 
analysing LAPUE in data driven problems.
As far as we are concerned, this kind of analysis is new in the literature of transportation user equilibrium.
We are unable to show whether the original LAPUE based on the max penalty function for user's lateness is statistically robust when all sample data are potentially perturbed because the error bound in (\ref{ieq:SR-solu}) goes to infinity when $t$ is driven to zero. It does not necessarily mean that the LAPUE is not quantitatively statistically robust, rather it is because our current mathematical framework of analysis does not enable us to establish the result. We leave this interested readers to explore.  

Another important point that we would like to bring readers to attend is that differing from the main stream research in this area, the sample average approximation problem in this paper is based on
the samples of the underlying uncertainty factors 
behind random travel time rather than samples of travel time. This is because travel time depends not only the uncertainty factors but also volume of arc flow/path flow which are decision variables in the model. We will not be able to undertake 
the statistical robustness analysis if we mix up the decision variables and the uncertainty parameters in the empirical data.
To address this discrepancy between the theoretical model (SAA) and the practical availability of the sample data, we may use regression models to derive 
samples of uncertain parameters from samples of travel time via (\ref{func:GBPR}).

\section*{CRediT authorship contribution statement}

$\textbf{Manlan Li}$: Writing- Original draft preparation, Validation, Investigation,  Writing-Reviewing and Editing.
$\textbf{Huifu Xu}$: Supervision, Methodology, Writing - Original Draft, Writing-Reviewing and Editing, Supervision, Funding acquisition.

\section*{Declaration of Competing Interest}
The authors declare no conflict of interest regarding the publication of this paper.

\section*{Acknowledgments}
This work is supported by 
GRC Grant (14204821) and 
CUHK start-up grant.

\newpage

\bibliographystyle{unsrt}
\bibliography{reference}

\begin{thebibliography}{10}

\bibitem{Watling06}
D.~Watling.
\newblock User equilibrium traffic network assignment with stochastic travel
  times and late arrival penalty.
\newblock {\em European Journal of Operational Research}, 175(3):1539--1556,
  2006.

\bibitem{Wardrop52}
J.~Wardrop.
\newblock Some theoretical aspects of road traffic research.
\newblock {\em Proceedings of the Institution of Civil Engineers},
  1(3):325--362, 1952.

\bibitem{BMW56}
M.~Beckmann, C.~McGuire, and C.~Winsten.
\newblock {\em Studies in the Economics of Transportation}.
\newblock New Haven, CT, USA: Yale Univ. Press, 1956.

\bibitem{Smith79}
M.~Smith.
\newblock The existence, uniqueness and stability of traffic equilibria.
\newblock {\em Transportation Research Part B: Methodological}, 13(4):295--304,
  1979.

\bibitem{Dafermos80}
S.~Dafermos.
\newblock Traffic equilibrium and variational inequalities.
\newblock {\em Transportation Science}, 14(1):42--54, 1980.

\bibitem{MXMC20}
J.~Ma, M.~Xu, Q.~Meng, and L.~Cheng.
\newblock Ridesharing user equilibrium problem under od-based surge pricing
  strategy.
\newblock {\em Transportation Research Part B: Methodological}, 134:1--24,
  2020.

\bibitem{MMCL22}
J.~Ma, Q.~Meng, L.~Cheng, and Z.~Liu.
\newblock General stochastic ridesharing user equilibrium problem with elastic
  demand.
\newblock {\em Transportation Research Part B: Methodological}, 162:162--194,
  2022.

\bibitem{PZDG22}
J.~Pang, M.~Zhang, M.~Dessouky, W.~Gu, Pacific~Southwest Region,
  METRANS~Transportation Center, et~al.
\newblock Modeling multi-modal mobility in a coupled morning-evening commute
  framework that considers deadheading and flexible pooling.
\newblock Technical report, METRANS Transportation Center (Calif.), 2022.

\bibitem{DaganzoY77}
C.~Daganzo and Y.~Sheffi.
\newblock On stochastic models of traffic assignment.
\newblock {\em Transportation Science}, 11(3):253--274, 1977.

\bibitem{LouYL10}
Y.~Lou, Y.~Yin, and S.~Lawphongpanich.
\newblock Robust congestion pricing under boundedly rational user equilibrium.
\newblock {\em Transportation Research Part B: Methodological}, 44(1):15--28,
  2010.

\bibitem{PSS18}
A.~Prakash, R.~Seshadri, and K.~Srinivasan.
\newblock A consistent reliability-based user-equilibrium problem with
  risk-averse users and endogenous travel time correlations: Formulation and
  solution algorithm.
\newblock {\em Transportation Research Part B: Methodological}, 114:171--198,
  2018.

\bibitem{CarrionL12}
C.~Carrion and D.~Levinson.
\newblock Value of travel time reliability: A review of current evidence.
\newblock {\em Transportation Research Part A: Policy and Practice},
  46(4):720--741, 2012.

\bibitem{Watling02}
D.~Watling.
\newblock A second order stochastic network equilibrium model, {I}: Theoretical
  foundation.
\newblock {\em Transportation Science}, 36(2):149--166, 2002.

\bibitem{NakayamaW14}
S.~Nakayama and D.~Watling.
\newblock Consistent formulation of network equilibrium with stochastic flows.
\newblock {\em Transportation Research Part B: Methodological}, 66:50--69,
  2014.

\bibitem{XLYZ11}
H.~Xu, Y.~Lou, Y.~Yin, and J.~Zhou.
\newblock A prospect-based user equilibrium model with endogenous reference
  points and its application in congestion pricing.
\newblock {\em Transportation Research Part B: Methodological}, 45(2):311--328,
  2011.

\bibitem{MirchandaniS87}
P.~Mirchandani and H.~Soroush.
\newblock Generalized traffic equilibrium with probabilistic travel times and
  perceptions.
\newblock {\em Transportation Science}, 21(3):133--152, 1987.

\bibitem{LLS06}
H.~Lo, X.~Luo, and B.~Siu.
\newblock Degradable transport network: travel time budget of travelers with
  heterogeneous risk aversion.
\newblock {\em Transportation Research Part B: Methodological}, 40(9):792--806,
  2006.

\bibitem{ZhouC08}
Z.~Zhou and A.~Chen.
\newblock Comparative analysis of three user equilibrium models under
  stochastic demand.
\newblock {\em Journal of Advanced Transportation}, 42(3):239--263, 2008.

\bibitem{ConnorsS09}
R.~Connors and A.~Sumalee.
\newblock A network equilibrium model with travellers’ perception of
  stochastic travel times.
\newblock {\em Transportation Research Part B: Methodological}, 43(6):614--624,
  2009.

\bibitem{NikolovaS14}
E.~Nikolova and N.~Stier-Moses.
\newblock A mean-risk model for the traffic assignment problem with stochastic
  travel times.
\newblock {\em Operations Research}, 62(2):366--382, 2014.

\bibitem{ChenZ10}
A.~Chen and Z.~Zhou.
\newblock The $\alpha$-reliable mean-excess traffic equilibrium model with
  stochastic travel times.
\newblock {\em Transportation Research Part B: Methodological}, 44(4):493--513,
  2010.

\bibitem{YFW18}
Y.~Yang, Y.~Fan, and R.~Wets.
\newblock Stochastic travel demand estimation: Improving network
  identifiability using multi-day observation sets.
\newblock {\em Transportation Research Part B: Methodological}, 107:192--211,
  2018.

\bibitem{SiuL08}
B.~Siu and H.~Lo.
\newblock Doubly uncertain transportation network: Degradable capacity and
  stochastic demand.
\newblock {\em European Journal of Operational Research}, 191(1):166--181,
  2008.

\bibitem{ZCS11}
C.~Zhang, X.~Chen, and A.~Sumalee.
\newblock Robust wardrop’s user equilibrium assignment under stochastic
  demand and supply: expected residual minimization approach.
\newblock {\em Transportation Research Part B: Methodological}, 45(3):534--552,
  2011.

\bibitem{ZXQCC22}
Z.~Zang, X.~Xu, K.~Qu, R.~Chen, and A.~Chen.
\newblock Travel time reliability in transportation networks: A review of
  methodological developments.
\newblock {\em Transportation Research Part C: Emerging Technologies},
  143:103866, 2022.

\bibitem{Hampel68}
F.~Hampel.
\newblock Contributions to the theory of robust estimation.
\newblock {\em University of California, Berkeley, Calif, USA}, 1968.

\bibitem{Huber04}
P.~Huber.
\newblock {\em Robust statistics}, volume 523.
\newblock John Wiley \& Sons, 2004.

\bibitem{cont10}
R.~Cont, R.~Deguest, and G.~Scandolo.
\newblock Robustness and sensitivity analysis of risk measurement procedures.
\newblock {\em Quantitative Finance}, 10(6):593--606, 2010.

\bibitem{KSZ12}
V.~Kr{\"a}tschmer, A.~Schied, and H.~Z{\"a}hle.
\newblock Qualitative and infinitesimal robustness of tail-dependent
  statistical functionals.
\newblock {\em Journal of Multivariate Analysis}, 103(1):35--47, 2012.

\bibitem{KSZ14}
V.~Kr{\"a}tschmer, A.~Schied, and H.~Z{\"a}hle.
\newblock Comparative and qualitative robustness for law-invariant risk
  measures.
\newblock {\em Finance and Stochastics}, 18:271--295, 2014.

\bibitem{GuoXu21a}
S.~Guo and H.~Xu.
\newblock Statistical robustness in utility preference robust optimization
  models.
\newblock {\em Mathematical Programming}, 190(1-2):679--720, 2021.

\bibitem{LiuPang23}
J.~Liu and J.~Pang.
\newblock Risk-based robust statistical learning by stochastic
  difference-of-convex value-function optimization.
\newblock {\em Operations Research}, 71(2):397--414, 2023.

\bibitem{YML09}
Y.~Yin, S.~Madanat, and X.~Lu.
\newblock Robust improvement schemes for road networks under demand
  uncertainty.
\newblock {\em European Journal of Operational Research}, 198(2):470--479,
  2009.

\bibitem{ZZZY19}
Z.~Zhu, S.~Zhu, Z.~Zheng, and H.~Yang.
\newblock A generalized bayesian traffic model.
\newblock {\em Transportation Research Part C: Emerging Technologies},
  108:182--206, 2019.

\bibitem{Robinson76}
S.~Robinson.
\newblock Regularity and stability for convex multivalued functions.
\newblock {\em Mathematics of Operations Research}, 1(2):130--143, 1976.

\bibitem{Shapiro03}
A.~Shapiro.
\newblock Monte carlo sampling methods.
\newblock {\em Handbooks in Operations Research and Management Science},
  10:353--425, 2003.

\bibitem{Robinson92}
S.~Robinson.
\newblock Normal maps induced by linear transformations.
\newblock {\em Mathematics of Operations Research}, 17(3):691--714, 1992.

\bibitem{GYR99}
G.~G{\"u}rkan, A.~Yonca~{\"O}zge, and S.~Robinson.
\newblock Sample-path solution of stochastic variational inequalities.
\newblock {\em Mathematical Programming}, 84(2), 1999.

\bibitem{GuoX23}
S.~Guo and H.~Xu.
\newblock Data perturbations in stochastic generalized equations: statistical
  robustness in static and sample average approximated models.
\newblock {\em Mathematical Programming}, pages 1--34, 2023.

\bibitem{FerrisP97}
M.~Ferris and J.~Pang.
\newblock Engineering and economic applications of complementarity problems.
\newblock {\em SIAM Review}, 39(4):669--713, 1997.

\bibitem{zhu2019}
Z.~Zhu, S.~Zhu, Z.~Zheng, and H.~Yang.
\newblock A generalized bayesian traffic model.
\newblock {\em Transportation Research Part C: Emerging Technologies},
  108:182--206, 2019.

\bibitem{RWVa09}
R.~Rockafellar and R.~Wets.
\newblock {\em Variational analysis}, volume 317.
\newblock Springer Science \& Business Media, 2009.

\bibitem{Rockafellar89}
R.~Rockafellar.
\newblock Proto-differentiability of set-valued mappings and its applications
  in optimization.
\newblock In {\em Annales de l'Institut Henri Poincar{\'e} C, Analyse non
  lin{\'e}aire}, volume~6, pages 449--482. Elsevier, 1989.

\bibitem{Rachev91}
S.~Rachev.
\newblock {\em Probability metrics and the stability of stochastic models}.
\newblock Wiley, 1991.

\bibitem{gibbs2002choosing}
A.~Gibbs and F.~Su.
\newblock On choosing and bounding probability metrics.
\newblock {\em International Statistical Review}, 70(3):419--435, 2002.

\bibitem{Romisch03}
W.~R{\"o}misch.
\newblock Stability of stochastic programming problems.
\newblock {\em Handbooks in Operations Research and Management Science},
  10:483--554, 2003.

\bibitem{LevyR94}
A.~Levy and R.~Rockafellar.
\newblock Sensitivity analysis of solutions to generalized equations.
\newblock {\em Transactions of the American Mathematical Society},
  345(2):661--671, 1994.

\bibitem{Claus16}
M.~Claus.
\newblock {\em Advancing stability analysis of mean-risk stochastic programs:
  Bilevel and two-stage models}.
\newblock PhD thesis, Dissertation, Duisburg, Essen, Universit{\"a}t
  Duisburg-Essen, 2016, 2016.

\bibitem{Prokhorov56}
Y.~Prokhorov.
\newblock Convergence of random processes and limit theorems in probability
  theory.
\newblock {\em Theory of Probability \& Its Applications}, 1(2):157--214, 1956.

\end{thebibliography}

\appendix

\section*{Appendix}
\label{Append}

\section{Perliminaries}
\begin{definition}[Outer limit \cite{RWVa09}]
\label{def:outerlim}
    For a sequence $\by{\mathcal{F}}_t$ of subsets of
$\R^N$, the outer limit is the set
\bgeqn
\limsup_{t\downarrow 0}\by{\mathcal{F}}_t:=\left\{\by{f}\left|\exists t_k\to 0,\; \exists \by{f}_{t_k}\in \by{\mathcal{F}}_{t_k}\; \inmat{with} \; \by{f}_{t_k}\to\by{f}\right.\right\}.
\edeqn
\end{definition}

\begin{definition}[Proto-derivative \cite{Rockafellar89}]
\label{def:proto-d}
A set-valued mapping $\Gamma:\R^N\rightrightarrows\R^N$ is proto-differentiable at a point $z$ and for a particular element $u\in\Gamma(z)$ if the set-valued mappings
\bgeq
\Delta_{z,u,\varepsilon}:d\to\frac{\Gamma(z+\varepsilon d)-u}{\varepsilon}
\edeq
\end{definition}
regarded as a family indexed by $\varepsilon>0$, graph-converge as $\varepsilon\downarrow 0$. The limit, if exists, is denoted by $\Gamma'_{z,u}$ and called the {\it proto-derivative} of $\Gamma$ at $z$ for $u$. 

In the case that the proto-derivative exists, 
\bgeq
\Gamma'_{z,u}(d)=\limsup_{\varepsilon\downarrow 0,d'\to d}\frac{\Gamma(z+\varepsilon d')-u}{\varepsilon}
\edeq
for every $d\in\R^N$ via \cite[Proposition 2.3]{Rockafellar89}.
The graphical outer limit of the set-valued mappings 
$
\{\Delta_{z,u,\varepsilon(\cdot):\varepsilon>0}\}$ 
is called {\it graphical derivative}, denoted by $D\Gamma(z|u)$, this is, 
\bgeq
\inmat{ghp}D\Gamma(z|u):=\limsup_{\varepsilon\downarrow 0}{\inmat{ghp}\Delta_{z,u,\varepsilon}}=\limsup_{\varepsilon\downarrow 0}{\frac{\inmat{ghp}\Gamma -(z,u)}{\varepsilon}}.
\edeq
The graphical derivative always exists and its graph is the tangent cone of $\inmat{gph}\Gamma$ at point $(z,u)$. In the case the proto-derivative exists, the proto-derivative and the graphical derivative are equal.

\begin{definition}[Metric regularity]
\label{def:MR}
Let $\Upsilon:\R^N\rightrightarrows\R^N$ be a closed set-valued mapping.
For $\bar{\by{f}},\bar{\by{g}}\in\R^N$, $\Upsilon$ is said to be metrically regular at $\bar{\by{f}}$ for $\bar{\by{g}}$ if there exist a
constant $\alpha> 0$, neighborhoods $\mathcal{V}_{\bar{\by{f}}}$ and $\mathcal{V}_{\bar{\by{g}}}$ of $\bar{\by{f}}$ and  $\bar{\by{g}}$ such that
$$
d(\by{f},\Upsilon^{-1}(\by{g}))\leq\alpha d(\by{g},\Upsilon(\by{f})), \;\;\forall \by{f}\in\mathcal{V}_{\bar{\by{f}}}, \by{g}\in\mathcal{V}_{\bar{\by{g}}}.
$$
Here the inverse mapping is defined as $\Upsilon^{-1}(\by{g}):=\{\by{f}:\by{g}\in\Upsilon(\by{f})\}$. $\Upsilon$ is said to be strongly metrically regular at $\bar{\by{f}}$ for $\bar{\by{g}}$ if it is metrically regular
and there exist neighborhoods $\mathcal{V}_{\bar{\by{f}}}$ and $\mathcal{V}_{\bar{\by{g}}}$ such that for  $\by{g}\in\mathcal{V}_{\bar{\by{g}}}$ there is only one
 $\by{f}\in \mathcal{V}_{\bar{\by{f}}} \cap \Upsilon^{-1}(\by{g})$.
\end{definition}

\begin{definition}[Fortet-Mourier metric] Let 
\begin{equation}
\label{dis-FM}
\mathcal{F}_p(\R^k):=\left\{\phi:\R^k\to\R:|\phi(\xi)-\phi(\tilde{\xi})|\leq L_p(\xi,\tilde{\xi})\|\xi-\tilde{\xi}\|, \forall \xi,\tilde{\xi}\in\R^k\right\}
\end{equation}
where $L_p(\xi,\tilde{\xi}):=\max\{1,\|\xi\|,\|\tilde{\xi}\|\}^{p-1}$ for $p\geq 1$. The $p$th order Fortet-Mourier metric over $\mathscr{P}(\R^k)$ is defined by
\begin{equation}
\label{dist-FM}
    \dd_{\rm FM}(P,Q):=\sup_{\phi\in\mathcal{F}_p(\R^k)}\left|\int_{\R^k}\phi(\xi)P(d\xi)-\int_{\R^k}\phi(\xi)Q(d\xi)\right|.
    \end{equation}
\end{definition}
It is well-known that the Fortet-Mourier distance metricizes weak convergence on $\mathcal{M}^p_k$, see \cite[Theorem 6.2.1]{Rachev91}. In the case when $p=1$, it reduces to the {\em Kantorovich metric}
\begin{equation}
\label{dist-Kan}
    \dd_{\rm K}(P,Q):=\sup_{\phi\in\mathcal{F}_1(\R^k)}\left|\int_{\R^k}\phi(\xi)P(d\xi)-\int_{\R^k}\phi(\xi)Q(d\xi)\right|,
    \end{equation}
see also \cite{gibbs2002choosing}.

Let $\mathbb{F}_t$ be the set defined as in Assumption \ref{asm:compact-MSVIP}, $\tilde{\mathscr{G}}:=\{\by{g}(\cdot)=\hat{\by{u}}(\by{f},\by{C}(\by{f},\cdot),t):\by{f}\in\mathbb{F}_t,t>0\}$. For $P,Q\in\hat{\mathscr{P}}$, define
\bgeqn
\label{psm-u}
\mathscr{D}(P,Q):=\sup_{\by{g}\in\mathscr{G}}\|\mathbb{E}_P[\by{g}(\xi)]-\mathbb{E}_Q[\by{g}(\xi)]\|.
\edeqn
According to Proposition~\ref{prop:property-of-u}~(i), $\mathscr{D}(P,Q)$ is well-defined for any $P,Q\in\hat{\mathscr{P}}$. Then 
 $\mathscr{D}(P,Q)$
is a kind of pseudo-metric in the space of probability measures $\hat{\mathscr{P}}$ in the sense that
$\mathscr{D}(P,Q)=0$ implies $\mathbb{E}_P[\by{g}(\xi)]-\mathbb{E}_Q[\by{g}(\xi)]$ for all $\by{g}\in\mathscr{G}$ but not necessarily
$Q = P$. Note that this definition is slightly different from the classical definition of
pseudo-metric in the literature \cite{Romisch03} where $\by{g}$ is often restricted to real-valued functions. 



\begin{lemma}
\label{lem:LIp-solu}
Let $V$ be an open neighborhood of $z_0$ and consider the Banach space $U:=C^1(V,W)$ of continuously differentiable mapping $\psi:V\to W$ having finite norm
\bgeqn
\label{def:norm-psi}
\|\psi\|_{1,V}:=\sup_{z\in V}\|\psi(z)\|+\sup_{z\in V}\|D\psi(z)\|.
\edeqn
If $z_0$ is a strongly regular solution of the abstract generalized equation (\ref{eq:abstract_GE-0}),
then for all $\tilde{\phi}$ in a neighborhood of $\phi$ w.r.t. the norm $\|\cdot\|_1$,
the generalized equation 
$$
0\in \tilde{\phi}(z)+N(z)
$$
has a Lipschitz continuous and unique solution $\bar{z}(\tilde{\phi})$ in a neighborhood of $z_0$.
\end{lemma}

\section{Proof of section \ref{sec2:MLPAUE}}
\subsection{Proof of Proposition \ref{prop:property-of-u}}
\begin{proof}
Part (i). Since for any $t>0$, $\hat{\by{u}}(\by{f},\by{C}(\by{f},\xi),t)$ is continuously differentiable w.r.t. $\by{C}$, which implies that $\hat{\by{u}}(\cdot,\by{C}(\cdot,\xi),t)$ is continuously differentiable w.r.t. $\by{f}$ via Assumption \ref{asm:T-value} and $\by{C}(\by{f},\xi)=\Delta\by{T}(\Delta\by{f},\xi)$.

  Part (ii). By Assumption \ref{asm:T-value}(b), we have 
    $$
    \sup_{\by{f}\in\mathbb{F}}\|\by{C}(\by{f},\xi)\|=\sup_{\by{v}\in\mathbb{V}_t}\|\Delta^\top\by{T}(\by{v},\xi)\|\leq  \eta_{\Delta^\top}\sup_{\by{v}\in\mathbb{V}}\|\by{T}(\by{v},\xi)\|\leq  \eta_{\Delta^\top}\Phi_1(\xi).
    $$
    Moreover, 
    \bgeq
    \sup_{\by{f}\in\mathbb{F}_t}\|\hat{\by{u}}(\by{f},\by{C}(\by{f},\xi),t)\|
    &=&\sup_{\by{f}\in\mathbb{F}_t}\|\theta_0\by{d}+\theta_1\by{C}(\by{f},\xi)+\theta_2\hat{\by{h}}(\by{f},\by{C}(\by{f},\xi),t)\|\\ 
    &\leq&\theta_0\|\by{d}\|+\sup_{\by{f}\in\mathbb{F}_t}\|\theta_1\by{C}(\by{f},\xi)\|+\sup_{\by{f}\in\mathbb{F}_t}\theta_2\|\hat{\by{h}}(\by{f},\by{C}(\by{f},\xi),t)\|\\
    &\leq&\theta_0\|\by{d}\|+\theta_1 \eta_{\Delta^\top}\Phi_1(\xi)+\theta_2\sup_{\by{f}\in\mathbb{F}_t}\|\by{C}(\by{f},\xi)-\by{\tau}+\by{t}\|,\\
    &\leq&\eta_{\Delta^\top}(\theta_1+\theta_2)\Phi_1(\xi)+\theta_0\|\by{d}\|+\theta_2\|\by{\tau}\|+t,
    \edeq
    Let $\tilde{\Phi}_1(\xi,t)=\eta_{\Delta^\top}(\theta_1+\theta_2)\Phi_1(\xi)+\theta_0\|\by{d}\|+\theta_2\|\by{\tau}\|+t$. Since $\mathbb{E}_P[\Phi_1(\xi)]<\infty$, then we can easily obtain $\mathbb{E}_P[\tilde{\Psi}_1(\xi,t)]<\infty$ for any fixed $t>0$.

    Part (iii).
    Since
    \bgeq
    &&\|\hat{\by{u}}(\by{f},\by{C}(\by{f},\xi),t)-\hat{\by{u}}(\by{f}',\by{C}(\by{f}',\xi),t)\|\\
    &=&\|\theta_1(\by{C}(\by{f},\xi)-\by{C}(\by{f}',\xi))+\hat{\by{h}}(\by{f},\by{C}(\by{f},\xi),t)-\hat{\by{h}}(\by{f}',\by{C}(\by{f}',\xi),t)\|\\
    &\leq&(\theta_1+\theta_2)\|\Delta^\top\|\|\by{T}(\by{v},\xi)-\by{T}(\by{v}',\xi)\|\\
    &\leq&(\theta_1+\theta_2 \eta_{\Delta^\top}\Phi_2(\xi)\|\Delta\|\|\by{f}-\by{f}'\|\\
    &\leq&\eta_{\Delta^\top}\eta_{\Delta}(\theta_1+\theta_2)\Phi_2(\xi)\|\by{f}-\by{f}'\|, \quad\forall\by{f},\by{f}'\in\mathbb{F}_t,t>0.
    \edeq 
    The first inequality holds is due to the fact that each component of $\hat{\by{h}}(\by{f},\by{C}(\by{f},\xi),t)$ is Lipschitz continuous in $\by{C}(\by{f},\xi)$ with modulus $\theta_2$ for any $t>0$ and $\by{C}(\by{f},\xi)=\Delta^\top\by{T}(\by{v},\xi)$.
    The proof is completed. \hfill $\Box$
\end{proof}
\subsection{Proof of Theorem \ref{thm:smth-approx}}
\begin{proof} 
Observe firs that by Proposition \ref{coy:exist-solu-MSVIP}, 
$\by{\mathcal{F}}_t(P)\neq \emptyset$. 


Part (i). Let $\by{f}_t(P)\in\by{\mathcal{F}}_t(P)$. Since $\by{\mathcal{F}}_t(P)$ is bounded, we show that any 
cluster point of the sequence 
is a solution of \eqref{eq:SVI-LAPUE}.
By taking a subsequence if necessary, we may assume for the simplicity of notation that 
 \bgeqn
 \lim_{t\downarrow 0}\by{f}_t(P)=\by{f}(P).
 \edeqn
    Since $\hat{\by{u}}(\by{f}_t(P),\by{C}(\by{f}_t(P),\xi),t)
    $
    is a approximation of $\by{u}(\by{f}(P),\by{C}(\by{f}(P),\xi))$,
    we have by the continuity of $\by{C}$ w.r.t. $\by{f}$,   
    \bgeqn
    \lim_{{t\downarrow 0}}\hat{\by{u}}(\by{f}_t(P),\by{C}(\by{f}_t(P),\xi),t)=\by{u}(\by{f}(P),\by{C}(\by{f}(P),\xi)).
    \edeqn
    Moreover, by Proposition \ref{prop:property-of-u},  
    $\hat{\by{u}}(\by{f}_t(P),\by{C}(\by{f}_t(P),\xi),t)$ is integrally bounded.
    Consequently, we have by Lebesgue's dominated  convergence theorem 
   \bgeq
   \lim_{{t\downarrow 0}}\hat{\phi}_t(\by{f}_t(P),P)&=&\lim_{{t\downarrow 0}}\mathbb{E}_P\left[\hat{\by{u}}(\by{f}_t(P),\by{C}(\by{f}_t(P),\xi),t)\right] \\
   &=&
    \mathbb{E}_P\left[\lim_{{t\downarrow 0}}\hat{\by{u}}(\by{f}_t(P),\by{C}(\by{f}_t(P),\xi),t)\right]\\
    &=&\mathbb{E}_P[\by{u}(\by{f}(P),\by{C}(\by{f}(P),\xi))]=\phi(\by{f}(P),P).
    \edeq
    The conclusion follows immediately from this and the fact that the normal cone is upper semi-continuous. 
    
    Part (ii). Let $\by{f}(P)\in\by{\mathcal{F}}(P)$. It is easy to prove that there exist a sufficiently small constant $t_0>0$ and closed neighborhood $\mathcal{V}_{\by{f}(P)}$ of $\by{f}(P)$ such that $\{\by{f}_t(P)\in\by{\mathcal{F}}_t(P):
    \lim_{t\downarrow 0} \by{f}_t(P) = \by{f}(P), t\in(0,t_0]\}\subseteq\mathcal{V}_{\by{f}(P)}$ by virtue of (i). Let $\mathcal{B}$ be a unit ball in $\R^N$ and $s$ be a constant such that 
    $$
    s>\max_{\by{f}\in\mathcal{V}_{\by{f}(P)}\cap D,t\in(0,t_0]}\{\|\hat{\phi}_t(\by{f},P)\|,\|\phi(\by{f},P)\|\}.
    $$
    Then for any $\by{f}\in\mathcal{V}_{\by{f}(P)}\cap D$,
    $$
    0\in\phi(\by{f},P)+\mathcal{N}_D(\by{f})\cap s\mathcal{B},
    $$
    where $\mathcal{B}$ is a unit ball in  $\R^N$.
        Likewise, for $\by{f}\in\{\by{f}_t(P):t\in(0,t_0]\}$,
    \bgeqn
    \label{eq:ft}
0\in\hat{\phi}_t(\by{f},P)+\mathcal{N}_D(\by{f})\cap s\mathcal{B}.
    \edeqn
    On the other hand, the metric regularity of 
     $\Upsilon(\by{f})$ at $\by{f}(P)$ for $0$ with regularity modulus $\alpha$ 
     implies that there exists a neighborhood $\mathcal{V}_{\by{f}(P)}$ of $\by{f}(P)$ 
     such that 
     \bgeqn
     d(\by{f}_t(P),\by{\mathcal{F}}(P))\leq\alpha d(0,\Upsilon(\by{f}_t(P)))
     \edeqn
     for all $\by{f}_t(P)\in\by{\mathcal{F}}_t(P)\cap\mathcal{V}_{\by{f}(P)}$ and $t\in(0,t_0]$. 
     Since
     $$
     \Upsilon(\by{f})=\phi(\by{f},P)+\mathcal{N}_D(\by{f})\supset\phi(\by{f},P)+\mathcal{N}_D(\by{f})\cap s\mathcal{B},
     $$
     then
     $$
     d(0,\Upsilon(\by{f}))\leq d (0,\phi(\by{f},P)+\mathcal{N}_D(\by{f})\cap s\mathcal{B}).
     $$
     Moreover, since $\by{f}_t(P)$ satisfies \eqref{eq:ft}, then
     \bgeqn
      d(\by{f}_t(P),\by{\mathcal{F}}(P))&\leq&\alpha d (0,\phi(\by{f}_t(P),P)+\mathcal{N}_D(\by{f}_t(P))\cap s\mathcal{B})\nonumber\\
     &\leq&\alpha \mathbb{D}(\hat{\phi}_t(\by{f}_t(P),P)+\mathcal{N}_D(\by{f}_t(P))\cap s\mathcal{B},\phi(\by{f}_t(P),P)+\mathcal{N}_D(\by{f}_t(P))\cap s\mathcal{B})\nonumber\\
     &\leq&\alpha\|\hat{\phi}_t(\by{f}_t(P),P)-\phi(\by{f}_t(P),P)\|\nonumber\\
     &\leq&\alpha\sup_{\by{f}\in\mathcal{V}_{\by{f}(P)}}\|\hat{\phi}_t(\by{f},P)-\phi(\by{f},P)\| \nonumber
     \edeqn
     for all$\by{f}_t(P)\in\by{\mathcal{F}}_t(P)\cap\mathcal{V}_{\by{f}(P)}$ and $t\in(0,t_0]$. Therefore, we obtain \eqref{thm3.1-equ2}. Inequality \eqref{thm3.1-equ3} follows straight-forwardly from \eqref{thm3.1-equ2} and strong mertic regularity.
    \hfill $\Box$
\end{proof}

\section{Proof of Section \ref{sec3:SR-MLAPUE}}
\subsection{Proof of Proposition \ref{prop:IF}}
\begin{proposition}
\label{prop:IF}
    For fixed $t$, let $\by{f}_t^*$ be a 
 solution of \eqref{eq:SVI-LAPUE-smooting} and 
     $\hat{\by{u}}^*_t=-\mathbb{E}_P[\hat{\by{u}}(\by{f}_t^*,\by{C}
    (\by{f}_t^*,\xi),t)]$.  
    Then the generalized influence function $\inmat{GIF}(\tilde{\xi},\by{f}_t(\cdot),P)$ exists and has 
    the following formula:
    \begin{equation}
    \label{expression:IF}
    \begin{split}
    \inmat{GIF}(\tilde{\xi};\by{\mathcal{F}}_t(\cdot),P,\by{f}_t^*)=\{\by{f}\in\R^N|&-\hat{\by{u}}(\by{f}_t^*,\by{C}(\by{f}_t^*,\tilde{\xi}),t)+\hat{\phi}_t(\by{f}_t^*,P)\\
    &-\mathbb{E}_P[ \nabla_{\by{f}}\hat{\by{u}}(\by{f}_t^*,\by{C}(\by{f}_t^*,\xi),t)]\by{f}
    \in\mathcal{N}_{\by{f}_t^*,\hat{\by{u}}^*_t}(\by{f})\}.
    \end{split}
    \end{equation}
\end{proposition}
\begin{proof}
    By \cite[Definition 3.1]{LevyR94}, the directional derivative of $\Psi_{\tilde{\xi},P,t}(\cdot,\cdot)$ at $(\by{f}_t^*,0)$ 
     in the direction $(\by{f},1)$
    can be written as
    \bgeq
    &&\Psi_{\tilde{\xi},P,t;\by{f}_t^*,0}'(\by{f},1)\\
    &=&\lim_{\epsilon\downarrow 0,(\by{f}',s)\to(\by{f},1)}\frac{\Psi_{\tilde{\xi},P,t}(\by{f}_t^*+\epsilon\by{f}',\epsilon s)-\Psi_{\tilde{\xi},P,t}(\by{f}_t^*,0)}{\epsilon}\\
    &=&\lim_{\epsilon\downarrow 0,(\by{f}',s)\to(\by{f},1)}\left\{\frac{\hat{\phi}_t(\by{f}_t^*+\epsilon\by{f}',P)-\hat{\phi}_t(\by{f}_t^*,P)}{\epsilon}\right.\\
    &&\quad\quad\quad\quad\left.+\frac{\epsilon s\left(\hat{\by{u}}(\by{f}_t^*+\epsilon\by{f}',\by{C}(\by{f}_t^*+\epsilon\by{f}',\tilde{\xi}),t)-\hat{\phi}_t(\by{f}_t^*+\epsilon\by{f}',P)\right)}{\epsilon}\right\}\\
    &=&\hat{\by{u}}(\by{f}_t^*,\by{C}(\by{f}_t^*,\tilde{\xi}),t)-\hat{\phi}_t(\by{f}_t^*,P)+\nabla_{\by{f}}\hat{\phi}_t(\by{f}_t^*,P)\by{f}.
    \edeq
    Thus, by the continuous differentiability of $\hat{\by{u}}(\cdot,\by{C}(\cdot,\xi),t)$ and \cite[Theorem 4.1]{LevyR94}, we conclude \eqref{expression:IF}.
    \hfill $\Box$
\end{proof}
\subsection{Proof of Theorem \ref{thm:IF}}
\begin{proof}
   Part (i). By the condition, 
    \bgeq
\by{\tilde{\mathcal{F}}}'_{\tilde{\xi},P,t;0,\by{f}_t^*}(0)=\left\{\by{f}\in\R^N| 0\in\mathbb{E}_P\left[\nabla_{\by{f}}\hat{\by{u}}(\by{f}_t^*,\by{C}(\by{f}_t^*,\xi),t)\right]\by{f}+\mathcal{N}_{\by{f}_t^*,\hat{\by{u}}^*}'(\by{f})\right\}=\{0\},
    \edeq
    then by \cite[Theorem 4.1]{Rockafellar89}, there exist constants $\alpha_0$, $\rho$ and $\beta\in(0,1)$ depending on $\tilde{\xi}$ such that
    \bgeq
    \tilde{\by{\mathcal{F}}}_{\tilde{\xi},P,t}(\varepsilon)\cap(\by{f}_t^*+\alpha_0\mathcal{B})\subset\by{f}_t^*+\rho \varepsilon\mathcal{B}
    \edeq
    for all $\varepsilon\in(0,\beta)$, where $\mathcal{B}$ denotes a unit ball in $\R^N$. In addition,
    \bgeq
    \frac{\tilde{\by{\mathcal{F}}}_{\tilde{\xi},P,t}(\epsilon s)-\by{f}_t^*}{\epsilon}\cap\frac{\alpha_0}{\epsilon}\mathcal{B}\subset\rho s\mathcal{B}
    \edeq
    for all $\epsilon s\in(0,\beta)$. 
    Therefor 
    \bgeq
   \inmat{GIF}(\tilde{\xi};\by{\mathcal{F}}_t(\cdot),P,\by{f}_t^*)=\limsup_{\epsilon\downarrow 0,s\to 1}\frac{\tilde{\by{\mathcal{F}}}_{\tilde{\xi},P,t}(\epsilon s)-\by{f}_t^*}{\epsilon}\in\rho\mathcal{B},
    \edeq
    which implies
    \bgeq
    \|\inmat{GIF}(\tilde{\xi};\by{\mathcal{F}}_t(\cdot),P,\by{f}_t^*)\|\leq\rho.
    \edeq
    Part (ii). Similar to the proof of Theorem 2.2 (ii), under the strong regularity condition, we can obtain that 
    \bgeqn
    \label{eq:solu}
    \|\by{f}_t(Q),\by{f}_t^*\|\leq \gamma \|\hat{\phi}_t(\by{f}_t(Q),P)-\hat{\phi}_t(\by{f}_t(Q),Q)\|
    \edeqn
    for all $\by{f}_t(Q)\in\by{\mathcal{F}}_t(Q)\cap\mathcal{V}_{\by{f}_t^*}$ and $Q\in\mathscr{P}(\Xi)$.
    Let $Q=(1-\varepsilon)P+\varepsilon\delta_{\tilde{\xi}}$ and $\epsilon_0>0$ such that $\by{\mathcal{F}}_t(Q)\cap\mathcal{V}_{\by{f}_t^*,\epsilon_0}\neq\emptyset$. We have  
    \bgeqn
    \label{eq:solu-cont}
    \by{\mathcal{F}}_t(Q)\cap\mathcal{V}_{\by{f}_t^*,\epsilon_0}
    &\subset& \by{f}_t^*+\gamma \|\hat{\phi}_t(\by{f}_t(Q),P)-\hat{\phi}_t(\by{f}_t(Q),Q)\|\mathcal{B}\nonumber\\
    &\subset& \by{f}_t^*+\gamma \sup_{\by{f}\in\mathbb{F}_t}\|\hat{\phi}_t(\by{f},P)-\hat{\phi}_t(\by{f},Q)\|\mathcal{B}\nonumber\\
    &=& \by{f}_t^*+\gamma\varepsilon \sup_{\by{f}\in\mathbb{F}_t}\|\hat{\phi}_t(\by{f},P)-\hat{\by{u}}(\by{f},\by{C}(\by{f},\tilde{\xi}),t)\|\mathcal{B}
    \edeqn
    via \eqref{eq:solu}, where $\mathcal{V}_{\by{f}_t^*,\epsilon_0}$ denotes a $\epsilon_0$-neighborhood of $\by{f}_t^*$ and $\mathbb{F}_t$ is defined as in Assumption \ref{asm:compact-MSVIP}. Note that $\by{\mathcal{F}}_t(Q)=\tilde{\by{\mathcal{F}}}_{\tilde{\xi},P,t}(\varepsilon)$, then \eqref{eq:solu-cont} implies that 
    \bgeq
    \frac{\tilde{\by{\mathcal{F}}}_{\tilde{\xi},P,t}(\epsilon s)-\by{f}_t^*}{\epsilon}\cap\frac{\epsilon_0}{\epsilon}\mathcal{B}\subset\gamma s\sup_{\by{f}\in\mathbb{F}_t}\|\hat{\phi}_t(\by{f},P)-\hat{\by{u}}(\by{f},\by{C}(\by{f},\tilde{\xi}),t)\|\mathcal{B}
    \edeq
    for $\gamma s\in(0,1)$. Thus 
    \bgeq
   \inmat{GIF}(\tilde{\xi};\by{\mathcal{F}}_t(\cdot),P,\by{f}_t^*)&=&\limsup_{\epsilon\downarrow 0,s\to 1}\frac{\tilde{\by{\mathcal{F}}}_{\tilde{\xi},P,t}(\epsilon s)-\by{f}_t^*}{\epsilon}\in\rho\mathcal{B},\\
   &\subset&
   \sup_{\by{f}\in\mathbb{F}_t}\|\hat{\phi}_t(\by{f},P)-\hat{\by{u}}(\by{f},\by{C}(\by{f},\tilde{\xi}),t)\|\mathcal{B},
    \edeq
    which implies 
    \bgeq
    \|\inmat{GIF}(\tilde{\xi};\by{\mathcal{F}}_t(\cdot),P,\by{f}_t^*)\|\leq\gamma \sup_{\by{f}\in\mathbb{F}_t}\|\hat{\phi}_t(\by{f},P)-\hat{\by{u}}(\by{f},\by{C}(\by{f},\tilde{\xi}),t)\|.
    \edeq
    for all $\tilde{\xi}\in\R^k$. By taking the supremum w.r.t. $\tilde{\xi}\in\R^k$, we derive \eqref{GIF:bounded}. The proof is completed.
    \hfill $\Box$
\end{proof}

\subsection{Proof of Proposition \ref{prop:u-LipC}}
\begin{proof}
Part (i). 
Since $\hat{\by{h}}$ is globally Lipschitz continuous in the second argument uniformly w.r.t. the first and the third arguments with modulus $\theta_2$, 
then
 \bgeq
&&\|\hat{\by{u}}(\by{f},\by{C}(\by{f},\xi),t)-\hat{\by{u}}(\by{f},\by{C}(\by{f},\xi'),t)\|\\
&=&\|\by{g}(\by{f},\by{C}(\by{f},\xi))-\by{g}(\by{f},\by{C}(\by{f},\xi'))+\hat{\by{h}}(\by{f},\by{C}(\by{f},\xi),t)-\hat{\by{h}}(\by{f},\by{C}(\by{f},\xi'),t)\|\\
&\leq& \|\by{g}(\by{f},\by{C}(\by{f},\xi))-\by{g}(\by{f},\by{C}(\by{f},\xi'))\|+\|\hat{\by{h}}(\by{f},\by{C}(\by{f},\xi),t)-\hat{\by{h}}(\by{f},\by{C}(\by{f},\xi'),t)\|\\
&
\leq&(\theta_1+\theta_2)\|\by{C}(\by{f},\xi))-\by{C}(\by{f},\xi'))\|\\
&\leq&L\eta_{\Delta^\top}(\theta_1+\theta_2)  \|\xi-\xi'\|.
\edeq 
Part (ii). Since, 
\bgeqn
&&\|\nabla_{\by{f}}\hat{\by{u}}(\by{f},\by{C}(\by{f},\xi),t)-\nabla_{\by{f}}\hat{\by{u}}(\by{f},\by{C}(\by{f},\xi'),t)\|\nonumber\\
&&=\|\nabla_{\by{f}}\by{g}(\by{f},\by{C}(\by{f},\xi))-\nabla_{\by{f}}\by{g}(\by{f},\by{C}(\by{f},\xi'))+\nabla_{\by{f}}\hat{\by{h}}(\by{f},\by{C}(\by{f},\xi),t)-\nabla_{\by{f}}\hat{\by{h}}(\by{f},\by{C}(\by{f},\xi'),t)\|\nonumber\\
&&\leq \|\nabla_{\by{f}}\by{g}(\by{f},\by{C}(\by{f},\xi))-\nabla_{\by{f}}\by{g}(\by{f},\by{C}(\by{f},\xi'))\|+\|\nabla_{\by{f}}\hat{\by{h}}(\by{f},\by{C}(\by{f},\xi),t)-\nabla_{\by{f}}\hat{\by{h}}(\by{f},\by{C}(\by{f},\xi'),t)\|\nonumber\\
&&\leq\theta_{1}\|\Delta^\top\|\|\nabla_{\by{v}}\by{T}(\Delta\by{f},\xi)-\nabla_{\by{v}}\by{T}(\Delta\by{f},\xi')\|\|\Delta\|+\|\nabla_{\by{f}}\hat{\by{h}}(\by{f},\by{C}(\by{f},\xi),t)-\nabla_{\by{f}}\hat{\by{h}}(\by{f},\by{C}(\by{f},\xi'),t)\|\nonumber
\edeqn
Moreover, since for the case that $C_r(\by{f},\xi)$ and $C_r(\by{f},\xi')$ are in the same segment interval, we have  
\bgeq
&&\|[\nabla_{\by{f}}\hat{\by{h}}(\by{f},\by{C}(\by{f},\xi),t)]_r-[\nabla_{\by{f}}\hat{\by{h}}(\by{f},\by{C}(\by{f},\xi'),t)]_r\|\\
&&=\left\{
\begin{array}{ll}
\theta_2\|\left[\nabla_{\by{f}}\by{C}(\by{f},\xi)-\nabla_{\by{f}}\by{C}(\by{f},\xi')\right]_r\|, & \inmat{for} \;\; C_r,C_r'>t+\tau_k\\
    \frac{\theta_2}{2t}\|\left(C_r(\by{f},\xi)-\tau_k+t\right)[\nabla_{\by{f}}\by{C}(\by{f},\xi)]_r&\\
    -\left(C_r(\by{f},\xi')-\tau_k+t\right)\left[\nabla_{\by{f}}\by{C}(\by{f},\xi')\right]_r\|& \inmat{for}\; \ -t+\tau_k\leq C_r,C_r'\leq t+\tau_k,\\
    0\in\R^{N\times N}, & \inmat{for}\; \ C_r,C_r' < -t+\tau_k,
\end{array}
\right.\\
&&=\left\{
\begin{array}{ll}
\theta_2\|\left[\nabla_{\by{f}}\by{C}(\by{f},\xi)-\nabla_{\by{f}}\by{C}(\by{f},\xi')\right]_r\|, & \inmat{for} \;\; C_r,C_r'>t+\tau_k\\
    \frac{\theta_2}{2t}\|\left(C_r(\by{f},\xi)-\tau_k+t\right)\left[\nabla_{\by{f}}\by{C}(\by{f},\xi)-\nabla_{\by{f}}\by{C}(\by{f},\xi')\right]_r&\\
    +\left(C_r(\by{f},\xi)-C_r(\by{f},\xi')\right)\left[\nabla_{\by{f}}\by{C}(\by{f},\xi')\right]_r\|& \inmat{for}\; \ -t+\tau_k\leq C_r,C_r'\leq t+\tau_k,\\
    0\in\R^{N\times N}, & \inmat{for}\; \ C_r,C_r' < -t+\tau_k,
\end{array}
\right.
\\
&&\leq\left\{
\begin{array}{ll}
\theta_2\|\left[\nabla_{\by{f}}\by{C}(\by{f},\xi)-\nabla_{\by{f}}\by{C}(\by{f},\xi')\right]_r\|, & \inmat{for} \;\; C_r,C_r'>t+\tau_k\\
    \frac{\theta_2}{2t}|C_r(\by{f},\xi)-\tau_k+t|\|\left[\nabla_{\by{f}}\by{C}(\by{f},\xi)-\nabla_{\by{f}}\by{C}(\by{f},\xi')\right]_r\|&\\
    +|C_r(\by{f},\xi)-c_r(\by{f},\xi'|\|\left[\nabla_{\by{f}}\by{C}(\by{f},\xi')\right]_r\|& \inmat{for}\; \ -t+\tau_k\leq C_r,C_r'\leq t+\tau_k,\\
    0\in\R^{N\times N}, & \inmat{for}\; \ C_r,C_r' < -t+\tau_k,
\end{array}
\right.
\\
&&\leq\left\{
\begin{array}{ll}
\theta_2\|\left[\nabla_{\by{f}}\by{C}(\by{f},\xi)-\nabla_{\by{f}}\by{C}(\by{f},\xi')\right]_r\|, & \inmat{for} \;\; C_r,C_r'>t+\tau_k\\
    \theta_2\|\left[\nabla_{\by{f}}\by{C}(\by{f},\xi)-\nabla_{\by{f}}\by{C}(\by{f},\xi')\right]_r\|&\\
    +\frac{\theta_2}{2t}\|\left[\nabla_{\by{f}}\by{C}(\by{f},\xi)\right]_r\|\|\left[\left(\by{C}(\by{f},\xi)-\by{C}(\by{f},\xi')\right)\right]_r\| &  \inmat{for}\; \ -t+\tau_k\leq C_r,C_r'\leq t+\tau_k,\\
    0\in\R^{N\times N}, & \inmat{for}\; \ C_r,C_r' < -t+\tau_k,
\end{array}
\right.
\edeq
where $C_r=C_r(\by{f},\xi)$ and $C_r'=C_r(\by{f},\xi')$. Then, it is not hard to obtain that 
\bgeq
\|\nabla_{\by{f}}\hat{\by{h}}(\by{f},\by{C}(\by{f},\xi),t)-\nabla_{\by{f}}\hat{\by{h}}(\by{f},\by{C}(\by{f},\xi'),t)\|\leq L\theta_2\eta_{\Delta^\top}\left(\eta_{\Delta}+\frac{C_0'}{2t}\right)\|\xi-\xi'\|,
\edeq
where $C_0'=\sup_{\by{f}\in D,\xi\in\Xi}\|\nabla_{\by{f}}\by{C}(\by{f},\xi)\|$. Since $\by{C}(\by{f},\xi)$ is continuously differentiable w.r.t. $\by{f}$ and continuous w.r.t. $\xi$ and $\Xi$ and $\mathbb{F}$ are compact sets, then $C_0'<\infty$.
Therefore, it follows that
\bgeq
\|\nabla_{\by{f}}\hat{\by{u}}(\by{f},\by{C}(\by{f},\xi),t)-\nabla_{\by{f}}\hat{\by{u}}(\by{f},\by{C}(\by{f},\xi'),t)\|\leq L\eta_{\Delta^\top}\left(\theta_1\eta_{\Delta}+\theta_2\eta_{\Delta}+\theta_2\hat{C}(t)\right)\|\xi-\xi'\|
\edeq
for $t>0$, where $\hat{C}(t)=\frac{C_0'}{2t}$. 
\hfill $\Box$
\end{proof}

\subsection{Proof of Theorem \ref{thm:Lips-solu}}
\begin{proof}
Part (i). It follows 
by Lemma \ref{lem:LIp-solu}, there exist a unique vector-valued function 
   $\by{f}_t(\cdot)$ and positive constants $\kappa$ and $\delta$ such that
\bgeqn
    \|\by{f}_t(Q_1)-\by{f}_t(Q_2)\|\leq\kappa\|\hat{\phi}_t(\cdot,Q_1)-\hat{\phi}_t(\cdot,Q_2)\|_{1,\mathcal{O}_t}
\edeqn
for any $Q_1,Q_2\in{\cal M}_k^1$ satisfying $\dd_K(Q_1,P)\leq\delta$ and $\dd_K(Q_2,P)\leq\delta$.
By the definition of $\|\cdot\|_{1,\mathcal{O}_t}$,
\begin{equation}
\begin{split}
&\|\hat{\phi}_t(\cdot,Q_1)-\hat{\phi}_t(\cdot,Q_2)\|_{1,{\cal  O}_t}\\
&= \sup_{{\bm f}\in {\cal  O}_t} \|\hat{\phi}_t({\bm f},Q_1)-\hat{\phi}_t(\by{f},Q_2)\|+ \sup_{{\bm f}\in {\cal  O}_t}\|\nabla_{\bm f}\hat{\phi}_t({\bm f},Q_1)-\nabla_{\bm f}\hat{\phi}_t({\bm f},Q_2)\| \nonumber \\
&= \sup_{{\bm f}\in {\cal  O}_t,\|\by{\eta}\|\leq1} \left|\by{\eta}^\top\hat{\phi}_t({\bm f},Q_1)-\by{\eta}^\top\hat{\phi}_t({\bm f},Q_2)\right|\nonumber\\
& \ \ \ + \sup_{{\bm f}\in {\cal  O}(t),\|\by{\eta}\|\leq1,\|\by{\nu}\|\leq1}\left|\by{\eta}^\top\nabla_{\bm f}\hat{\phi}_t({\bm f},Q_1)\by{\nu}|-\by{\eta}^\top\nabla_{\bm f}\hat{\phi}_t({\bm f},Q_2)\by{\nu}\right| \nonumber \\
&\leq L\eta_{\Delta^\top}(\theta_1+\theta_2) \dd_K(Q_1,Q_2)+L\eta_{\Delta^\top}\left(\theta_1\eta_{\Delta}+\theta_2\eta_{\Delta}+\theta_2\hat{C}(t)\right)\dd_K(Q_1,Q_2)\nonumber\\
&= L\eta_{\Delta^\top}\left(\theta_1+\theta_2+\theta_1\eta_{\Delta}+\theta_2\eta_{\Delta}+\theta_2\hat{C}(t)\right) \dd_K(Q_1,Q_2),
\end{split}
\end{equation}
where $\hat{\phi}_t({\bm f},Q)=\mathbb{E}_Q[\hat{\by{u}}(\by{f},\by{C}(\by{f},\xi),t)]$.
The first inequality holds due to $\frac{\by{\eta}^\top\hat{\by{u}}(\by{f},\by{C}(\by{f},\xi),t)}{ L\eta_{\Delta^\top}(\theta_1+\theta_2)  }$ and $\frac{\by{\eta}^\top\nabla_{\by{f}}\hat{\by{u}}(\by{f},\by{C}(\by{f},\xi),t)\by{\nu}}{L\eta_{\Delta^\top}\left(\theta_1\eta_{\Delta}+\theta_2\eta_{\Delta}+\theta_2\hat{C}(t)\right)}\in\mathcal{F}_1(\R^k)$ via Proposition \ref{prop:u-LipC}, where $\mathcal{F}_1(\R^k)$ is defined as in \eqref{dist-Kan}. 

Part (ii). Since $\mathcal{M}_k^\gamma$ is relatively compact under $\|\cdot\|$-weak topology when $\gamma>1$ (see \cite[Lemma 2.69]{Claus16}) and $\dd_K$ metrizes the $\|\cdot\|$-topology (see \cite[ Theorem 6.3.1]{Rachev91}), thus by Prokhorov's theorem \cite{Prokhorov56}, we can construct a $\epsilon$-net $\mathcal{Q}_J:=\{Q_1,\cdots,Q_J\}$ in the closure of $\mathcal{M}_{k,M_0}^1$ under the metric $\dd_K$ such that $\mathcal{M}_{k,M_0}^1\subset\cup_{j=1}^J\mathcal{B}(Q_j,\epsilon)$, where $\mathcal{B}(Q_j,\epsilon)$ represents a closed ball centered at $Q_j$ with radius $\epsilon$ under the metric $\dd_K$. Since $\by{f}_t(\cdot)$ is continuous over $\mathcal{M}_{k,M_0}^\gamma$ and the strong regularity condition holds at $\by{f}_t(Q)$ for $Q\in\mathcal{M}_{k,M_0}^\gamma$ by the conditions in (iii), then we can set $\epsilon$ sufficiently small such that $\by{f}_t(Q)$ for $Q\in\mathcal{B}(Q_j,\epsilon)$, is the unique solution to \eqref{eq:GE-purt} in a neighborhood of $\by{f}_t(Q_j)$ containing $\{\by{f}_t(Q):Q\in\mathcal{B}(Q_j,\epsilon)\}$ for $j=1,\cdots,J$. Following a similar argument to Part (ii), we can obtain that there exists a constant $\kappa_j>0$ such that
\bgeqn
\|\by{f}_t(Q')-\by{f}_t(Q'')\|\leq L\eta_{\Delta^\top}\kappa_j\left(\theta_1+\theta_2+\theta_1\eta_{\Delta}+\theta_2\eta_{\Delta}+\theta_2\hat{C}(t)\right)\dd_K(Q',Q'')
\edeqn
for any $Q',Q''\in\mathcal{B}(Q_j,\epsilon)$.

For any $Q',Q''\in\mathcal{M}_{k,M_0}^\gamma$ and $\beta\in(0,1)$, define $Q(\beta)=(1-\beta)Q'+\beta Q''$. Let $\hat{Q}_1\in\mathcal{Q}_J$ be such that $Q'\in\mathcal{B}(\hat{Q}_1,\epsilon)$ and $\beta_1$ be the smallest value in $(0,1)$ such that $Q(\beta_1)$ lies at the boundary of $\mathcal{B}(\hat{Q}_1,\epsilon)$ and in the next ball $\mathcal{B}(\hat{Q}_2,\epsilon)$, where $\hat{Q}_2\in\mathcal{Q}_J$. Next, we let $\beta_2$ be the smallest value in $[\beta_1,1)$  such that $Q(\beta_2)$ lies at the boundary of $\mathcal{B}(\hat{Q}_2,\epsilon)$ and in the next ball $\mathcal{B}(\hat{Q}_3,\epsilon)$, where $\hat{Q}_3\in\mathcal{Q}_J$. Continuing the process, we let $\beta_{J-1}$ be the smallest value in $[\beta_{J-2},1)$  such that $Q(\beta_{J-1})$ lies at the boundary of $\mathcal{B}(\hat{Q}_{J-1},\epsilon)$ and in the next ball $\mathcal{B}(\hat{Q}_J,\epsilon)$, where $\hat{Q}_J\in\mathcal{Q}_J$ and $Q''\in\mathcal{B}(\hat{Q}_J,\epsilon)$.  Consequently, we have 
\begin{equation*}
\begin{split}
\|\by{f}_t(Q')-\by{f}_t(Q'')\|
&\leq\|\by{f}_t(Q')-\by{f}_t\left(Q(\beta_1)\right)\|+\sum_{j=1}^{J-2}\|\by{f}_t\left(Q(\beta_j)\right)-\by{f}_t\left(Q(\beta_{j+1})\right)\|\\
&\quad +\|\by{f}_t\left(Q(\beta_{J-1})\right)-\by{f}_t(Q'')\|\\
&\leq L\eta_{\Delta^\top}\left(\theta_1+\theta_2+\theta_1\eta_{\Delta}+\theta_2\eta_{\Delta}+\theta_2\hat{C}(t)\right)\Big[\kappa_{1}\dd_K\left(Q',Q(\beta_1)\right)\\
&\quad +\sum_{j=1}^{J-2}\kappa_{j+1}\dd_K\left(Q(\beta_j),Q(\beta_{j+1})\right) + \kappa_J\dd_K\left(Q(\beta_J),Q''\right)\Big]\\
&\leq\tilde{C}L\eta_{\Delta^\top}\left(\theta_1+\theta_2+\theta_1\eta_{\Delta}+\theta_2\eta_{\Delta}+\theta_2\hat{C}(t)\right)\Big[\beta_1\dd_K\left(Q',Q''\right)+\sum_{j=1}^{J-2}(\beta_{j+1}-\beta_j)\\
&\quad \dd_K\left(Q',Q''\right)+ (1-\beta_j)\dd_K\left(Q',Q''\right)\Big]\\
&=\tilde{C}L\eta_{\Delta^\top}\left(\theta_1+\theta_2+\theta_1\eta_{\Delta}+\theta_2\eta_{\Delta}+\theta_2\hat{C}(t)\right)\dd_K(Q',Q''),
\end{split}
\end{equation*}
where $\tilde{C}=\max_{1\leq j\leq J} \kappa_j$.
The proof is complete.
\hfill $\Box$
\end{proof}

\subsection{Proof of Theorem \ref{thm:SR-LAPUE}}
\begin{proof}
  By definition
    \bgeq
\dd_K(Q_M,P_M)=\sup_{g\in\mathcal{F}_1(\R^k)}\left|\frac{1}{M}\sum_{i=1}^M g(\tilde{\xi}^i)-\frac{1}{M}\sum_{i=1}^M g(\xi^i)\right|\leq\frac{1}{M}\sum_{i=1}^M\|\tilde{\xi}^i-\xi^i\|,
    \edeq
    where $\mathcal{F}_1(\R^k)$ is defined as in \eqref{dist-Kan} and the last inequality is due to the fact that $g$ is Lipschitz continuous with modulus being $1$.
    By rewriting $\hat{\by{f}}_{M,t}(\tilde{\xi}^1,\cdots,\tilde{\xi}^M)$ and $\hat{\by{f}}_{M,t}(\xi^1,\cdots,\xi^M)$ respectively for $\by{f}_t(Q_M)$ and $\by{f}_t(P_M)$, we have 
    \begin{equation*}
        \begin{split}
    \|\hat{\by{f}}_{M,t}(\tilde{\xi}^1,\cdots,\tilde{\xi}^M)-\hat{\by{f}}_{M,t}(\xi^1,\cdots,\xi^M)\|&=\|\by{f}_t(Q_M)-\by{f}_t(P_M)\|\\
    &\leq L\tilde{C}\eta_{\Delta^\top}\left(\theta_1+\theta_2+\theta_1\eta_{\Delta}+\theta_2\eta_{\Delta}+\theta_2\hat{C}(t)\right)\dd_K(Q_M,P_M)\\
    &\leq\frac{L\tilde{C}\eta_{\Delta^\top}\left(\theta_1+\theta_2+\theta_1\eta_{\Delta}+\theta_2\eta_{\Delta}+\theta_2\hat{C}(t)\right)}{M}\sum_{i=1}^M\|\tilde{\xi}^i-\xi^i\|,
        \end{split}
    \end{equation*}
    where the first inequality holds due to Theorem \ref{thm:Lips-solu} (iii) and the fact that $\mathscr{P}(\Xi)$ is a weakly compact set according to the compactness of $\Xi$. Moreover,
   by \cite[Lemma 1]{GuoXu21a},
    \bgeq
    &&\dd_K(P^{\otimes M}\circ\hat{\by{f}}_{M,t}^{-1},Q^{\otimes M}\circ\hat{\by{f}}_{M,t}^{-1})\\
    &&=\sup_{g\in\mathcal{F}_1(\R^k)}\left|\int_{\R}g(s)P^{\otimes M}\circ\hat{\by{f}}_{M,t}^{-1}(ds)-\int_{\R}g(s)Q^{\otimes M}\circ\hat{\by{f}}_{M,t}^{-1}(ds)\right|\\
    &&=\sup_{g\in\mathcal{F}_1(\R^k)}\left|\int_{\R}g(\hat{\by{f}}_{M,t}(\vec{\xi}^M))P^{\otimes M}(d\vec{\xi}^M)-\int_{\R}g(\hat{\by{f}}_{M,t}(\vec{\tilde{\xi}}^M))Q^{\otimes M}(d\vec{\tilde{\xi}}^M)\right|\\
    &&=\dd_\psi(P^{\otimes M},Q^{\otimes M})\leq L\tilde{C}\eta_{\Delta^\top}\left(\theta_1+\theta_2+\theta_1\eta_{\Delta}+\theta_2\eta_{\Delta}+\theta_2\hat{C}(t)\right)\dd_{K}(P,Q)
    \edeq
    for all $M$, where $\vec{\xi}^M=(\xi^1,\cdots,\xi^M)$ and $\vec{\tilde{\xi}}^M=(\tilde{\xi}^1,\cdots,\tilde{\xi}^M)$. The proof is completed. \hfill $\Box$
\end{proof}

\end{document}